\documentclass[12pt,reqno]{amsart}
\usepackage{latexsym}
\usepackage{amssymb}
\usepackage{amscd}
\usepackage{comment}

\numberwithin{equation}{section}

\newtheorem{theorem}{Theorem}
\newtheorem{proposition}[theorem]{Proposition}
\newtheorem{lemma}[theorem]{Lemma}
\newtheorem{corollary}[theorem]{Corollary}
\newtheorem{conjecture}[theorem]{Conjecture}

\theoremstyle{remark}
\newtheorem*{remark}{Remark}
\newtheorem{remarknu}[theorem]{Remark}
\newtheorem*{note}{Note}

\def\al{\alpha}
\def\be{\beta}
\def\de{\delta}
\def\ga{\gamma}
\def\ep{\varepsilon}
\def\la{\lambda}
\def\om{\omega}
\def\De{\bigtriangleup}
\def\Na{\bigtriangledown}
\def\Z{{\mathbb Z}}
\def\V#1{\Vert #1\Vert}
\def\po#1#2{(#1)_#2}
\def\coef#1{\left\langle#1\right\rangle}

\def\iso{\operatorname{\#iso}}
\def\Iso{\operatorname{\#Iso}}
\def\End{\operatorname{end}}
\def\Estring{\operatorname{\#1str}}
\def\EEstring{\operatorname{\#2str}}
\def\Pol{\operatorname{Pol}}
\def\fl#1{\left\lfloor#1\right\rfloor}
\def\cl#1{\left\lceil#1\right\rceil}

\begin{document}
\title[Mod-$3^k$ behaviour of recursive sequences]{A 
method for determining the mod-$3^k$ behaviour of
recursive sequences}
\author[C. Krattenthaler and 
T.\,W. M\"uller]{C. Krattenthaler$^{\dagger}$ and
T. W. M\"uller$^*$} 

\address{$^{\dagger*}$Fakult\"at f\"ur Mathematik, Universit\"at Wien,
Nordbergstrasze~15, A-1090 Vienna, Austria.
WWW: {\tt http://www.mat.univie.ac.at/\lower0.5ex\hbox{\~{}}kratt}.}

\address{$^*$School of Mathematical Sciences, Queen Mary
\& Westfield College, University of London,
Mile End Road, London E1 4NS, United Kingdom.
}

\thanks{$^\dagger$Research partially supported by the Austrian
Science Foundation FWF, grants Z130-N13 and S9607-N13,
the latter in the framework of the National Research Network
``Analytic Combinatorics and Probabilistic Number
Theory"\newline\indent
$^*$Research supported by Lise Meitner Grant M1201-N13 of the Austrian
Science Foundation FWF}



\begin{abstract}
We present a method for obtaining congruences modulo powers of~$3$ for
sequences given by recurrences of finite 
depth with polynomial coefficients.
We apply this method to 
Catalan numbers, Motzkin numbers, Riordan numbers,
Schr\"oder numbers, Eulerian numbers, 
trinomial coefficients, Delannoy numbers, and to 
functions counting free subgroups of finite index in
the inhomogeneous modular group and its lifts.
This leads to numerous new results, including many extensions
of known results to higher powers of~$3$. 
\end{abstract}
\maketitle

\section{Introduction}

In \cite{DeSaAA}, Deutsch and Sagan present a rather systematic
study of the behaviour of Catalan and Motzkin numbers and related
sequences (mostly) modulo $2$ and~$3$. In order to arrive at their
results, they use clever ad hoc arguments, combining combinatorial
and computational techniques. The results in \cite{DeSaAA} which are
of particular interest to us here are those concerning the modulus~3.
One finds numerous such theorems there,
determining the behaviour of Motzkin numbers, Motzkin prefix numbers,
Riordan numbers, central trinomial coefficients, central binomial
coefficients, Catalan numbers, central Delannoy numbers, Schr\"oder
numbers, hex tree numbers, central Eulerian numbers, and Ap\'ery
numbers modulo~3. The purpose of the present paper is to demonstrate
that, with the 
exception of the Ap\'ery numbers, 
all these --- seemingly unrelated --- results are in fact
direct consequences of a {\it single} ``master theorem,"  namely
Theorem~\ref{thm:general} in Section~\ref{sec:gen}. 
Moreover, this same theorem also allows us to refine all these
results to congruences modulo {\it arbitrary powers} of~$3$. 

A second motivation for our paper comes from \cite{MuPu}.
In Theorem~3(i) of that paper, the number of free subgroups 
of given index in the inhomogeneous modular group $PSL_2(\Z)$ is
considered modulo~3. Our original goal was to extend that theorem to
higher powers of~3. The corresponding results are described in
Section~\ref{sec:Free}. In the course of our work on this problem, we
became aware of the paper \cite{DeSaAA} of Deutsch and Sagan, which
pointed the way to many further applications of the method used, and,
in particular, to the unifying ``master theorem."

This ``master theorem" embeds itself into a general
method which applies to a class of sequences given by a
recurrence of finite depth with polynomial
coefficients, and gives rise to
congruences modulo arbitrarily large powers of~$3$.
More precisely, let
\begin{equation} \label{eq:Psidef} 
\Psi(z)=\sum _{k\ge0} ^{}\sum _{n_1>\cdots >n_k\ge0} ^{}
z^{3^{n_1}+3^{n_2}+\dots+3^{n_k}}
=\prod _{j=0} ^{\infty} (1+z^{3^j}).
\end{equation}
Suppose we are given an integer sequence $(f_n)_{n\ge0}$, 
whose generating function
$F(z)=\sum _{n\ge0} ^{}f_n\,z^n$ satisfies a (formal) differential
equation of the form
\begin{equation} \label{eq:diffeq}
\mathcal P(z;F(z),F'(z),F''(z),\dots,F^{(s)}(z))=0,
\end{equation}
where $\mathcal P$ is a polynomial with integer coefficients, which has a
unique power series solution $F(z)$ with integer coefficients.
(In the literature, series obeying a polynomial relation of the
form \eqref{eq:diffeq} are known as {\it differentially algebraic}
series; see, for instance, \cite{BeReAA}.)
In its simplest form,
given a $3$-power $3^\ga$,
what our method affords is an algorithmic procedure to find
a polynomial in the {\it basic series} $\Psi(z)$ given by
\eqref{eq:Psidef} with coefficients which are Laurent polynomials in
$z$ and $1+z$, such that this
polynomial agrees with $F(z)$ when coefficients are reduced 
modulo~$3^\ga$. It turns out that, to make the method more
flexible, in applications one will also have to use variations of the basic
series $\Psi(z)$ obtained by substitutions of the variable $z$.
See Section~\ref{sec:method} for the description of the method and
these variations. 

We point out that this method is in the spirit of
the one developed in \cite{KaKMAA}, where a generating function
method was described to determine the behaviour of
recursive sequences modulo powers of $2$. The most significant
difference lies in the choice of the basic series. In \cite{KaKMAA},
one works with the series $\Phi(z)=\sum _{n\ge0} ^{}z^{2^n}$ and
variations thereof, whereas here we put the series $\Psi(z)$ in the
focus of our attention. A second notable difference is that, because
of the particular behaviour of powers of $\Psi(z)$ modulo~$3$, we
have to allow powers of $1+z$ as denominators in the
coefficients of the polynomials of $\Psi(z)$ that we are considering
here. (In \cite{KaKMAA}, coefficients of the polynomials in $\Phi(z)$
were Laurent polynomials in $z$ only.)

The special feature of the aforementioned ``master theorem,"
Theorem~\ref{thm:general}, is that it treats, in a uniform fashion,
functional equations of the form \eqref{eq:diffeq}, which, after reduction
modulo~$3$, become {\it quadratic} in $F(z)$, apart from satisfying a few more
technical conditions. 
The ubiquity of applications of Theorem~\ref{thm:general}, which we
describe in Sections~\ref{sec:Motzkin}--\ref{sec:Eulerian},
ultimately stems from the fact that numerous combinatorial sequences
have generating functions satisfying a quadratic functional equation.

\medskip
In the remainder of this introduction, we briefly describe the
contents of the paper. 
The next two sections discuss our main character, the formal
power series $\Psi(z)$ defined in \eqref{eq:Psidef}. 
The announced method for finding congruences for recursive sequences
requires knowledge of certain theoretical properties of $\Psi(z)$.
First of all, one needs to know how to extract the explicit value of 
a concrete coefficient in a polynomial expression in $\Psi(z)$ 
modulo a given $3$-power. 
This is discussed in Section~\ref{sec:extr}. 
The appendix 
provides the resulting congruences modulo~$27$
for the coefficients of $\Psi^3(z)$ and $\Psi^5(z)$.
Section~\ref{sec:extr} also contains several theoretical results pertaining
to dependence/independence of powers of $\Psi(z)$, which are
used in Section~\ref{sec:Psi}.
Second, it is of great interest to know 
polynomial identities for $\Psi(z)$ modulo a given $3$-power
which are of {\it minimal degree}. These {\it ``minimal polynomials"}
are the subject of Section~\ref{sec:Psi}. While it is indeed
easy to see that $\psi(z)$ is algebraic modulo powers of~$3$,
the finding of polynomial relations of {\it minimal\/} degree is quite
challenging. We provide a lower bound on the minimal degree, which
in particular shows that
the series $\Psi(z)$ is transcendental over 
$\mathbb Z[z]$ (see Theorem~\ref{conj:1}).
In
Section~\ref{sec:method}, we describe our method of expressing the
generating function of a recursive sequence, when reduced modulo a
given $3$-power, as a polynomial in $\Psi(z)$ with coefficients that
are Laurent polynomials in $z$ and $1+z$. 

We then proceed with 
a first illustration of our method, by applying it to ``almost
central" binomial coefficients, see Section~\ref{sec:sample}.

The ``centrepiece" of our paper is Section~\ref{sec:gen},
which presents the aforementioned ``master theorem." This theorem
describes a general family of functional-differential equations,
where our method from Section~\ref{sec:method} works in a completely
automatic fashion. We demonstrate the power of this theorem 
in the subsequent sections by applying it --- in this order ---
to Motzkin numbers, Motzkin prefix numbers, Riordan numbers,
central trinomial coefficients, central binomial coefficients and their
sums, Catalan numbers, central Delannoy numbers, Schr\"oder numbers,
hex tree numbers, free subgroup numbers
for lifts of the inhomogeneous modular group,
and central Eulerian numbers. All congruence results modulo~$3$
from \cite{DeSaAA}
are derived completely automatically, and moreover we display
the generating function results modulo~9 and 27 which are obtained by
applying our method. The case of central Eulerian numbers is in fact
rather challenging, since it requires several preparations until one
sees that our method is actually applicable, cf.\
Section~\ref{sec:Eulerian}. 
It would cover too much space to display the explicit description
of the congruence classes modulo~$9$ or $27$ (or higher $3$-powers)
for all the above sequences, although this only
amounts to routine calculations. We therefore confine ourselves
to just one illustration, concerning the free subgroup numbers
of $PSL_2(\Z)$, our original motivation for
this paper; see Section~\ref{sec:Free}.

Our paper concludes with a short discussion of congruences modulo powers
of~$3$ for Ap\'ery numbers. They provide examples of number sequences
where our method does not work. Nevertheless, the pattern which they
satisfy modulo~$9$ seems to be very interesting; see
Conjectures~\ref{thm:Apery9} and \ref{thm:Apery39}.

\begin{note}
This paper is accompanied by a {\sl Mathematica} file
and a {\sl Mathematica} notebook so that an interested reader is
able to redo (most of) the computations that are presented in
this article. File and notebook are available at the article's
website
{\tt http://www.mat.univie.ac.at/\lower0.5ex\hbox{\~{}}kratt/artikel/3psl2z.html}.
\end{note}

\section{Coefficient extraction from powers of $\Psi(z)$}
\label{sec:extr}

In Section~\ref{sec:method} we are going to describe a method for expressing
differentially algebraic formal power series, after
the coefficients of the series have been reduced modulo $3^\ga$, 
as polynomials in the
Cantor-like series\footnote{We call it ``Cantor-like"
since the exponents of the monomials
appearing in $\Psi(z)$ all have a ternary expansion consisting of
$1$'s and $0$'s only. This is reminiscent of the ternary expansion of
the numbers in the classical Cantor set all of which have a ternary
expansion consisting of $2$'s and $0$'s only.} $\Psi(z)$ (see \eqref{eq:Psidef} for the definition), the
coefficients being Laurent polynomials in $z$ and $1+z$. 
Such a method would be
without value if we could not, at the same time, provide a procedure
for extracting coefficients from the series $(1+z)^{-L}\Psi^K(z)$,
where $K$ and $L$ are non-negative integers. The description
of such a procedure is the topic of this section.
Roughly speaking, this procedure consists of expanding powers of
$\Psi(z)$ in terms of the auxiliary series in \eqref{eq:tHdef}
(see \eqref{eq:Psipoteven} and \eqref{eq:Psipotodd}), which can
be further expanded using only a subclass of these series 
(see Lemma~\ref{lem:Hbi}),
and then reading the coefficients from these auxiliary series, which
is (more or less) straightforward (see the appendix).

\medskip
We begin with an expansion of $\Psi^2(z)$, on which all subsequent
arguments in this section will be based.

\begin{lemma} \label{lem:Psi-2}
We have
\begin{equation} \label{eq:Psi-2}
\Psi^2(z)=\frac {1} {1+z}\sum_{s\ge0}
\sum_{k_1>\dots>k_s\ge0}3^s
\prod _{j=1} ^{s}\frac {z^{3^{k_j}}(1+z^{3^{k_j}})} {1+z^{3^{k_j+1}}},
\end{equation}
where the term for $s=0$ has to be interpreted as $1$.
\end{lemma}

\begin{proof}
We have
\begin{align*}
\Psi^2(z)&=\prod _{j=0} ^{\infty} (1+z^{3^j})^2\\
&=\frac {1} {1+z}\prod _{j=0} ^{\infty}
\frac {(1+z^{3^j})^3} {1+z^{3^{j+1}}} \\
&=\frac {1} {1+z}\prod _{j=0} ^{\infty}
\left(1+3\frac {z^{3^j}(1+z^{3^j})} {1+z^{3^{j+1}}} \right).
\end{align*}
Expansion of the product directly results in \eqref{eq:Psi-2}.
\end{proof}

\begin{remark}
The significance of Lemma~\ref{lem:Psi-2} lies largely in the
fact that, if one considers \eqref{eq:Psi-2} modulo a $3$-power,
then the sum over $s$ on the right-hand side gets truncated. We shall
exploit this frequently. The simplest example is the reduction of
\eqref{eq:Psi-2} modulo~3, which we state here explicitly as
\begin{equation} \label{eq:PsiEq}
\Psi^2(z)=\frac {1} {1+z}\quad 
\text {modulo }3,
\end{equation}
since we shall make use of it later at several points.
\end{remark}

For convenience, we write
\begin{equation} \label{eq:tHdef} 
\widetilde 
H_{a_1,a_2,\dots,a_s}(z):=
\sum_{k_1>\dots>k_s\ge0}
\prod _{j=1} ^{s}\left(\frac {z^{3^{k_j}}(1+z^{3^{k_j}})}
{1+z^{3^{k_j+1}}}\right)^{a_j}.
\end{equation}
Using this notation, Lemma~\ref{lem:Psi-2} can be rephrased as
\begin{equation} \label{eq:Psi-2A}
\Psi^2(z)=\frac {1} {1+z}\sum_{s\ge0}
3^s\widetilde H_{\underbrace{\scriptstyle 1,1,\dots,1}_{s\text{ times}}}(z).
\end{equation}
We point out that $\widetilde 
H_{a_1,a_2,\dots,a_s}(z)$ is the monomial quasi-symmetric function
$$M_{a_1,a_2,\dots,a_s}(R(0),R(1),\dots),$$ 
where
\begin{equation} \label{eq:R} 
R(k)=\frac {z^{3^{k}}(1+z^{3^{k}})}
{1+z^{3^{k+1}}}
\end{equation}
(cf.\ \cite[p.~357]{StanBI}).

It is not difficult to see that powers of
$\Psi(z)$ can be expressed in the form
\begin{equation} \label{eq:Psipoteven}
\Psi^{2K}(z)=\frac {1} {(1+z)^K}\sum _{r=1} ^{K}
{\sum _{a_1,\dots,a_r\ge1} ^{}}
c_{K}(a_1,a_2,\dots,a_r)
\widetilde H_{a_1,a_2,\dots,a_r}(z),
\end{equation}
respectively as
\begin{equation} \label{eq:Psipotodd}
\Psi^{2K+1}(z)=\frac {1} {(1+z)^K}\Psi(z)\sum _{r=1} ^{K}
{\sum _{a_1,\dots,a_r\ge1} ^{}}
c_{K}(a_1,a_2,\dots,a_r)
\widetilde H_{a_1,a_2,\dots,a_r}(z),
\end{equation}
where the coefficients $c_{K}(a_1,a_2,\dots,a_r)$ 
can be determined explicitly.

Consequently, 
the coefficient extraction problem will be solved if we
are able to say how to extract coefficients from the series
$$(1+z)^{-L}\widetilde H_{a_1,a_2,\dots,a_r}(z)\quad \text {and}\quad 
(1+z)^{-L}\Psi(z)\widetilde H_{a_1,a_2,\dots,a_r}(z).$$
This is in fact rather straightforward,
if one has the patience to distinguish enough subcases. 
Roughly speaking, one analyses coefficient extraction from the subexpressions 
$$
\Psi(z)\frac {1} {(1+z)^L}
\prod _{j=1} ^{s}\left(\frac {z^{3^{k_j}}(1+z^{3^{k_j}})} 
{1+z^{3^{k_j+1}}}\right)^{a_j}
$$
and
$$
\frac {1} {(1+z)^L}
\prod _{j=1} ^{s}\left(\frac {z^{3^{k_j}}(1+z^{3^{k_j}})} 
{1+z^{3^{k_j+1}}}\right)^{a_j}.
$$
This is made more precise in the appendix.
How this is used is exemplified for the case of 
free subgroup numbers of 
the inhomogeneous modular group in Section~\ref{sec:Free}.

\medskip
It is also important to have information on dependence or
independence of the series $\widetilde H_{a_1,a_2,\dots,a_r}(z)$.
When $a_1,a_2,\dots,a_r$ vary over all possible choices, 
the series $\widetilde H_{a_1,a_2,\dots,a_r}(z)$ are {\it not\/}
linearly independent over the ring $(\Z/3\Z)[z]$, or over
$(\Z/3^\ga\Z)[z]$ for a positive integer $\ga$.

However, we shall show below (see Corollary~\ref{lem:Hind}) that, 
if we restrict to $a_i$'s which are not divisible by~$3$, 
then the corresponding series $\widetilde H_{a_1,a_2,\dots,a_r}(z)$,
together with the (trivial) series $1$, 
{\it are} linearly independent over $(\Z/3\Z)[z]$, and hence
as well over $(\Z/3^\ga\Z)[z]$ for arbitrary positive integers
$\ga$ (and also over $\Z[z]$).
This fact underpins the arguments in the next section.

For {\it arbitrary} tuples of integers $(b_1,b_2,\dots,b_s)$, the
series $\widetilde H_{b_1,b_2,\dots,b_s}(z)$
can be expressed as a linear combination over
$(\Z/3^\ga\Z)[z,(1+z)^{-1}]$ 
of $1$ and the {\it former} series, and the corresponding
computation can be carried out in
an algorithmic fashion; see Lemma~\ref{lem:Hbi}. 

\medskip
The problem of linear dependence/independence will be disposed of by
transferring it to the problem of dependence/independence for the
simpler series
$$
H_{a_1,a_2,\dots,a_r}(z):=\sum_{n_1>n_2>\dots>n_r\ge0}
z^{a_13^{n_1}+a_23^{n_2}+\dots+a_r3^{n_r}}.
$$
These series are in fact the analogues, for the prime~$3$, of series
which were considered in \cite[Sec.~3]{KaKMAA} for the prime $2$, 
and which were instrumental there in a theory for solving
differential and functional equations modulo powers of $2$.
To see their relevance in the present context, one should observe
that 
\begin{align} \notag
\widetilde H_{a_1,a_2,\dots,a_r}(z)&=
\sum_{k_1>\dots>k_r\ge0}
\prod _{j=1} ^{s}\left(\frac {z^{3^{k_j}}(1+z^{3^{k_j}})}
{1+z^{3^{k_j+1}}}\right)^{a_j}\\
\notag
&=
\sum_{k_1>\dots>k_r\ge0}
\prod _{j=1} ^{s}\left(\frac {z^{3^{k_j}}(1+z)^{3^{k_j}}}
{(1+z)^{3^{k_j+1}}}\right)^{a_j}\quad 
\text {modulo }3\\
&=
H_{a_1,a_2,\dots,a_r}\left(\frac {z} {(1+z)^2}\right)\quad 
\text {modulo }3.
\label{eq:HH}
\end{align}
The idea which we have in mind is that,
as a result of the last congruence,
independence of (certain) series $H_{a_1,a_2,\dots,a_r}(z)$ implies
{\it a fortiori} independence of the corresponding series
$\widetilde H_{a_1,a_2,\dots,a_r}(z)$.

The auxiliary results which follow now are analogues of Lemma~6,
Corollary~7, and Lemma~9 in \cite{KaKMAA}.
The first of these pertains to the uniqueness of
representations of integers as sums of powers of~$3$ with
multiplicities, tailor-made for application to the series
$H_{a_1,a_2,\dots,a_r}(z)$.

\begin{lemma} \label{lem:aiodd}
Let $d,r,s$ be positive integers with $r\ge s$, $c$ an integer
with $\vert c\vert\le d$, and let 
$a_1,a_2,\dots,a_r$ respectively $b_1,b_2,\dots,b_s$ be two sequences of 
integers, none of them divisible by~$3$, 
with $1\le a_i\le d$ for $1\le i\le r$, and $1\le b_i\le
d$ for $1\le i\le s$. If
\begin{equation} \label{eq:a2b2}
a_13^{2rd}+a_23^{2(r-1)d}+\dots+a_r3^{2d}=
b_13^{n_1}+b_23^{n_2}+\dots+b_s3^{n_s}+c
\end{equation}
for integers $n_1,n_2,\dots,n_s$ with $n_1>n_2>\dots>n_s\ge0$, then
$r=s$, $c=0$, $a_i=b_i$, and $n_i=2d(r+1-i)$ for $i=1,2,\dots,r$.
\end{lemma}

\begin{proof}
We use induction on $r$.

First, let $r=1$. Then $s=1$ as well, and \eqref{eq:a2b2} becomes
\begin{equation} \label{eq:a2b2A}
a_13^{2d}=b_13^{n_1}+c.
\end{equation}
If $n_1>2d$, then the above equation, together with the assumption
that $a_1$ is not divisible by~$3$, implies
$$
a_13^{2d}\equiv c\quad \text {modulo }3^{2d+1}.
$$
However, by assumption, we have $\vert c\vert\le d<a_13^{2d}$, which is
absurd.

If $d<n_1< 2d$, then it follows from \eqref{eq:a2b2A} that $c$ must
be divisible by $3^{n_1}$. Again by assumption, we have 
$\vert c\vert\le d<3^d<3^{n_1}$, so that $c=0$. But then
\eqref{eq:a2b2A} cannot be satisfied since $b_1$ is
assumed not to be divisible by~$3$.

If $0\le n_1\le d$, then we estimate
$$
b_13^{n_1}+c\le d\left(3^{d}+1\right)\le (3^d-1)(3^d+1)<3^{2d},
$$
which is again a contradiction to \eqref{eq:a2b2A}. 

The only remaining possibility is $n_1=2d$. If this is substituted in
\eqref{eq:a2b2A} and the resulting equation is combined with 
$\vert c\vert\le d<3^{2d}$, then the conclusion is that the equation
can only be satisfied if $c=0$ and $a_1=b_1$, in accordance with
the assertion of the lemma.

\medskip
We now perform the induction step. We assume that the assertion of
the lemma is established for all $r<R$, and we want to show that this
implies its validity for $r=R$. Let $t$ be maximal such that $n_t\ge
2d$. Then reduction of \eqref{eq:a2b2} modulo $3^{2d}$ yields
\begin{equation} \label{eq:a2b2B}
b_{t+1}3^{n_{t+1}}+b_{t+2}3^{n_{t+2}}+\dots+b_s3^{n_s}+c\equiv 0\quad 
\text {modulo }3^{2d}.
\end{equation}
Let us write $b\cdot 3^{2d}$ for the left-hand side in
\eqref{eq:a2b2B}. Then, by dividing \eqref{eq:a2b2} (with $R$ instead
of $r$) by $3^{2d}$, we obtain
\begin{equation} \label{eq:a2b2C}
a_13^{2(R-1)d}+a_23^{2(R-2)d}+\dots+a_{R-1}3^{2d}=
b_13^{n_1-2d}+b_23^{n_2-2d}+\dots+b_t3^{n_t-2d}+b-a_R.
\end{equation}
We have
\begin{align*}
0\le b&\le
3^{-2d}d\left(3^{2d-1}+3^{2d-2}+\dots+3^{2d-s+t}+1\right)\\[2mm]
&\le 3^{-2d}d\left(\tfrac {1} {2}(3^{2d}-3^{2d-s+t})+1\right)\le d.
\end{align*}
Consequently, we also have $\vert b-a_R\vert\le d$.
This means that we are in a position to apply the induction
hypothesis to \eqref{eq:a2b2C}. The conclusion is that
$t=R-1$, $b-a_R=0$, $a_i=b_i$, and $n_i=2d(R+1-i)$ for
$i=1,2,\dots,R-1$. If this is used in \eqref{eq:a2b2} with $r=R$,
then we obtain
$$
a_R3^{2d}=c
$$
or
$$
a_R3^{2d}=b_R3^{n_R}+c,
$$
depending on whether $s=R-1$ or $s=R$. The first case is absurd since
$c\le d<3^{2d}\le a_R3^{2d}$. On the other hand, the second case
has already been
considered in \eqref{eq:a2b2A}, and we have seen there that it
follows that $c=0$, $a_R=b_R$, and $n_R=2d$.

This completes the proof of the lemma.
\end{proof}

The independence of the series $H_{a_1,a_2,\dots,a_r}(z)$
with all $a_i$'s not divisible by~$3$ is now an easy consequence.

\begin{corollary} \label{lem:Hind}
The series $H_{a_1,a_2,\dots,a_r}(z)$, with all $a_i$\!'s not divisible
by~$3$, together
with the series $1$ are
linearly independent over $(\Z/3\Z)[z]$, and consequently as
well over $(\Z/3^\ga\Z)[z]$ for an arbitrary positive integer
$\ga$, and over $\Z[z]$.
\end{corollary}

\begin{proof}
Let us suppose that
\begin{equation} \label{eq:Hlincomb}
p_0(z)+\sum _{i=1}
^{N}p_i(z)H_{a_1^{(i)},a_2^{(i)},\dots,a_{r_i}^{(i)}}(z)=0,
\end{equation} 
where the $p_i(z)$'s are non-zero polynomials in $z$
over $\Z/3\Z$ (respectively over $\Z/3^\ga\Z$, or over $\Z$), the
$r_i$'s are positive integers, and 
$a_j^{(i)}$, $j=1,2,\dots,r_i$, $i=1,2,\dots,N$, are integers not
divisible by~$3$. 
We may also assume that the tuples
$(a_1^{(i)},a_2^{(i)},\dots,a_{r_i}^{(i)})$, $i=1,2,\dots,N$, 
are pairwise distinct. Choose $i_0$ such
that $r_{i_0}$ is maximal among the $r_i$'s. Without loss of
generality, we may assume that the coefficient of $z^0$ in $p_{i_0}(z)$
is non-zero (otherwise we could multiply both sides of 
\eqref{eq:Hlincomb} by an appropriate power of $z$). Let $d$ be the
maximum of all $a_j^{(i)}$'s and the absolute values of exponents of
$z$ appearing in monomials with non-zero coefficient in
the polynomials $p_i(z)$, $i=0,1,\dots,N$. Then, according to
Lemma~\ref{lem:aiodd} with $r=r_{i_0}$, $a_j=a_j^{(i_0)}$,
$j=1,2,\dots,r_{i_0}$, the coefficient of
$$
z^{a_1^{(i_0)}3^{2rd}+a_2^{(i_0)}3^{2(r-1)d}+\dots+a_r^{(i_0)}3^{2d}}
$$
is $1$ in $H_{a_1^{(i_0)},a_2^{(i_0)},\dots,a_{r_{i_0}}^{(i_0)}}(z)$,
while it is zero in series 
$z^eH_{a_1^{(i_0)},a_2^{(i_0)},\dots,a_{r_{i_0}}^{(i_0)}}(z)$, where
$e$ is a non-zero integer with $\vert e\vert\le d$, and in all other series
$z^eH_{a_1^{(i)},a_2^{(i)},\dots,a_{r_i}^{(i)}}(z)$,
$i=1,\dots,i_0-1,i_0+1,\dots,N$, where $e$ is a (not necessarily
non-zero) integer with $\vert e\vert\le d$.
This contradiction to \eqref{eq:Hlincomb} establishes
the assertion of the corollary.
\end{proof}

In view of \eqref{eq:HH}, the same result holds for the series
$\widetilde H_{a_1,a_2,\dots,a_r}(z)$.

\begin{corollary} \label{cor:Htildeind}
The series $\widetilde H_{a_1,a_2,\dots,a_r}(z)$, 
with all $a_i$\!'s not divisible by~$3$, together
with the series $1$ are
linearly independent over $(\Z/3\Z)[z]$, and consequently as
well over $(\Z/3^\ga\Z)[z]$ for an arbitrary positive integer
$\ga$, and over $\Z[z]$.
\end{corollary}

The reader should recall
from \eqref{eq:Psipoteven} and \eqref{eq:Psipotodd} 
that even powers of $\Psi(z)$ can be
expanded in the series $\widetilde H_{a_1,a_2,\dots,a_s}(z)$,
while odd powers can be expanded in the series
$\Psi(z)\widetilde H_{a_1,a_2,\dots,a_s}(z)$. The next
lemma shows that these two families of series are linearly independent
from each other.

\begin{lemma} \label{lem:evenodd}
Let $\ga$ be a positive integer.
If
\begin{equation} \label{eq:Psirel}
\sum _{j=0} ^{N}a_{2j}(z)\Psi^{2j}(z)
=\sum _{j=1} ^{N}a_{2j-1}(z)\Psi^{2j-1}(z)
\end{equation}
for polynomials $a_i(z)$ in $z$ over $\Z/3^\ga\Z$, then
both sides in \eqref{eq:Psirel} must be zero.
\end{lemma}

\begin{proof}
We use the expansions \eqref{eq:Psipoteven} and \eqref{eq:Psipotodd}. 
This leads to a relation of
the form
\begin{equation} \label{eq:PsiH}
\sum _{} ^{}c_{a_1,a_2,\dots,a_s}(z)
\widetilde H_{a_1,a_2,\dots,a_s}(z)=
\Psi(z)\sum _{} ^{}d_{a_1,a_2,\dots,a_s}(z)
\widetilde H_{a_1,a_2,\dots,a_s}(z),
\end{equation}
for certain polynomials $c_{a_1,a_2,\dots,a_s}(z)$ and
$d_{a_1,a_2,h\dots,a_s}(z)$ in $z$. We consider this relation modulo~$3$.
By \eqref{eq:HH}, we then have
\begin{multline*}
\sum _{} ^{}c_{a_1,a_2,\dots,a_s}(z)
H_{a_1,a_2,\dots,a_s}\left(\frac {z} {(1+z)^2}\right)\\
=
\Psi(z)\sum _{} ^{}d_{a_1,a_2,\dots,a_s}(z)
H_{a_1,a_2,\dots,a_s}\left(\frac {z} {(1+z)^2}\right)\quad 
\text {modulo }3.
\end{multline*}
Next we apply the square root to both sides of
\eqref{eq:PsiEq}, to obtain
$$
\Psi(z)=(1+z)^{-1/2}\quad \text {modulo }3.
$$
Substitution of this congruence in the congruence above then yields
\begin{multline} \label{eq:HpsiH}
\sum _{} ^{}c_{a_1,a_2,\dots,a_s}(z)
H_{a_1,a_2,\dots,a_s}\left(\frac {z} {(1+z)^2}\right)\\
=
(1+z)^{-1/2}\sum _{} ^{}d_{a_1,a_2,\dots,a_s}(z)
H_{a_1,a_2,\dots,a_s}\left(\frac {z} {(1+z)^2}\right)\quad 
\text {modulo }3.
\end{multline}
Now we expand both sides as Laurent series in $1+z$, allowing only a
finite number of positive powers of $1+z$. In order to do so, we
observe that 
$$
\frac {z^k} {(1+z)^{2k}}=\sum _{\ell=0} ^{k}(-1)^\ell \binom k\ell
(1+z)^{-k-\ell}.
$$
This is used in each term of the expansions corresponding to
$H_{a_1,a_2,\dots,a_s}(z/(1+z)^2)$. As a consequence, we see that the
left-hand side of \eqref{eq:HpsiH} is a (formal) Laurent series in $1+z$
with only finitely many positive powers of $1+z$, while the
right-hand side is a Laurent series in $1+z$ of the same type, but
multiplied by $(1+z)^{-1/2}$. This contradiction completes
the proof of the lemma.
\end{proof}

We may now combine 
Corollary~\ref{cor:Htildeind} and Lemma~\ref{lem:evenodd}.

\begin{theorem} \label{thm:Htildeind}
The series $\widetilde H_{a_1,a_2,\dots,a_r}(z)$ and
$\Psi(z)\widetilde H_{a_1,a_2,\dots,a_r}(z)$, 
with all $a_i$\!'s not divisible by~$3$, together
with the series $1$ are
linearly independent over $(\Z/3\Z)[z]$, and consequently as
well over $(\Z/3^\ga\Z)[z]$ for an arbitrary positive integer
$\ga$, and over $\Z[z]$.
\end{theorem}

In order to show that each 
$\widetilde H_{b_1,b_2,\dots,b_s}(z)$
can be expressed as a linear combination over $(\Z/{3^\ga}\Z)[z,(1+z)^{-1}]$
of the series $1$ and series of the form
$\tilde H_{a_1,a_2,\dots,a_r}(z)$, where none of the $a_i$'s is
divisible by~$3$ (see Lemma~\ref{lem:Hbi} below),
we need two auxiliary identities, which are the subject of the
following two lemmas.

\begin{lemma} \label{lem:2ids}
Let $n$ and $k$ be positive integers. Then
\begin{multline} \label{eq:id0}
\left(\frac {z^{3^k}(1+z^{3^k})} {1+z^{3^{k+1}}}\right)^{3n}
=\sum _{b\ge0} ^{}3^{b}\binom {-n}b
\left(\frac {z^{3^{k+1}}(1+z^{3^{k+1}})}
{1+z^{3^{k+2}}}\right)^{n+b}\\
+\underset{b\ge0}{\sum _{a\ge1} ^{}}
\sum _{s_1,s_2,\dots,t_1,t_2,\dots\ge0} ^{}
3^{a+b+\V s+\V t}\binom na\binom {-n}b
\left(\prod _{j\ge1} ^{}\binom {s_{j-1}}{s_j}\binom
{-s_{j-1}}{t_j}\right)\\
\cdot
\left(\frac {z^{3^k}(1+z^{3^k})} {1+z^{3^{k+1}}}\right)^{3(\V
s+n+b)+\V t+a},
\end{multline}
where, by convention, $s_0=n+b$, and we write $\V s$ for
$s_1+s_2+\cdots$, with an analogous meaning for $\V t$.
\end{lemma}
\begin{proof}
First we argue that
the sum in \eqref{eq:id0} is well-defined as a formal power series. 
Namely, given a positive integer $n+b$,
non-zero summands in \eqref{eq:id0} arise only for
$n+b=s_0\ge s_1\ge s_2\ge \cdots$. Hence, 
if we concentrate on vectors $s$ and $t$ for which $\V s+\V t=N$,
with $N$ a fixed positive integer,
then there is only a finite number of vectors $s$ for which non-zero
summands exist. In particular, we have $s_i=0$ for $i>N$. In its
turn, this entails $t_i=0$ for $i>N$ if one wants to have a
non-vanishing summand. Therefore there are as well only finitely
many possibilities for vectors $t$ with $\V t\le N$. 

Now, in order to prove the identity in \eqref{eq:id0},
we start with the expansion
\begin{align} 
\notag
\left(\frac {z^{3^k}(1+z^{3^k})} {1+z^{3^{k+1}}}\right)^{3n}
&=z^{3^{k+1}n}\frac {(1+z^{3^{k+1}})^n} 
{(1+z^{3^{k+2}})^n}
\frac {\left(1+3z^{3^k}\frac {1+z^{3^k}}
{1+z^{3^{k+1}}}\right)^n} 
{\left(1+3z^{3^{k+1}}\frac {1+z^{3^{k+1}}}
{1+z^{3^{k+2}}}\right)^n}\\
&=\sum _{a,b\ge0} ^{}3^{a+b}\binom na\binom {-n}b
\left(\frac {z^{3^k}(1+z^{3^k})} {1+z^{3^{k+1}}}\right)^a
\left(\frac {z^{3^{k+1}}(1+z^{3^{k+1}})}
{1+z^{3^{k+2}}}\right)^{n+b}.
\label{eq:id1}
\end{align}
The part of the double sum consisting of the terms with $a=0$
directly yields the first sum on the right-hand side of
\eqref{eq:id0}. If $a\ge1$, we want ``to get rid of" the
power with exponent $n+b$. In order to do so, we observe that, for an
arbitrary positive integer $m$, we have
\begin{multline*} 
\left(\frac {z^{3^{k+1}}(1+z^{3^{k+1}})}
{1+z^{3^{k+2}}}\right)^{m}
=z^{3^{k+1}m}\frac {(1+z^{3^k})^{3m}} {(1+z^{3^{k+1}})^{3m}} 
\frac {\left(1+3z^{3^{k+1}}\frac {1+z^{3^{k+1}}}
{1+z^{3^{k+2}}}\right)^{m}} 
{\left(1+3z^{3^{k}}\frac {1+z^{3^{k}}}
{1+z^{3^{k+1}}}\right)^{m}}\\
=\sum _{s_1,t_1\ge0} ^{}3^{s_1+t_1}\binom m{s_1}\binom {-m}{t_1}
\left(\frac {z^{3^{k}}(1+z^{3^{k}})}
{1+z^{3^{k+1}}}\right)^{3m+t_1}
\left(\frac {z^{3^{k+1}}(1+z^{3^{k+1}})}
{1+z^{3^{k+2}}}\right)^{s_1}.
\end{multline*}
This relation is now iterated, to get
\begin{multline} \label{eq:k+1->k}
\left(\frac {z^{3^{k+1}}(1+z^{3^{k+1}})}
{1+z^{3^{k+2}}}\right)^{m}\\
=\sum _{s_1,s_2,\dots,t_1,t_2,\dots\ge0} ^{}
3^{\V s+\V t}\left(\prod _{j\ge1} ^{}\binom {s_{j-1}}{s_j}\binom
{-s_{j-1}}{t_j}\right)
\left(\frac {z^{3^k}(1+z^{3^k})} {1+z^{3^{k+1}}}\right)^{3m+3\V
s+\V t},
\end{multline}
where, by convention, $s_0=m$.
Finally, if this identity with $m=n+b$ is substituted in 
\eqref{eq:id1} in the terms with $a\ge1$, then we obtain the
right-hand side of \eqref{eq:id0}.
\end{proof}

\begin{remark}
What Lemma~\ref{lem:2ids} affords is an expansion of
a power of 
\begin{equation} \label{eq:3k} 
z^{3^{k+1}}(1+z^{3^{k+1}})/(1+z^{3^{k+2}})
\end{equation}
with exponent divisible by~$3$
in terms which either have {\it coefficients} with higher
divisibility by~$3$ or are a power of \eqref{eq:3k}
with {\it exponent\/} of
smaller $3$-adic valuation than the original exponent 
(see the term for $b=0$ in the first sum on the right-hand 
side of \eqref{eq:id0}).
\end{remark}

\begin{lemma} \label{lem:1+z}
For all non-negative integers $j$ and positive integers $\al$ and
$\be$, we have
$$
\frac {1} {(1+z^{3^j})^\al}=\Pol_{j,\al,\be}\left(z,(1+z)^{-1}\right)
\quad \text {\em modulo }3^\be,
$$
where $\Pol_{j,\al,\be}\left(z,(1+z)^{-1}\right)$ is a polynomial in $z$ and
$(1+z)^{-1}$ with integer coefficients.
\end{lemma}

\begin{proof} We perform an induction with respect to $j+\be$.
Clearly, for $\be=1$ and $j=0$, there is nothing to prove.

For the induction step, we write
\begin{align*}
\frac {1} {(1+z^{3^j})^\al}&=
\frac {1} {(1+z^{3^{j-1}})^{3\al}}
+\frac {1} {(1+z^{3^j})^\al}
-\frac {1} {(1+z^{3^{j-1}})^{3\al}}\\
&=
\frac {1} {(1+z^{3^{j-1}})^{3\al}}
+\frac {1} {(1+z^{3^{j-1}})^{3\al}(1+z^{3^j})^\al}
\left((1+z^{3^{j-1}})^{3\al}-(1+z^{3^j})^\al\right)\\
&=
\frac {1} {(1+z^{3^{j-1}})^{3\al}}
+\frac {1} {(1+z^{3^{j-1}})^{3\al}(1+z^{3^j})^\al}\\
&\kern4cm
\times
\left((1+3z^{3^{j-1}}+3z^{2\cdot 3^{j-1}}+z^{3^{j}})^{\al}
-(1+z^{3^j})^\al\right)\\
&=
\frac {1} {(1+z^{3^{j-1}})^{3\al}}
+\frac {1} {(1+z^{3^{j-1}})^{3\al}(1+z^{3^j})^\al}\\
&\kern4cm
\times
\sum _{\ell=1} ^{\al}\binom\al\ell 
3^\ell\left(z^{3^{j-1}}+z^{2\cdot 3^{j-1}}\right)^{\ell}
\left(1+z^{3^j}\right)^{\al-\ell}\\
&=
\frac {1} {(1+z^{3^{j-1}})^{3\al}}
+
\sum _{\ell=1} ^{\al}\binom\al\ell 
\frac {3^\ell z^{\ell \cdot 3^{j-1}}} 
{(1+z^{3^{j-1}})^{3\al-\ell}(1+z^{3^j})^\ell}.
\end{align*}
Now the induction hypothesis can be applied, and yields
\begin{multline*}
\frac {1} {(1+z^{3^j})^\al}=
\Pol_{j-1,3\al,\be}\left(z,(1+z)^{-1}\right)\\
+
\sum _{\ell=1} ^{\al}\binom\al\ell 
{3^\ell z^{\ell \cdot 3^{j-1}}} 
\Pol^{3\al-\ell}_{j-1,3\al-\ell,\be-\ell}\left(z,(1+z)^{-1}\right)
\Pol^{\ell}_{j,\ell,\be-\ell}\left(z,(1+z)^{-1}\right)\\
\text {modulo }3^\be.
\end{multline*}
This completes the induction.
\end{proof}

We are now ready for the proof of the announced expansion result.

\begin{lemma} \label{lem:Hbi}
For positive integers $b_1,b_2,\dots,b_s$ and $\ga$, the series
$\widetilde H_{b_1,b_2,\dots,b_s}(z)$, 
can be expressed as a linear combination
over $(\Z/3^\ga\Z)[z,(1+z)^{-1}]$ 
of the series $1$ and series of the form
$\widetilde H_{a_1,a_2,\dots,a_r}(z)$, 
where none of the $a_i$'s is divisible by~$3$.
\end{lemma}

\begin{remark}
What the proof below actually does is to find an expansion of
$\widetilde H_{b_1,b_2,\dots,b_s}(z)$ over $\Z_3[z,(1+z)^{-1}]$,
where $\Z_3$ denotes the ring of $3$-adic integers.
As this point of view will not play an explicit role in the sequel,
we shall not pursue this aspect any further.
\end{remark}

\begin{proof}[Proof of Lemma~\em\ref{lem:Hbi}]
With each term
$$c_{b_1,b_2,\dots,b_s}
\widetilde H_{b_1,b_2,\dots,b_s}(z),$$
$c_{b_1,b_2,\dots,b_s}$ being some integer,
we associate the quadruple $(s,v,i,t)$, where
$$
v=v_3(c_{b_1,b_2,\dots,b_s}),
$$
$i$ is the maximal index such that  
$b_i$ is divisible by~$3$, and
$t=v_3(b_i)$.
We describe an algorithmic procedure for expressing 
$\widetilde H_{b_1,b_2,\dots,b_s}(z)$ 
in terms of series 
$\widetilde H_{a_1,a_2,\dots,a_r}(z)$
with associated quadruples less than the quadruple associated with 
$\widetilde H_{b_1,b_2,\dots,b_s}(z)$, according to
the following total order of quadruples:
we define $(s_1,v_1,i_1,t_1)\prec(s_2,v_2,i_2,t_2)$ if, and only if,
\begin{align*} 
s_1&<s_2,\\
\text {or }s_1&=s_2 \text { and }v_1>v_2,\\
\text {or }s_1&=s_2,\ v_1=v_2, \text { and }i_1<i_2,\\
\text {or }s_1&=s_2,\ v_1=v_2,\ i_1=i_2,\text { and }t_1<t_2.
\end{align*}

Our algorithmic procedure consists of four recurrence relations,
\eqref{eq:Rek1}--\eqref{eq:Rek4} below. These are based on
Hou's reduction idea in \cite{HouQAA}.

Let $h$ be an integer between $1$ and $s$.
Furthermore, we assume that $b_h$ is divisible by~$3$ and that
all $b_i$'s with $i>h$ are not divisible by~$3$. We write
$b_h=3b'_h$. Then,
from the definition of the series
$\widetilde H_{b_1,b_2,\dots,b_s}(z)$,
and by Lemma~\ref{lem:2ids}, we have
\begin{align}
\notag
\widetilde H&_{b_1,b_2,\dots,b_s}(z)
= \sum _{k_1>\dots>k_{s-1}>k_s\ge 0} ^{}
\left(\underset{j\ne h}{\prod _{j=1} ^{s}}
\left(\frac {z^{3^{k_j}}(1+z^{3^{k_j}})}
{1+z^{3^{k_j+1}}}\right)^{b_j}\right)
\left(\frac {z^{3^{k_h}}(1+z^{3^{k_h}})}
{1+z^{3^{k_h+1}}}\right)^{3b'_h}\\[2mm]
\notag
&= 
\sum _{b\ge0} ^{}3^{b}\binom {-b'_h}b
\sum _{k_1>\dots>k_{s-1}>k_s\ge 0} ^{}
\left(\underset{j\ne h}{\prod _{j=1} ^{s}}
\left(\frac {z^{3^{k_j}}(1+z^{3^{k_j}})}
{1+z^{3^{k_j+1}}}\right)^{b_j}\right)
\left(\frac {z^{3^{k_h+1}}(1+z^{3^{k_h+1}})}
{1+z^{3^{k_h+2}}}\right)^{b'_h+b}\\
\notag
&\kern1cm
+\underset{b\ge0}{\sum _{a\ge1} ^{}}
\sum _{s_1,s_2,\dots,t_1,t_2,\dots\ge0} ^{}
3^{a+b+\V s+\V t}\binom {b'_h}a\binom {-b'_h}b
\left(\prod _{j\ge1} ^{}\binom {s_{j-1}}{s_j}\binom
{-s_{j-1}}{t_j}\right)\\
&\kern1.3cm
\cdot
\sum _{k_1>\dots>k_{s-1}>k_s\ge 0} ^{}
\left(\underset{j\ne h}{\prod _{j=1} ^{s}}
\left(\frac {z^{3^{k_j}}(1+z^{3^{k_j}})}
{1+z^{3^{k_j+1}}}\right)^{b_j}\right)
\left(\frac {z^{3^{k_h}}(1+z^{3^{k_h}})} {1+z^{3^{k_h+1}}}\right)^{3(\V
s+b'_h+b)+\V t+a}.
\label{eq:Hred1}
\end{align}

Now we distinguish four cases. First, let $1<h<s$.
In the first of the above sums on the right-hand side,
let $k'_h=k_h+1$ be a new summation index. We observe that
we have
\begin{multline*} 
\{(k_1,\dots,k_{h-1},k_h+1,k_{h+1},\dots,k_s):
k_1>\dots>k_{h-1}>k_h>k_{h+1}>\dots>k_s\ge0\}
\kern.8cm
\\=
\Big(
\{(k_1,\dots,k_{h-1},k'_h,k_{h+1},\dots,k_s):
k_1>\dots>k_{h-1}>k'_h>k_{h+1}>\dots>k_s\ge0\}
\kern.9cm\\
\dot\cup
\{(k_1,\dots,k_{h-1},k_{h-1},k_{h+1},\dots,k_s):
k_1>\dots>k_{h-1}>k_{h+1}>\dots>k_s\ge0\}
\Big)
\\
\big\backslash
\{(k_1,\dots,k_{h-1},k_{h+1}+1,k_{h+1},\dots,k_s):
k_1>\dots>k_{h-1}>k_{h+1}>\dots>k_s\ge0\},
\end{multline*}
where the union is a disjoint union, and the last set is
entirely contained in that disjoint union.
If we apply this observation to the (new) index set of the first
sum, then, from \eqref{eq:Hred1}, we obtain
\begin{align*}
\notag
\widetilde H&_{b_1,b_2,\dots,b_s}(z)
= \sum _{b\ge0} ^{}3^{b}\binom {-b'_h}b
\left(\widetilde H_{b_1,b_2,\dots,b_{h-1},b'_h+b,b_{h+1},\dots,b_s}(z)
\vphantom{\left(\underset{j\ne h}{\prod _{j=1} ^{s}}
\left(\frac {z^{3^{k_j}}(1+z^{3^{k_j}})}
{1+z^{3^{k_j+1}}}\right)^{b_j}\right)}
\right.\\
\notag
&\kern7cm
+\widetilde H_{b_1,b_2,\dots,b_{h-1}+b'_h+b,b_{h+1},\dots,b_s}(z)
\\
\notag
&\left.
-
\sum _{k_1>\dots>k_{s-1}>k_s\ge 0} ^{}
\left(\underset{j\ne h}{\prod _{j=1} ^{s}}
\left(\frac {z^{3^{k_j}}(1+z^{3^{k_j}})}
{1+z^{3^{k_j+1}}}\right)^{b_j}\right)
\left(\frac {z^{3^{k_{h+1}+1}}(1+z^{3^{k_{h+1}+1}})}
{1+z^{3^{k_{h+1}+2}}}\right)^{b'_h+b}
\right)\\
\notag
&\kern1cm
+\underset{b\ge0}{\sum _{a\ge1} ^{}}
\sum _{s_1,s_2,\dots,t_1,t_2,\dots\ge0} ^{}
3^{a+b+\V s+\V t}\binom {b'_h}a\binom {-b'_h}b
\left(\prod _{j\ge1} ^{}\binom {s_{j-1}}{s_j}\binom
{-s_{j-1}}{t_j}\right)\\
&\kern5cm
\cdot
\widetilde H_{b_1,b_2,\dots,b_{h-1},3(\V
s+b'_h+b)+\V t+a,b_{h+1},\dots,b_s}(z).
\end{align*}
To the term with exponent $b'_h+b$ we apply the identity
\eqref{eq:k+1->k}. Then the above relation becomes
\begin{align}
\notag
\widetilde H&_{b_1,b_2,\dots,b_s}(z)
= \sum _{b\ge0} ^{}3^{b}\binom {-b'_h}b
\left(\widetilde H_{b_1,b_2,\dots,b_{h-1},b'_h+b,b_{h+1},\dots,b_s}(z)
\vphantom{{\prod _{j=1} ^{s}}}
\right.\\
\notag
&\kern7cm
+\widetilde H_{b_1,b_2,\dots,b_{h-1}+b'_h+b,b_{h+1},\dots,b_s}(z)
\\
\notag
&
-
\sum _{u_1,u_2,\dots,v_1,v_2,\dots\ge0} ^{}
3^{\V u+\V v}\left(\prod _{j\ge1} ^{}\binom {u_{j-1}}{u_j}\binom
{-u_{j-1}}{v_j}\right)\\
\notag
&\kern6.5cm
\left.\vphantom{{\prod _{j=1} ^{s}}}
\cdot \widetilde H_{b_1,b_2,\dots,b_{h-1},
3(b'_h+b+\V u)+\V v+b_{h+1},\dots,b_s}(z)
\right)\\
\notag
&\kern1cm
+\underset{b\ge0}{\sum _{a\ge1} ^{}}
\sum _{s_1,s_2,\dots,t_1,t_2,\dots\ge0} ^{}
3^{a+b+\V s+\V t}\binom {b'_h}a\binom {-b'_h}b
\left(\prod _{j\ge1} ^{}\binom {s_{j-1}}{s_j}\binom
{-s_{j-1}}{t_j}\right)\\
&\kern5cm
\cdot
\widetilde H_{b_1,b_2,\dots,b_{h-1},3(\V
s+b'_h+b)+\V t+a,b_{h+1},\dots,b_s}(z).
\label{eq:Rek1}
\end{align}
It should be observed that the quadruples associated with the
terms on the right-hand side of \eqref{eq:Rek1} 
are indeed less than the quadruple associated with
$\widetilde H_{b_1,b_2,\dots,b_s}(z)$
in the order $\prec$.

Next, we consider the case where $1<h=s$. Reasoning in a 
way similar to the case leading to \eqref{eq:Rek1},
we obtain the recurrence relation
\begin{align}
\notag
\widetilde H&_{b_1,b_2,\dots,b_s}(z)
= \sum _{b\ge0} ^{}3^{b}\binom {-b'_s}b
\bigg(\widetilde H_{b_1,b_2,\dots,b_{s-1},b'_s+b}(z)\\
\notag
&\kern4cm
+\widetilde H_{b_1,b_2,\dots,b_{s-1}+b'_s+b}(z)
-\left(\frac {z(1+z)} {1+z^3}\right)^{b'_s+b}
\widetilde H_{b_1,b_2,\dots,b_{s-1}}(z)\bigg)\\
\notag
&\kern1cm
+\underset{b\ge0}{\sum _{a\ge1} ^{}}
\sum _{s_1,s_2,\dots,t_1,t_2,\dots\ge0} ^{}
3^{a+b+\V s+\V t}\binom {b'_s}a\binom {-b'_s}b
\left(\prod _{j\ge1} ^{}\binom {s_{j-1}}{s_j}\binom
{-s_{j-1}}{t_j}\right)\\
&\kern5cm
\cdot
\widetilde H_{b_1,b_2,\dots,b_{s-1},3(\V
s+b'_s+b)+\V t+a}(z).
\label{eq:Rek2}
\end{align}
Again, the quadruples associated with the
terms on the right-hand side of \eqref{eq:Rek2} 
are less than the quadruple associated with
$\widetilde H_{b_1,b_2,\dots,b_s}(z)$
in the order $\prec$.

Now let $1=h<s$. In this case, we obtain the recurrence relation
\begin{align}
\notag
\widetilde H&_{b_1,b_2,\dots,b_s}(z)
= \sum _{b\ge0} ^{}3^{b}\binom {-b'_1}b
\left(\widetilde H_{b'_1+b,b_{2},\dots,b_s}(z)
\vphantom{{\prod _{j=1} ^{s}}}
\right.\\
\notag
&
-
\sum _{u_1,u_2,\dots,v_1,v_2,\dots\ge0} ^{}
3^{\V u+\V v}\left(\prod _{j\ge1} ^{}\binom {u_{j-1}}{u_j}\binom
{-u_{j-1}}{v_j}\right)
\left.\vphantom{{\prod _{j=1} ^{s}}}
\widetilde H_{3(b'_1+b+\V u)+\V v+b_{2},\dots,b_s}(z)
\right)\\
\notag
&\kern1cm
+\underset{b\ge0}{\sum _{a\ge1} ^{}}
\sum _{s_1,s_2,\dots,t_1,t_2,\dots\ge0} ^{}
3^{a+b+\V s+\V t}\binom {b'_1}a\binom {-b'_1}b
\left(\prod _{j\ge1} ^{}\binom {s_{j-1}}{s_j}\binom
{-s_{j-1}}{t_j}\right)\\
&\kern5cm
\cdot
\widetilde H_{3(\V s+b'_1+b)+\V t+a,b_{2},\dots,b_s}(z).
\label{eq:Rek3}
\end{align}

Finally, in the degenerate case $1=h=s$, summation over $k$
on both sides of \eqref{eq:id0} with $n=b'_1$ directly yields
\begin{multline} \label{eq:Rek4}
\widetilde H_{b_1}(z)
=\sum _{b\ge0} ^{}3^{b}\binom {-b'_1}b
\left(\widetilde H_{b'_1+b}(z)
-\left(\frac {z(1+z)}
{1+z^{3}}\right)^{b'_1+b}\right)\\
+\underset{b\ge0}{\sum _{a\ge1} ^{}}
\sum _{s_1,s_2,\dots,t_1,t_2,\dots\ge0} ^{}
3^{a+b+\V s+\V t}\binom {b'_1}a\binom {-b'_1}b
\left(\prod _{j\ge1} ^{}\binom {s_{j-1}}{s_j}\binom
{-s_{j-1}}{t_j}\right)\\
\cdot \widetilde H_{3(\V s+b'_1+b)+\V t+a}(z).
\end{multline}

It is clear that, if we recursively apply \eqref{eq:Rek1}--\eqref{eq:Rek4} to a
given series 
$\widetilde H_{b_1,b_2,\dots,b_s}(z)$, 
and use $\widetilde H_\emptyset=1$ as an initial condition,
we will eventually arrive at a linear combination of $1$,
powers of $z(1+z)/(1+z^3)$, and series
$\widetilde H_{a_1,a_2,\dots,a_r}(z)$
with all $a_i$'s being not divisible by~$3$, where the coefficients are 
polynomials in $z$ and $(1+z^3)^{-1}$. In view of Lemma~\ref{lem:1+z},
this completes the proof of the lemma.
\end{proof}


How the results of this section are implemented to extract coefficients
of concrete powers of $z$ from powers of $\Psi(z)$ modulo a given 
$3$-power is explained in more detail in the appendix.
As an illustration, the coefficients of $\Psi^3(z)$ and of
$\Psi^5(z)$ are worked out explicitly modulo~$9$ and modulo~$27$, see
Propositions~\ref{prop:Psi-3} and \ref{prop:Psi-5}.

\section{Minimal polynomials for $\Psi(z)$}
\label{sec:Psi}

Here we consider polynomial relations of 
the Cantor-like formal power series $\Psi(z)$ defined in
\eqref{eq:Psidef}, when coefficients are reduced modulo powers of~$3$.
The series is transcendental over
$\mathbb Z[z]$ (see Theorem~\ref{conj:1} below). However, if the
coefficients of $\Psi(z)$ are considered modulo a $3$-power $3^\ga$,
then $\Psi(z)$ obeys a polynomial relation with coefficients that are
Laurent polynomials in $z$ and $1+z$. 
The focus of this section is on what can be said 
concerning such polynomial relations, and, in particular, about those of
minimal length. 

\medskip
Here and in the sequel,
given integral power series (or Laurent series) $f(z)$ and $g(z)$,
we write 
$$f(z)=g(z)~\text {modulo}~3^\ga$$ 
to mean that the coefficients
of $z^i$ in $f(z)$ and $g(z)$ agree modulo~$3^\ga$ for all $i$.

We say that a polynomial $A(z,t)$ in 
$z$ and $t$ is {\it minimal for the modulus
$3^\ga$}, if it is monic (as a polynomial in $t$), 
has coefficients which are Laurent polynomials in $z$ and $1+z$, satisfies
$A(z,\Psi(z))=0$~modulo~$3^\ga$, and there is no monic polynomial $B(z,t)$
whose coefficients are Laurent polynomials in $z$ and $1+z$, whose
$t$-degree is less than that of $A(z,t)$, and which satisfies
$B(z,\Psi(z))=0$~modulo~$3^\ga$.
Minimal polynomials are not unique; see Remark~\ref{rem:min} below.

Our results concerning minimal polynomials are analogous to the
corresponding ones in \cite[Sec.~2]{KaKMAA}, albeit more difficult to prove.
We start by providing a lower bound on degrees of minimal polynomials.
We strongly believe that this lower bound is actually sharp; see
Conjecture~\ref{conj:1A} below.
In order to formulate the result,
we let $v_3(\al)$ denote the $3$-adic valuation of the
integer $\al$, that is, the maximal exponent $e$ such that $3^e$
divides $\al$.

\begin{theorem} \label{conj:1}
The degree of a minimal polynomial for the modulus $3^\ga$, $\ga\ge1$,
is $2d$, where $d$ satisfies
$d+v_3(d!)\ge\ga$. In particular, the series $\Psi(z)$ is transcendental
over $\Z[z]$.
\end{theorem} 

For the proof of this theorem, we need the following 
auxiliary result.

\begin{lemma} \label{lem:diff}
Let $a,b,c_1,c_2, \dots, c_d$ be positive integers with
$a>b$. Then
\begin{equation} \label{eq:diff} 
\sum_{k=0}^b (-1)^{b-k}\binom bk
\prod _{j=1} ^{a}3^{c_j}\binom k{c_j}
\end{equation}
is divisible by $3^a b!$. Moreover, 
if at least one of the $c_j$'s is at least~$3$, then
\eqref{eq:diff} is divisible by $3^{a+1}b!$.
\end{lemma}

\begin{proof}
We need to recall some facts from classical difference
operator calculus (cf.\ \cite{RoKOAA}). We denote by
$E$ the shift operator defined by $Ef(y)=f(y+1)$, and
by $\De$ the forward difference operator defined by $\De
f(y)=f(y+1)-f(y)$. The reader should observe that
$\De=E-I$, where $I$ stands for the identity
operator. In terms of these operators, we may rewrite
the expression in \eqref{eq:diff} as
\begin{align} \notag
\sum_{k=0}^b (-1)^{b-k}\binom bk
E^k\prod _{j=1} ^{a}3^{c_j}\binom y{c_j}\bigg\vert_{y=0}
&=(E-I)^b 
\prod _{j=1} ^{a}3^{c_j}\binom y{c_j}\bigg\vert_{y=0}\\
&=\left(
\prod _{j=1} ^{a}\frac {(-3)^{c_j}} {{c_j}!}\right)\De^b 
\prod _{j=1} ^{a}(-y)_{c_j}\bigg\vert_{y=0},
\label{eq:E-b}
\end{align}
where the Pochhammer symbol is defined by
$(\alpha)_m=\alpha(\alpha+1) \cdots(\alpha+m-1)$ for $m\ge1$, and
$(\alpha)_0=1$,
The second product over $j$ is a polynomial in $y$ with integer
coefficients. It is well-known from difference calculus that
the constant term of the $b$-fold application of the operator $\De$ 
on such a polynomial 
is divisible by $b!$. For the sake of self-containedness, we 
provide the simple argument here: let $p(y)$ be a polynomial in~$y$
with integer coefficients. Then we may expand $p(y)$ in Pochhammer
symbols,
$$
p(y)=\sum_{m\ge0}\al_m (-y)_m,
$$
where the coefficients $\al_m$ are all integers. Since
$E(-y)_m=-m\cdot(-y)_{m-1}$, we infer
$$
E^bp(y)\Big\vert_{y=0}
=(-1)^b\sum_{m\ge0}\al_m \frac {m!} {(m-b)!}(-y)_{m-b}\Big\vert_{y=0}
=(-1)^b\al_b b!,
$$
which is visibly divisible by $b!$.

If we use what we know in \eqref{eq:E-b}, together with the
easily verified fact that 
$$v_3(3^\al)\ge \begin{cases} v_3(\al!)+1,&\text{if $\al\ge1$,}\\
v_3(\al!)+2,&\text{if $\al\ge3$,}\\
\end{cases}
$$
then the assertion of the lemma follows immediately.
\end{proof}

\begin{proof}[Proof of Theorem~\em\ref{conj:1}]
The first observation is that Lemma~\ref{lem:evenodd}
implies that minimal polynomials must have even degree, and,
moreover, that they can be chosen as even polynomials in $t$, that is,
as polynomials in $t^2$.

So, for a contradiction, 
let us suppose that $A(z,t)$ is the minimal polynomial for the
modulus $3^\ga$, that it is even, 
and that it has degree $2d$, where $d+v_3(d!)<\ga$.

Since $A(z,t)$ is even in~$t$. 
it does not matter whether we consider $A(z,t)$ as a polynomial
in $t^2$ (and~$z$ and $(1+z)^{-1}$) 
or as a polynomial in $\left(t^2-\frac {1} {1+z}\right)$
(and~$z$ and $(1+z)^{-1}$). We take the latter point of view.

We now expand a power $\left(\Psi^{2}(z)-\frac {1} {1+z}\right)^b$,
$1\le b\le d$, in terms
of the series $\widetilde H_{a_1,a_2,\dots,a_s}(z)$ (see
\eqref{eq:tHdef} for the definition) using \eqref{eq:Psi-2A}
and the short notation \eqref{eq:R}{:}
\begin{align} \notag
\left(\Psi^{2}(z)-\frac {1} {1+z}\right)^b
&=\frac {1} {(1+z)^b}\left(\sum_{s\ge1}
3^s\widetilde H_{\underbrace{\scriptstyle 1,1,\dots,1}_{s\text{
      times}}}(z)\right) ^b\\
\notag
&=\frac {1} {(1+z)^b}\left(\sum_{s\ge1}
3^s\sum_{k_1>\dots>k_{s}\ge0}
R(k_1)R(k_2)\cdots R(k_{s})
\right) ^b\\
\notag
&=\frac {1} {(1+z)^b}\left(\bigg(\prod_{j\ge0}
(1+3R(j))\bigg)-1
\right) ^b\\
\notag
&=\frac {1} {(1+z)^b}
\sum_{k=0}^b (-1)^{b-k}\binom bk
\prod _{j\ge0} ^{}
(1+3R(j))^k\\
&=\frac {1} {(1+z)^b}\sum _{s\ge1}\sum_{a_1,\dots,a_s\ge1}
c(b;a_1,a_2,\dots,a_s)\,
\widetilde H_{a_1,a_2,\dots,a_s}(z),
\label{eq:Psi2d} 
\end{align}
where the $c(b;a_1,a_2,\dots,a_s)$'s are certain combinatorial
(and, hence, integer) coefficients. 
Not all of the $\widetilde H$-series which occur
with non-zero coefficient $c(b;a_1,a_2,\dots,a_s)$ will have the
property that their indices $a_j$ are not divisible by $3$.
So, in a second step, one will have to apply the reduction procedure
described in Lemma~\ref{lem:Hbi} to arrive at an expansion of the form
\begin{equation} \label{eq:Hbexp2} 
\left(\Psi^{2}(z)-\frac {1} {1+z}\right)^b
=\frac {1} {(1+z)^b}
\sum _{s\ge1}\underset{3\nmid a_1,\dots,3\nmid a_s}{\sum_{a_1,\dots,a_s\ge1}}
\widetilde c(b;a_1,a_2,\dots,a_s)
\widetilde H_{a_1,a_2,\dots,a_s}(z),
\end{equation}
where the $\widetilde c(b;a_1,a_2,\dots,a_s)$'s are certain modified
coefficients. It should be observed
(cf.\ Corollary~\ref{cor:Htildeind}) that the 
$\widetilde H$-series on the right-hand
side of \eqref{eq:Hbexp2} are linearly independent over 
$(\Z/3^\ga\Z)[z,(1+z)^{-1}]$.

We are interested in the coefficient $\widetilde c(b;1,1,\dots,1)$
with $a$ occurrences of $1$. We claim that
\begin{equation} \label{eq:coef111} 
\widetilde c(b;\underbrace {1,1,\dots,1}_{a\text{ times}})=
\begin{cases} 
3^{a+1}\,b!\,N_1,&\text{if $a<b$,}\\
3^b\,b!+3^{b+v_3(b!)+1}N_2,&\text{if $a=b$,}\\
3^a\,b!\,N_3,&\text{if $a>b$,}
\end{cases}
\end{equation}
where $N_1,N_2,N_3$ are integers (depending on $a$ and $b$).

In order to see this, we must examine how terms $\widetilde
H_{1,\dots,1}(z)$ with a non-zero coefficient $\widetilde
c(b;1,1,\dots,1)$ in the expansion \eqref{eq:Hbexp2} arise from
\eqref{eq:Psi2d}.
From the proof of Lemma~\ref{lem:Hbi}, we see that such terms
come from terms
$$
c(b;3^{d_1},3^{d_2},\dots,3^{d_a})\,
\widetilde H_{3^{d_1},3^{d_2},\dots,3^{d_a}}(z),
$$
in \eqref{eq:Psi2d} with non-zero coefficients $c(\dots)$.
(To be correct, there are also contributions to $\widetilde
c(b;1,1,\dots,1)$ from other terms, but these will have a
higher $3$-divisibility than those discussed here.
We shall not mention them any further. Suffice it to point out that
they are subsumed in the terms $N_4(b;d_1,d_2,\dots,d_a)$ below.)
If these observations are put together, one obtains 
\begin{equation} \label{eq:tcc} 
\widetilde c(b;\underbrace{1,1,\dots,1}_{a\text{ times}})
=\sum_{d_1,\dots,d_a\ge0}
c(b;3^{d_1},3^{d_2},\dots,3^{d_a})\,
\big(1+3N_4(b;d_1,d_2,\dots,d_a)\big),
\end{equation}
where the $N_4(b;d_1,d_2,\dots,d_a)$'s are integers.

We must now compute the coefficient
$c(b;3^{d_1},3^{d_2},\dots,3^{d_a})$. 
As we already said, this is a combinatorial coefficient arising
in the expansion \eqref{eq:Psi2d}. Recalling the definition 
\eqref{eq:tHdef} of the $\widetilde H$-series, 
one sees that this coefficient equals the
coefficient of 
$$
R(0)^{3^{d_1}}
R(1)^{3^{d_2}}
\cdots
R(a-1)^{3^{d_a}}
$$
in the line just above \eqref{eq:Psi2d}. An easy extraction yields
\begin{equation*} 
c(b;3^{d_1},3^{d_2},\dots,3^{d_a})
=\sum_{k=0}^b (-1)^{b-k}\binom bk
\prod _{j=1}^a3^{3^{d_j}}\binom k{3^{d_j}}.
\end{equation*}
By Lemma~\ref{lem:diff} we know that this expression is divisible
by $3^a\,b!$, and, if one of the $d_j$'s is positive, even 
by $3^{a+1}\,b!$. If this is used in \eqref{eq:tcc}, then
our claim \eqref{eq:coef111} follows without much effort, 
the term $3^b\,b!$ (occurring in \eqref{eq:coef111} 
in the case where $a=b$) being contributed by
$$
c(b;\underbrace{1,1,\dots,1}_{b\text{ times}})=3^b\,b!.
$$

We are now in the position to complete the argument.
Let us return to the minimal polynomial $A(z,t)$ for the modulus
$3^\ga$. We expand it in powers of $\left(t^2-\frac {1} {1+z}\right)$,
\begin{equation} \label{eq:A(z,t)} 
A(z,t)=
\sum_{b=0}^{d}A_b\left(t^2-\frac {1} {1+z}\right)^b,
\end{equation}
where $A_d=1$.
(The reader should recall that minimal polynomials are monic by
definition.)
We substitute $\Psi(z)$ for $t$ and claim that
$A(z,\Psi(z))$ cannot vanish modulo $3^\ga$. 
In order to see this, we perform
an induction on the index $b$ in \eqref{eq:A(z,t)}.
%
The claim that we are going to prove by this induction is 
that the coefficient
$A_a$ is divisible by $3^{b+v_3(b!)-a-v_3(a!)}$ for all $a$ with 
$1\le a\le b$.
This is trivially true for $b=1$, since in that case the only
possible $a$ is $a=1$, in which case the condition is void.

We proceed to the induction step.
In the following, we shall use the short notation $\widetilde
H_{1^r}(z)$ for $\widetilde
H_{1,1,\dots,1}(z)$ with $r$ occurrences of $1$.
We consider the coefficient of $\widetilde H_{1^b}(z)$
in $\left(t^2-\frac {1} {1+z}\right)^{b}$. Let us denote it by $C$. 
According to
\eqref{eq:coef111}, it equals $3^{b}\,b!+3^{b+v_3(b!)+1}N_2$, where
$N_2$ is some integer. In particular, it is divisible by $3^{b+v_3(b!)}$,
but not by a higher power of~$3$. 
By using \eqref{eq:coef111} again (this time with $a$ and $b$
interchanged), in combination with
the induction hypothesis, we see that the coefficient of 
$\widetilde H_{1^b}(z)$
in the term $A_a\left(\Psi^2(z)-\frac {1} {1+z}\right)^{a}$, $a<b$, is
divisible by
$3^{b+v_3(a!)}3^{b+v_3(b!)-a-v_3(a!)}=3^{2b-a+v_3(b!)}\ge 3^{b+v_3(b!)+1}$. 
This is a higher divisibility by~$3$ than that of the coefficient $C$. 
Moreover, by another application of
\eqref{eq:coef111}, the coefficient of 
$\widetilde H_{1^b}(z)$
in the term $A_a\left(\Psi^2(z)-\frac {1} {1+z}\right)^{a}$, $a>b$, is
divisible by 
\begin{equation} \label{eq:3b+1} 
3^{b+1+v_3(a!)}\ge 3^{b+v_3((b+1)!)+1}\ge
3^{b+v_3(b!)+1},
\end{equation}
which is again a
higher $3$-divisibility than that of the coefficient~$C$.
So none of these terms can cancel $C$, unless $A_b$ is divisible
by~$3$. Assuming the latter, one considers the coefficient of $\widetilde
H_{1^{b-1}}$ in $A_{b-1}\left(\Psi^2(z)-\frac {1} {1+z}\right)^{b-1}$
and compares its $3$-divisibility to the $3$-divisibility of 
the corresponding coefficients in the other terms. The conclusion
is that also $A_{b-1}$ must be divisible by another $3$, that is
(the reader should recall the induction hypothesis),
it must be divisible by $3^{b+v_3(b!)-(b-1)-v_3((b-1)!)+1}$.
One repeats this same argument with the coefficient of $\widetilde
H_{1^{a}}$ in $A_{a}\left(\Psi^2(z)-\frac {1} {1+z}\right)^{a}$,
$a=b-2,\dots,2,1$, to conclude that $A_a$ must be divisible by
$3^{b+v_3(b!)-a-v_3(a!)+1}$, $a=1,2,\dots,b$. 

Now one returns to the term
$A_{b}\left(\Psi^2(z)-\frac {1} {1+z}\right)^{b}$.
So far, we know that $A_b$ must be divisible by~$3$, and that
the coefficient of $\widetilde
H_{1^{b}}$ in $\left(\Psi^2(z)-\frac {1} {1+z}\right)^{b}$ is
divisible by $3^{b+v_3(b!)}$, but not by a higher power of~$3$.
If we repeat the arguments of the previous paragraph, but take
into account that they led to an ``improvement" of the induction
hypothesis, then one is led to infer that $A_b$ must
have an even higher $3$-divisibility, entailing analogous
higher $3$-divisibility of the coefficients $A_a$ with $a<b$,
until one reaches the conclusion that $A_b$ must be divisible
by $3^{b+1+v_3((b+1)!)-b-v_3(b!)}$ (cf.\ the middle term in
\eqref{eq:3b+1}), and consequently $A_a$ by
$3^{b+1+v_3((b+1)!)-a-v_3(a!)}$, $a=1,2,\dots,b-1$.
This is exactly the induction hypothesis with $b$ replaced by $b+1$.

When we use what we have just proven for $b=d$, then 
a contradiction arises. Namely, if $A_a$ is divisible by
$3^{d+v_3(d!)-a-v_3(a!)}$, $a=1,2,\dots,d$, then the above
arguments would imply that $A_d$ must be divisible by (at least)~$3$.
This, however, is impossible since $A_d=1$.

We have thus shown that, indeed, a minimal polynomial for the
modulus $3^\ga$, where $d+v_3(d!)<\ga$ must have degree larger than
$2d$. This completes the proof of the theorem.
\end{proof}

Based on the observations in Proposition~\ref{prop:minpol}
later in this section, as well as 
Theorem~\ref{conj:1}, we propose the following conjecture.

\begin{conjecture} \label{conj:1A}
The degree of a minimal polynomial for the modulus $3^\ga$, $\ga\ge1$,
is $2d$, where $d$ is the least positive integer such that 
$d+v_3(d!)\ge\ga$. 
\end{conjecture} 

\begin{remarknu} \label{rem:1}
(1) Given the ternary expansion of $d$, say
$$d=d_0+d_1\cdot 3+d_2\cdot
9+\cdots+d_r\cdot 3^r, \quad 0\le \ga_i\le1,$$ 
we have, by the well-known formula of
Legendre \cite[p.~10]{LegeAA}, 
\begin{align}\notag
v_3(d!)&=\sum _{\ell=1} ^{\infty}\fl{\frac {d} {3^\ell}}=
\sum _{\ell=1} ^{\infty}\fl{\sum _{i=0} ^{r}d_i 3^{i-\ell}}=
\sum _{\ell=1} ^{\infty}\sum _{i=\ell} ^{r}d_i 3^{i-\ell}\\
\label{eq:Leg}
&=
\sum _{i=1} ^{r}\sum _{\ell=1} ^{i}d_i 3^{i-\ell}=
\sum _{i=1} ^{r}\frac {d_i} {2}\left(3^{i}-1\right)=
\frac {1} {2}(d-s(d)),
\end{align}
where $s(d)$ denotes the sum of digits of $d$ in its ternary
expansion.
Consequently, an equivalent way of phrasing Theorem~\ref{conj:1}
is to say that the degree of a minimal\break polynomial for the 
modulus $3^\ga$ is $2d$, where $d$ is the least 
positive integer with\break $\frac {1} {2}(3d-s(d))\ge\ga$.

\medskip
(2)
We claim that, in order to establish Conjecture~\ref{conj:1A}, 
it suffices to prove the conjecture
for $\ga=\frac {1}
{2}(3^{\de+1}-1)$, $\de=0,1,\dots$. If we take into account
the remark in (1), this means that it is
sufficient to prove that, for each $\de\ge0$, there is
a polynomial $A_\de(z,t)$ of degree $2\cdot 3^{\de}$ such that
\begin{equation} \label{eq:min2ga}
A_\de(z,\Psi(z))=0\quad \text {modulo }3^{(3^{\de+1}-1)/2}.
\end{equation}
For, arguing by induction, let us suppose that 
we have already constructed $A_0(z,t),\break
A_1(z,t),
\dots,A_{m}(z,t)$ satisfying \eqref{eq:min2ga}.
Let
$$\al=\al_0+\al_1\cdot 3+\al_2\cdot
9+\cdots+\al_{m}\cdot 3^{m}, \quad 0\le \al_i\le2,$$ 
be the ternary expansion of the positive integer
$\al$. In this situation, we have
\begin{equation} \label{eq:betaind}
\prod _{\de=0} ^{m}A_{\de}^{\al_{\de}}(z,\Psi(z))=0\quad \text {modulo }
\prod _{\de=0} ^{m}3^{\al_{\de}(3^{\de+1}-1)/2}=
3^{(3\al-s(\al))/2}.
\end{equation}
On the other hand, the degree of the left-hand side of
\eqref{eq:betaind} as a polynomial in $\Psi(z)$ is
$\sum _{\de=0} ^{m}\al_{\de} \cdot2\cdot3^\de=2\al$.

Let us put these observations together.
In view of \eqref{eq:Leg},
our lower bound assumption says that the degree of a minimal
polynomial for the modulus $3^\ga$ cannot be smaller than the
least even integer, $2d^{(\ga)}$ say, 
for which $\frac {1} {2}(3d^{(\ga)}-s(d^{(\ga)}))\ge\ga$. 
If we take into account that
the quantity $\frac {1} {2}(3\al-s(\al))$, as a function in $\al$, is weakly monotone
increasing in $\al$, then \eqref{eq:betaind} tells us that,
as long as $d^{(\ga)}\le 2+2\cdot3+2\cdot9+\dots+2\cdot3^{m}
=3^{m+1}-1$, 
we have found a monic
polynomial of degree $2d^{(\ga)}$, $B_\ga(z,t)$ say, for which
$B_\ga(z,\Psi(z))=0$ modulo~$3^\ga$, namely the left-hand side
of \eqref{eq:betaind} with $\al$ replaced by $d^{(\ga)}$, to wit
$$
B_\ga(z,t)=\prod _{\de=1} ^{m}A_{\de}^{d^{(\ga)}_\de}(z,t),
$$
where $d^{(\ga)}=d^{(\ga)}_0+d^{(\ga)}_1\cdot 3+d^{(\ga)}_2\cdot
9+\cdots+d^{(\ga)}_m\cdot 3^m$ is the ternary expansion of 
$d^{(\ga)}$. Hence, it must necessarily be a minimal polynomial
for the modulus $3^\ga$.

Since $\frac {1} {2}\big(3(3^{m+1}-1)-s(3^{m+1}-1)\big)
=\frac {1} {2}(3^{m+2}-3-2(m+1))$, we have thus found minimal
polynomials for all moduli $3^\ga$ with $\ga\le \frac {1}
{2}(3^{m+2}-2m-5)$. 
Now we should note that the
quantity $\frac {1} {2}(3\al-s(\al))$ makes a jump from 
$\frac {1}{2}(3^{m+2}-2m-5)$ to $\frac {1} {2}(3^{m+2}-1)$
when we move from $\al=3^{m+1}-1$ to $\al=3^{m+1}$.
If we take $A_m^3(z,t)$, which has degree $2\cdot 3^{m+1}$,
then, by \eqref{eq:min2ga}, 
we also have a minimal polynomial for the modulus 
$\left(3^{(3^{m+1}-1)/2}\right)^3=3^{(3^{m+2}-3)/2}$
and, in view of the preceding remark, as well for all moduli $3^\ga$
with $\ga$ between 
$\frac {1}{2}(3^{m+2}-2m-5)$ and $\frac {1} {2}(3^{m+2}-3)$.

So, indeed, the first modulus for which we do not have a minimal
polynomial is the modulus $3^{(3^{m+2}-1)/2}$. This is the role
which $A_{m+1}(z,t)$ (see \eqref{eq:min2ga} with $m+1$ in place of 
$\de$) would have to play.

\medskip
(3)
The arguments in Item~(2) show at the same time that, supposing that we have
already constructed $A_0(z,t),A_1(z,t),\dots,A_m(z,t)$, the polynomial 
$A^3_m(z,t)$ is a very close ``approximation" to the polynomial
$A_{m+1}(z,t)$ that we are actually looking for next,
which is only ``off" by a factor of~$3$.
This can be used to recursively compute polynomials $A_\de(z,t)$
satisfying \eqref{eq:min2ga}, by expanding $A_m(z,\Psi(z))$ in terms
of the series $\widetilde H_{a_1,a_2,\dots,a_s}(z)$ with all
$a_i$'s not divisible by~$3$, as discussed
in the previous section (see in particular \eqref{eq:Psi-2A} and
Lemma~\ref{lem:Hbi}), dividing the result through by
$3^{(3^{m+1}-1)/2}$, and considering the quotient, say
$$
3^{-(3^{m+1}-1)/2}A_m(z,\Psi(z))=\sum_{s\ge1}\sum_{a_1,\dots,a_s\ge1}
c(a_1,a_2,\dots,a_s;z)
\widetilde H_{a_1,a_2,\dots,a_s}(z),
$$
where the coefficients $c(a_1,a_2,\dots,a_s;z)$ are
necessarily in $\Z[z,(1+z)^{-1}]$, modulo~$3$. Then one computes
\begin{multline} \label{eq:Am2}
\left(3^{-(3^{m+1}-1)/2}A_m(z,\Psi(z))\right)^3
=\sum_{s\ge1}\sum_{a_1,\dots,a_s\ge1}
c_{a_1,a_2,\dots,a_s}^3(z)
\widetilde H_{a_1,a_2,\dots,a_s}^3(z)
\quad \text{modulo 3},
\end{multline}
uses Lemma~\ref{lem:Hbi} again to perform the expansion
$$
\left(\widetilde H_{a_1,a_2,\dots,a_s}(z)\right)^3=
\widetilde H_{a_1,a_2,\dots,a_s}(z)+\text{other terms}
\quad \text{modulo 3},
$$
substitutes this back into \eqref{eq:Am2}, and expresses whatever
remains in terms of powers of $A_i(z,\Psi(z))$'s, with $i=0,1,\dots,m$.
We are not able to actually show that this always works, but
it does for small $m$. The corresponding results are listed in
Proposition~\ref{prop:minpol} below.
\end{remarknu}

\begin{proposition} \label{prop:minpol}
Minimal polynomials for the moduli
$3,9,27,81,243,729,2187,\dots,3^{13}$ are
\begin{alignat*}2 
&\,\ A_0(z,t):=t^2-\tfrac {1} {1+z}&&\text {\em modulo 3},\\
&
A_0^2(z,t)
&&\text {\em modulo 9},\\
&
A_0^3(z,t)
&&\text {\em modulo 27},\\
&
A_1(z,t):=\left(t^2-\tfrac {1} {1+z}\right)^3-\tfrac {9} {(1+z)^2}
\left(t^2-\tfrac {1} {1+z}\right)+\tfrac {27z} {(1+z)^5}
\quad 
&&\text {\em modulo 81},\\
&
A_0(z,t)A_1(z,t)
&&\text {\em modulo 243},\\
&
A_0^2(z,t)A_1(z,t)
&&\text {\em modulo 729},\\
&
A_1^2(z,t)
&&\text {\em modulo 2189},\\
&
A_1^2(z,t)
&&\text {\em modulo $3^8$},\\
&
A_0(z,t)A_1^2(z,t)
&&\text {\em modulo $3^9$},\\
&
A_0^2(z,t)A_1^2(z,t)
&&\text {\em modulo $3^{10}$},\\
&
A_1^3(z,t)
&&\text {\em modulo $3^{11}$},\\
&
A_1^3(z,t)
&&\text {\em modulo $3^{12}$},\\
&
A_2(z,t):=A_1^3(z,t)-\tfrac {3^8} {(1+z)^6}A_1(z,t)\\
&\kern2cm
+\tfrac {3^{10}z} {(1+z)^{9}}A_0^2(z,t)
-\tfrac {3^{11}z(1+z^2)} {(1+z)^{12}}A_0(z,t)
+\tfrac {3^{12}z^4} {(1+z)^{17}}\quad 
&&\text {\em modulo $3^{13}$}
.
\end{alignat*}
\end{proposition}

\begin{proof}
To begin with, the reader should recall from Theorem~\ref{conj:1}
that minimal polynomials always have even degree.

Then, considering \eqref{eq:Psi-2A}, we see that $A_0(z,t)$ is indeed a
minimal polynomial for the modulus~$3$. 

Continuing, we follow the procedure outlined in Item~(3) of
Remark~\ref{rem:1}. By \eqref{eq:Psi-2A}, we have
\begin{equation} \label{eq:A0} 
3^{-1}A_0(z,\Psi(z))=\frac {1} {1+z}\widetilde H_1(z)
\quad \text{modulo 3},
\end{equation}
whence
\begin{align*}
3^{-3}A_0^3(z,\Psi(z))&=\frac {1} {(1+z)^3}\widetilde H_1^3(z)
\quad \text{modulo 3}\\
&=\frac {1} {(1+z)^3}\left(\widetilde H_1(z)-\frac {z(1+z)} {1+z^3}\right)
\quad \text{modulo 3}.
\end{align*}
Putting this together with \eqref{eq:A0} and the simple observation that
$$
\frac {z(1+z)} {1+z^3}=\frac {z} {(1+z)^2}
\quad \text{modulo 3},
$$ 
it follows immediately that 
$A_1(z,\Psi(z))$ vanishes indeed modulo~$3^4$.
The fact that $A_1(z,t)$ is {\it minimal\/} for the modulus $3^4$
follows from the degree bound in Theorem~\ref{conj:1}.
This bound also
implies the assertions about minimal polynomials for the moduli
$9,27,3^5,3^6,\dots,3^{12}$.

The reasoning for $A_2(z,t)$ is analogous, albeit more laborious.
We leave the details to the reader. 
\end{proof}

\begin{remarknu}\label{rem:min}
Minimal polynomials are highly non-unique: for example, 
the polynomial
$$
\left(t^2-\frac {1} {1+z}\right)^2+3\left(t^2-\frac {1} {1+z}\right)
$$
is obviously also a minimal polynomial for the modulus $9$.
\end{remarknu}

In the remainder of this section, we show that the derivative of our
basic series $\Psi(z)$, when reduced modulo a given power of~$3$, 
can be expressed as a multiple of $\Psi(z)$,
where the multiplicative factor is a polynomial in $z$ and
$(1+z)^{-1}$. This is one of the crucial facts which
make the method described in Section~\ref{sec:method} work.

By straightforward differentiation, we obtain
\begin{equation} \label{eq:Psi'} 
\Psi'(z)
= \Psi(z)\sum_{j=0}^{\infty} \frac {3^j 
z^{3^j-1}} {1+z^{3^j}}.
\end{equation}
Modulo a
given $3$-power $3^\be$, the sum on the right-hand side of \eqref{eq:Psi'}
is finite. Hence, by Lemma~\ref{lem:1+z}, our claim follows.

\section{The method}
\label{sec:method}

We consider a (formal) differential equation of the form
\eqref{eq:diffeq}, satisfied by the integral power series $F(z)$,
In this
situation, we propose the following algorithmic approach to
determining $F(z)$ modulo a $3$-power $3^{3^\al}$,
for some positive integer $\al$. 
Let us fix $\ep\in\{1,-1\}$ and a positive integer $\ga$. 
We make the Ansatz
\begin{equation} \label{eq:Ansatz}
F(z)=\sum _{i=0} ^{2\cdot 3^{\al}-1}a_i(z)\Psi^i(\ep z^\ga)\quad 
\text {modulo }3^{3^\al},
\end{equation}
where $\Psi(z)$ is given by \eqref{eq:Psidef},
and where the $a_i(z)$'s
are (at this point) undetermined elements of
$\Z[z,z^{-1},(1+\ep z^\ga)^{-1}]$.
Now we substitute \eqref{eq:Ansatz} into \eqref{eq:diffeq}, and
we shall gradually determine approximations $a_{i,\be}(z)$ to $a_i(z)$ such that
\eqref{eq:diffeq} holds modulo $3^\be$, for $\be=1,2,\dots,3^\al$. 
To start the procedure, we consider the differential equation
\eqref{eq:diffeq} modulo~$3$, with
\begin{equation} \label{eq:Ansatz1}
F(z)=\sum _{i=0} ^{2\cdot 3^{\al}-1}a_{i,1}(z)\Psi^i(\ep z^\ga)\quad \text {modulo
}3.
\end{equation}
We have 
$$\frac {d} {dz}\left(\Psi(\ep z^\ga)\right)
=\frac {\ep\ga z^{\ga-1}} {1+\ep z^\ga}\Psi(\ep z^\ga)
\quad \text {modulo }3.$$ 
Consequently, we see that
the left-hand side of \eqref{eq:diffeq} is a polynomial in
$\Psi(\ep z^\ga)$ with coefficients in 
$\Z[z,z^{-1},(1+\ep z^\ga)^{-1}]$. 
We reduce powers $\Psi^k(\ep z^\ga)$ with $k\ge2\cdot 3^{\al}$ using
the relation 
\begin{equation} \label{eq:PsiRel}
\left(\Psi^{2}(\ep z^\ga)-\frac {1} {1+\ep z^\ga}\right)^{3^\al}=0\quad \text {modulo
}3^{3^\al},
\end{equation}
which is implied by the minimal polynomial for the
modulus~$3$ given in Proposition~\ref{prop:minpol}.\footnote{Actually, 
if we would like to obtain an optimal result, 
we should use the relation implied by
the minimal polynomial for the modulus $3^{3^\al}$ in the sense of 
Section~\ref{sec:Psi}. However, since we have no general formula available
for such a minimal polynomial (cf.\ Items~(2) and (3) of 
Remark~\ref{rem:1} in that section),
and since we wish to prove results for arbitrary moduli, we choose instead
powers of the minimal polynomial for the modulus~$3$ as the best compromise.
In principle, it may happen that there exists a polynomial in
$\Psi(\ep z^\ga)$
with coefficients in $\Z[z,z^{-1},(1+\ep z^\ga)^{-1}]$, which is identical
with $F(z)$ after reduction of its coefficients modulo $3^{3^\al}$,
but the Ansatz \eqref{eq:Ansatz} combined with the
reduction \eqref{eq:PsiRel} fails because it is too restrictive. However,
in all the cases that we treat, this
obstruction does not occur. Moreover, once we are successful using this
(potentially problematic) Ansatz, then the result can be easily converted into
an optimal one by reducing the obtained polynomial further using
the relation implied by the actual minimal polynomial for the modulus
$3^{3^\al}$.}
Since, at this point, we are only interested in finding a solution to
\eqref{eq:diffeq} modulo~$3$, Relation~\eqref{eq:PsiRel} simplifies to
\begin{equation} \label{eq:PsiRel2}
\Psi^{2\cdot 3^{\al}}(\ep z^\ga)
-\frac {1} {(1+\ep z^\ga)^{3^{\al}}}=0\quad \text {modulo
}3.
\end{equation}
Now we compare coefficients of powers $\Psi^k(\ep z^\ga)$,
$k=0,1,\dots,2\cdot 3^{\al}-1$. This yields a
system of $2\cdot 3^{\al}$ (differential) equations (modulo~$3$)
for unknown functions $a_{i,1}(z)$ in
$\Z[z,z^{-1},(1+\ep z^\ga)^{-1}]$, $i=0,1,\dots,
2\cdot 3^{\al}-1$,
which may or may not have a solution. 

Provided we have already found 
functions $a_{i,\be}(z)$ in
$\Z[z,z^{-1},(1+\ep z^\ga)^{-1}]$, $i=0,1,\dots,\break
2\cdot 3^{\al}-1$, 
for some $\be$ with $1\le \be\le 3^{\al}-1$, such that
\begin{equation} \label{eq:Ansatz2}
F(z)=\sum _{i=0} ^{2\cdot 3^{\al}-1}a_{i,\be}(z)\Psi^i(\ep z^\ga)
\end{equation}
solves \eqref{eq:diffeq} modulo~$3^\be$, we put 
\begin{equation} \label{eq:Ansatz2a}
a_{i,\be+1}(z):=a_{i,\be}(z)+3^{\be}b_{i,\be+1}(z),\quad 
i=0,1,\dots,2\cdot 3^{\al}-1,
\end{equation}
where the $b_{i,\be+1}(z)$'s are (at this point) undetermined 
elements in $\Z[z,z^{-1},(1+\ep z^\ga)^{-1}]$. Next we substitute
$$
F(z)=\sum _{i=0} ^{2\cdot 3^{\al}-1}a_{i,\be+1}(z)\Psi^i(\ep z^\ga)
$$
in \eqref{eq:diffeq}. Using \eqref{eq:Psi'} and Lemma~\ref{lem:1+z},
we see that derivatives of $\Psi(\ep z^\ga)$ can be expressed as multiples of
$\Psi(z)$, where multiplicative factors are polynomials in $z$ and
$(1+\ep z^\ga)^{-1}$. Consequently, we may
expand the left-hand side of \eqref{eq:diffeq} 
as a polynomial in $\Psi(\ep z^\ga)$ with
coefficients in $\Z[z,z^{-1},(1+\ep z^\ga)^{-1}]$. Subsequently,
we apply again the reduction
using relation \eqref{eq:PsiRel}. By comparing coefficients of powers
$\Psi^k(\ep z^\ga)$, $k=0,1,\dots,2\cdot 3^{\al}-1$,
we obtain a
system of $2\cdot 3^{\al}$ (differential) equations (modulo~$3^{\be+1}$)
for the unknown functions $b_{i,\be+1}(z)$,
$i=0,1,\dots,2\cdot 3^{\al}-1$,
which may or may not have a solution. If we manage to push this 
procedure through until $\be=3^\al-1$, then,
setting $a_i(z)=a_{i,3^\al}(z)$, $i=0,1,\dots,2\cdot 3^{\al}-1$,
the series $F(z)$ as given in \eqref{eq:Ansatz} is a solution to
\eqref{eq:diffeq} modulo~$3^{3^\al}$, as required.

We should point out that, in order for the method to be meaningful,
we need to assume that the differential equation \eqref{eq:diffeq},
when considered modulo an arbitrary $3$-power $3^e$,
determines the solution $F(z)$ {\it uniquely} modulo~$3^e$.
Otherwise, our method could produce several different ``solutions,"
and it might be difficult to decide which of them would actually 
correspond to the series $F(z)$ we are interested in.
Concerning this point, see also the discussion in Remark~\ref{rem:2}
in Section~\ref{sec:gen}, as well as 
Section~\ref{sec:Apery}.

\section{A sample application}
\label{sec:sample}

In this section, we demonstrate how to apply the method from the
previous section, by considering the sequence of
``almost central" binomial coefficients $\left(\binom
{2n+2}n\right)_{n\ge0}$. To get started, we need a
functional equation for the generating function
$A(z)=\sum_{n\ge0}\binom {2n+2}nz^n$.
It is not difficult to see that $A(z)$ satisfies the
equation 
\begin{equation} \label{eq:AEF}
z^2A^2(z)+A(z)-\frac {1} {1-4z}=0.
\end{equation}
We fix $\al$,
choose $\ep=-1$ and $\ga=1$ in the method from Section~\ref{sec:method},
and make the Ansatz 
\begin{equation*} 
A(z)=\sum _{i=0} ^{2\cdot 3^{\al}-1}a_i(z)\Psi^i(-z)\quad 
\text {modulo }3^{3^\al},
\end{equation*}
where $\Psi(z)$ is given by \eqref{eq:Psidef},
and where the $a_i(z)$'s
are (at this point) undetermined Laurent polynomials in $z$ and
$1-z$.

The reader should recall that the idea is to gradually determine
approximations $a_{i,\be}(z)$ to $a_i(z)$ such that
\begin{equation} \label{eq:Abe} 
A_\be(z):=\sum _{i=0} ^{2\cdot 3^{\al}-1}a_{i,\be}(z)\Psi^i(-z)
\end{equation}
solves \eqref{eq:AEF} modulo $3^\be$, for $\be=1,2,\dots,3^\al$. 

For the ``base step" (that is, for $\be=1$), we claim that
$$
A_1(z)=\frac {1} {z^2}-\frac {1} {z^2}(1+z)(1-z)^{\frac {1}
  {2}(3^\al-1)}\Psi^{3^\al}(-z)
$$
solves \eqref{eq:AEF} modulo~3. Indeed, substitution of $A_1(z)$
for $A(z)$ on the left-hand side of \eqref{eq:AEF} yields
the expression
\begin{multline*}
\frac {1} {z^2}\left(1-2(1+z)(1-z)^{\frac {1}
  {2}(3^\al-1)}\Psi^{3^\al}(-z)+
(1+z)^2(1-z)^{3^\al-1}\Psi^{2\cdot 3^\al}(-z)\right)\\
+\frac {1} {z^2}-\frac {1} {z^2}(1+z)(1-z)^{\frac {1}
  {2}(3^\al-1)}\Psi^{3^\al}(-z)-\frac {1} {1-z}
\quad \text{modulo }3.
\end{multline*}
Now one uses the relation \eqref{eq:PsiRel2} (with $\ep=-1$ and
$\ga=1$) to ``get rid of" $\Psi^{2\cdot 3^\al}(-z)$. This leads
to the expression
\begin{multline*}
\frac {1} {z^2}\left(1-2(1+z)(1-z)^{\frac {1}
  {2}(3^\al-1)}\Psi^{3^\al}(-z)+
(1+z)^2(1-z)^{-1}\right)\\
+\frac {1} {z^2}-\frac {1} {z^2}(1+z)(1-z)^{\frac {1}
  {2}(3^\al-1)}\Psi^{3^\al}(-z)-\frac {1} {1-z}\\
=\frac {3} {z^2(1-z)}
-\frac {3} {z^2}(1+z)(1-z)^{\frac {1}
  {2}(3^\al-1)}\Psi^{3^\al}(-z)
\quad \text{modulo }3,
\end{multline*}
which is indeed $0$ modulo~3.

We proceed to the ``iteration step."
Supposing that $A_\be(z)$ as given in \eqref{eq:Abe}
solves \eqref{eq:AEF} modulo $3^\be$, we put 
\begin{equation*} 
a_{i,\be+1}(z):=a_{i,\be}(z)+3^{\be}b_{i,\be+1}(z),\quad 
i=0,1,\dots,2\cdot 3^{\al}-1,
\end{equation*}
as in \eqref{eq:Ansatz2a},
where the $b_{i,\be+1}(z)$'s are (at this point) undetermined 
Laurent polynomials in $z$ and $1-z$. Next we substitute
$$
\sum _{i=0} ^{2\cdot 3^{\al}-1}a_{i,\be+1}(z)\Psi^i(-z)
$$
in \eqref{eq:AEF} and consider the result modulo~$3^{\be+1}$. This 
leads to the congruence
\begin{multline*}
z^2A^2_\be(z)+A_\be(z)-\frac {1} {1-4z}
+3^\be \sum _{i=0} ^{2\cdot 3^{\al}-1}b_{i,\be+1}(z)\Psi^i(-z)\\
+2\cdot3^\be z^2\sum _{i=0} ^{2\cdot 3^{\al}-1}\sum _{j=0} ^{2\cdot 3^{\al}-1}
a_{i,\be}(z)b_{j,\be+1}(z)\Psi^{i+j}(-z)
=0
\quad 
\text {modulo }3^{\be+1}.
\end{multline*}
By our assumption on $A_{\be}(z)$, we may divide by $3^\be$, so that
we obtain the congruence
\begin{multline} \label{eq:sampleGl}
\sum _{i=0} ^{2\cdot 3^{\al}-1}\text{Rat}_i(z)\Psi^i(-z)
+ \sum _{i=0} ^{2\cdot 3^{\al}-1}b_{i,\be+1}(z)\Psi^i(-z)\\
-z^2 \sum _{i=0} ^{2\cdot 3^{\al}-1}\sum _{j=0} ^{2\cdot 3^{\al}-1}
a_{i,\be}(z)b_{j,\be+1}(z)\Psi^{i+j}(-z)
=0
\quad 
\text {modulo }3,
\end{multline}
where $\text{Rat}_i(z)$, $i=0,1,\dots,2\cdot 3^{\al}-1$, are certain
Laurent polynomials in $z$ and $1-z$ with integer coefficients.
Again, we may reduce powers $\Psi^{i+j}(-z)$ with $i+j\ge 2\cdot
3^\al$ by means of \eqref{eq:PsiRel2}. 
Furthermore, by construction, we know that
$a_{0,\be}(z)=1/z^2$~modulo~3 and that
$$a_{3^\al,\be}(z)=\frac {1} {z^2}(1+z)(1-z)^{\frac {1}
  {2}(3^\al-1)}
\quad \text{modulo }3.
$$
If we use all this in \eqref{eq:sampleGl}, then we obtain
\begin{multline*}
\sum _{i=0} ^{2\cdot 3^{\al}-1}\text{Rat}_i(z)\Psi^i(-z)
-(1+z)(1-z)^{\frac {1} {2}(3^\al-1)}\sum _{j=0} ^{3^{\al}-1}
b_{j,\be+1}(z)\Psi^{j+3^\al}(-z)\\
-\frac {1+z} {(1-z)^{\frac {1} {2}(3^\al+1)}}
\sum _{j=3^\al} ^{2\cdot 3^{\al}-1}
b_{j,\be+1}(z)\Psi^{j-3^\al}(-z)
=0
\quad 
\text {modulo }3.
\end{multline*}
Now we compare coefficients of powers of $\Psi(-z)$. 
This produces the system of congruences
$$
(1+z)(1-z)^{g_i}
b_{i,\be+1}(z)=\text {Rat}_{i+3^\al}(z)\quad 
\text {modulo }3,\quad \quad 
i=0,1,\dots,2\cdot 3^\al-1,
$$
where $g_i=\frac {1} {2}(3^\al-1)$ if $0\le i<3^\al$,
$g_i=-\frac {1} {2}(3^\al+1)$ if $3^\al\le i<2\cdot 3^\al$,
and the index $i+3^\al$ in $\text {Rat}_{i+3^\al}$ has to be
taken modulo $2\cdot 3^\al$. Clearly, this system is trivially 
uniquely solvable. Moreover, in all cases which we looked at,
the numerator of the rational function $\text{Rat}_i(z)$ is
divisible by $1+z$, $i=0,1,\dots,2\cdot 3^\al-1$. 
We have no explanation for this, yet.
However, under the assumption that this is indeed the case, 
we see that
the algorithm will go through for $\be=1,2,\dots,3^\al$, and,
consequently, for all non-negative integers $\al$ 
the series $A(z)$, when coefficients are reduced modulo~$3^{3^\al}$,
can be represented as a polynomial in $\Psi(-z)$ with coefficients
that are Laurent polynomials in $z$ and $1-z$. 

We have implemented the algorithm explained above.
For $\al=1$, it produces the following result.

\begin{theorem} \label{thm:sample-27}
Let $\Psi(z)$ be given by \eqref{eq:Psidef}.
Then, we have
\begin{multline} 
\label{eq:Loessample27}
\sum _{n\ge0} ^{}\binom {2n+2}n\,z^n
=\frac{13}{z^2}
+3 \left(-\frac{z+1}{z^2}+3\frac{
   1-z-z^2}{z^2(1-z)}\right)
   \Psi(-z)\\
+\left(4
   +\frac{6}{z}-\frac{4}{z^2}\right)
   \Psi^3(-z)
+3 \left(4
   z+2-\frac{1}{z}-\frac{5}{z^2}\right)
   \Psi^5(-z)
\\
\quad 
\text {\em modulo }27.
\end{multline}
\end{theorem}
The congruence results for $\Psi^3(z)$ and $\Psi^5(z)$, which one can
obtain by following the recipe provided in the appendix, 
and which are
presented explicitly in Propositions~\ref{prop:Psi-3} and \ref{prop:Psi-5},
could now be used to find explicit criteria when $\binom
{2n+2}n$ is congruent to $0,1,2,\dots,26$ modulo~27.\footnote{In
  principle, such results could also be obtained via
the extensions of Lucas' theorem given in \cite{DaWeAA}
or \cite[Theorem~1]{GranAA},
which provide a means for determining the congruence class of
a binomial coefficient modulo a given prime power.}
Since a display of such results is rather lengthy, we omit them
here, also in most of the subsequent sections. As we already said
in the Introduction, we confine ourselves to providing only one
illustration of coefficient extraction, namely for the case of free
subgroup numbers for $PSL_2(\Z)$; see Section~\ref{sec:Free}.

\section{A general theorem on functional--differential
equations of quadratic type}
\label{sec:gen}

The purpose of this section is to show that the method from 
Section~\ref{sec:method} 
always applies to a certain class of functional-differential
equations. The characteristic feature of the equations in this class
is that, when reduced modulo~3, they become quadratic
equations. Since there are numerous combinatorial objects
whose generating functions satisfy quadratic equations, the main
theorem of this section has a wide range of applications.
These will be discussed in
Sections~\ref{sec:Motzkin}--\ref{sec:Eulerian}.

\begin{theorem} \label{thm:general}
Let $\Psi(z)$ be given by \eqref{eq:Psidef}, 
and let $\al$ be some positive integer.
Furthermore, suppose that 
the formal power series $F(z)$ with integer coefficients
satisfies the functional-differential equation
\begin{equation} \label{eq:generalEF}
c_2(z)F^2(z)+c_1(z)F(z)+c_0(z)+
3\mathcal Q(z;F(z),F'(z),F''(z),\dots,F^{(s)}(z))=0,
\end{equation}
where 
\begin{enumerate} 
\item $c_2(z)=z^{e_1}(1+\ep z^{\ga})^{e_2}\
\text {\em modulo }3$, with non-negative integers $e_1,e_2$
and $\ep\in\{1,-1\}$;
\item $c_1^2(z)-c_0(z)c_2(z)=z^{2f_1}(1+\ep
  z^\ga)^{2f_2+1}\
\text {\em modulo }3$, with non-negative integers $f_1,f_2$;
\item 
$\mathcal Q$ is a polynomial with integer coefficients;
\item 
the equation \eqref{eq:generalEF} has a
uniquely determined formal power solution modulo~$3^{3^\al}$.
\end{enumerate}
Then $F(z)$, when coefficients are 
reduced modulo $3^{3^\al}$, 
can be expressed as a polynomial in $\Psi(\ep z^\ga)$ of the form
$$
F(z)=a_0(z)+\sum _{i=0} ^{2\cdot 3^{\al}-1}a_{i}(z)\Psi^{i}(\ep z^\ga)\quad 
\text {\em modulo }3^{3^\al},
$$ 
where the coefficients $a_{i}(z)$, $i=0,1,\dots,2\cdot 3^\al-1$, 
are Laurent polynomials in $z$ and $1+\ep z^\ga$.
\end{theorem}

\begin{remarknu} \label{rem:2}
(1) It is not sufficient to assume 
that Equation~\eqref{eq:generalEF} determines
a unique formal power series solution $F(z)$ {\it over the integers}.
In order to illustrate this point, let us consider the (linear)
differential equation 
\begin{equation} \label{eq:Apery0EF} 
\left(4 z^2+25
   z-1\right) F(z)
+\left(5 z^3+44
   z^2-3 z\right)
   F'(z)
+\left(z^4+11 z^3-z^2\right)
   F''(z)+z+3=0.
\end{equation}
This equation determines a unique formal power series solution 
$F(z)=\sum_{n\ge0}F_nz^n$,
as can easily be seen by noticing that 
it is equivalent to the recurrence relation
\begin{equation} \label{eq:Apery0Rek}
(n+3)^2 F_{n+2}
-\left(11 n^2+55 n+69\right)
   F_{n+1}
-(n+2)^2
   F_n
=0\quad \text{for }n\ge0,
\end{equation}
with initial conditions $F_0=3$ and $F_1=19$.
(The alert reader will have noticed that this is the recurrence relation
\eqref{eq:Apery2Rek} for 
the Ap\'ery numbers for $\zeta(2)$ given by \eqref{eq:Apery2},
shifted by~1. In particular, the differential equation
\eqref{eq:Apery0EF} {\it is} solved by a formal power series with
integer coefficients.)
 
On the other hand, the equivalence of \eqref{eq:Apery0EF} and
\eqref{eq:Apery0Rek} also shows that the differential equation
\eqref{eq:Apery0EF} {\it does not\/} determine a unique formal power series
solution {\it modulo}~3. In fact, it is straightforward to verify that 
the series $z+(z^2+z^4)f(z^3)$ solves \eqref{eq:Apery0EF} when reduced
modulo~3, where
$f(z)$ can be an {\it arbitrary} formal power series in $z$.
Hence, in this case, our method would yield an infinite
set of candidate solutions modulo~3, and there would not
be any means inherent to the method to figure out which of them corresponds
to the unique solution $F(z)$ to \eqref{eq:Apery0EF} over the integers, 
when coefficients are reduced modulo~$3$.

\medskip
(2) In Condition (4), it is not sufficient to require uniqueness only
modulo~3, one really needs uniqueness modulo the power of~$3$ for which
one aims at finding a solution to \eqref{eq:generalEF}. 
In order to illustrate this, let us consider the functional equation
\begin{equation} \label{eq:Cat2}
z^2F^4(z)-2zF^3(z)+(2z+1)F^2(z)-2F(z)+1=0. 
\end{equation}
Observing that the left-hand side is a perfect square,
so that \eqref{eq:Cat2} can actually be rewritten as
$$
(zF^2(z)-F(z)+1)^2=0,
$$
we see that it has a unique formal power series solution over the
integers, given by the generating function $C(z)=\sum_{n\ge0}C_nz^n$ 
of Catalan numbers of Section~\ref{sec:Catalan}. This is, at the same
time, the unique solution to \eqref{eq:Cat2} modulo~3.

On the other hand, if we consider \eqref{eq:Cat2} modulo~9, and
compare constant coefficients in \eqref{eq:Cat2}, then,
writing again $F(z)=\sum_{n\ge0}F_nz^n$, we see that
this leads to the congruence
$$
(F_0-1)^2\equiv 0\pmod9,
$$
which has no unique solution. In fact, it is straightforward to verify that 
the series $C(z)+3f(z)$ solves \eqref{eq:Cat2} when reduced
modulo~9, where
$f(z)$ can be an {\it arbitrary} formal power series in $z$.
Hence again, our method would yield an infinite
set of candidate solutions modulo~9, and there would not
be any means inherent to the method to figure out which of them equals
the unique solution $F(z)$ to \eqref{eq:Apery0EF}, 
when coefficients are reduced modulo~$9$. It is obvious that in a
similar fashion one can construct examples where
uniqueness modulo~$3^\be$, say, does not imply uniqueness 
modulo~$3^{\be+1}$.
\end{remarknu}

\begin{proof}[Proof of Theorem~\em\ref{thm:general}] 
We apply the method from Section~\ref{sec:method}. As we discussed
there and in the sample application in the previous section, 
it consists of two steps, a {\it base step}, which determines
a solution to the functional-differential equation
\eqref{eq:generalEF} modulo~3, and an {\it iteration step}, which affords
solutions modulo higher and higher powers of~$3$.

In the entire proof, we always assume without loss of generality that
$c_0(z)$, $c_1(z)$, and $c_2(z)$ are already reduced modulo~3
(otherwise we could put terms which are divisible by~3 into the
polynomial $\mathcal Q(\,.\,)$).

\medskip
{\sc Base step.}
We start by
substituting the Ansatz \eqref{eq:Ansatz1} in \eqref{eq:generalEF} and
reducing the result modulo~$3$. 
In this way, we obtain 
\begin{multline*} 
c_2(z)\sum _{i=0} ^{2\cdot 3^{\al}-1}\sum _{j=0} ^{2\cdot 3^{\al}-1}
a_{i,1}(z)a_{j,1}(z)\Psi^{i+j}(\ep z^\ga)\\
+c_1(z)\sum _{i=0} ^{2\cdot 3^{\al}-1}a_{i,1}(z)\Psi^i(\ep z^\ga)
+c_0(z)=0\quad \text
{modulo }3.
\end{multline*}
When $i+j\ge 2\cdot 3^\al$,
we may reduce $\Psi^{i+j}(\ep z^\ga)$ further using relation
\eqref{eq:PsiRel2} (with $z$ replaced by $\ep z^\ga$). This leads to
\begin{multline} \label{eq:Glgeneral}
c_2(z)\underset{i+j< 2\cdot 3^\al}{\sum _{0\le i,j\le 2\cdot 3^{\al}-1}}
a_{i,1}(z)a_{j,1}(z)\Psi^{i+j}(\ep z^\ga)\\
+c_2(z)\underset{i+j\ge 2\cdot 3^\al}{\sum _{0\le i,j\le 2\cdot 3^{\al}-1}}
a_{i,1}(z)a_{j,1}(z)\frac {1} {(1+\ep z^\ga)^{3^{\al}}}\Psi^{i+j-2\cdot 3^\al}(\ep z^\ga)\\
+c_1(z)\sum _{i=0} ^{2\cdot 3^{\al}-1}a_{i,1}(z)\Psi^i(\ep z^\ga)
+c_0(z)=0\quad \text
{modulo }3.
\end{multline}
Now we compare coefficients of $\Psi^i(\ep z^\ga)$, for
$i=0,1,\dots,2\cdot 3^{\al}-1$. We claim that the choice of
\begin{align} \label{eq:a0}
a_{0,1}(z)&=\frac {c_1(z)} {c_2(z)}=
z^{-e_1}(1+\ep z^{\ga})^{-e_2}{c_1(z)} ,\\[2mm]
\label{eq:a3al}
a_{3^{\al},1}(z)&=\pm\frac {(c_1^2(z)-c_0(z)c_2(z))^{1/2}} {c_2(z)}
(1+\ep z^\ga)^{3^\al/2}=
\pm z^{f_1-e_1}(1+\ep z^\ga)^{f_2-e_2+\frac {1} {2}(3^\al+1)},
\end{align}
with all other $a_{i,1}(z)$'s vanishing, provides two
solutions (depending on the sign in front of the expression for 
$a_{3^{\al},1}(z)$) modulo~$3$ to the system of congruences resulting from 
\eqref{eq:Glgeneral} in rational functions $a_{i,1}(z)$.
Namely, if we substitute this choice in
the left-hand side of \eqref{eq:Glgeneral}, then we
obtain 
\begin{multline*}
c_2(z)\frac {c_1^2(z)} {c_2^2(z)}+
2c_2(z)\frac {c_1(z)} {c_2(z)}a_{3^{\al},1}(z)
\Psi^{3^\al}(\ep z^\ga)
+c_2(z)\frac {c_1^2(z)-c_0(z)c_2(z)} {c_2^2(z)}\\
+c_1(z)\frac {c_1(z)} {c_2(z)}+c_1(z)a_{3^{\al},1}(z)
\Psi^{3^\al}(\ep z^\ga)
+c_0(z),
\end{multline*}
which indeed vanishes modulo 3.

The two solutions given by \eqref{eq:a0} and \eqref{eq:a3al} 
are different modulo~$3$ since $a_{3^\al,1}(z)\ne0$.
With both these solutions, we enter the iteration step.

\medskip
{\sc Iteration step.}
We consider
the Ansatz \eqref{eq:Ansatz2}--\eqref{eq:Ansatz2a}, where the
coefficients $a_{i,\be}(z)$ are supposed to provide a solution
$F_{\be}(z)=\sum _{i=0} ^{2\cdot 3^{\al}-1}a_{i,\be}(z)\Psi^i(\ep z^\ga)$ to
\eqref{eq:generalEF} modulo~$3^\be$. This Ansatz, substituted in
\eqref{eq:generalEF}, produces the congruence
\begin{multline*}
c_2(z)F^2_{\be}(z)+c_1(z)F_{\be}(z)+c_0(z)+
3\mathcal Q(z;F_\be(z),F'_\be(z),F''(z),\dots,F^{(s)}_\be(z))\\
+3^\be c_1(z)\sum _{i=0} ^{2\cdot 3^{\al}-1}b_{i,\be+1}(z)\Psi^i(\ep z^\ga)\\
+2\cdot3^\be c_2(z)\sum _{i=0} ^{2\cdot 3^{\al}-1}\sum _{j=0} ^{2\cdot 3^{\al}-1}
a_{i,\be}(z)b_{j,\be+1}(z)\Psi^{i+j}(\ep z^\ga)
=0
\quad 
\text {modulo }3^{\be+1}.
\end{multline*}
By our assumption on $F_{\be}(z)$, we may divide by $3^\be$.
Furthermore, by construction, we know that
$a_{0,\be}(z)=c_1(z)/c_2(z)$~modulo~3 and that
$$a_{3^\al,\be}(z)=\pm z^{f_1-e_1}(1+\ep z^\ga)^{f_2-e_2+\frac {1}
  {2}(3^\al+1)}\quad \text{modulo }3
$$
(the sign $\pm$ depending on the sign in \eqref{eq:a3al}),
while all other $a_{i,\be}(z)$'s vanish modulo~3. If we use these
facts and also perform the reduction \eqref{eq:PsiRel2}, 
then we arrive at the congruence
\begin{multline*}
\sum _{i=0} ^{2\cdot 3^{\al}-1}\text{Rat}_i(z)\Psi^i(\ep z^\ga)
+c_1(z)\sum _{i=0} ^{2\cdot 3^{\al}-1}b_{i,\be+1}(z)\Psi^i(\ep z^\ga)\\
+2\cdot c_2(z)\Bigg(\frac {c_1(z)} {c_2(z)}\sum _{j=0} ^{2\cdot
  3^{\al}-1}b_{j,\be+1}(z)\Psi^j(\ep z^\ga)
\kern6cm\\ 
\pm z^{f_1-e_1}(1+\ep z^\ga)^{f_2-e_2+\frac {1}
  {2}(3^\al+1)}
\sum _{j=0} ^{3^{\al}-1}
b_{j,\be+1}(z)\Psi^{j+3^\al}(\ep z^\ga)\\
\pm z^{f_1-e_1}(1+\ep z^\ga)^{f_2-e_2-\frac {1}
  {2}(3^\al-1)}
\sum _{j=3^\al} ^{2\cdot 3^{\al}-1}
b_{j,\be+1}(z)\Psi^{j- 3^\al}(\ep z^\ga)
\Bigg)
=0
\quad 
\text {modulo }3,
\end{multline*}
where $\text {Rat}_i(z)$, $i=0,1,\dots,2\cdot 3^{\al}-1$, are certain
Laurent polynomials in $z$ and $1+\ep z^\ga$ with integer coefficients.
Reducing further modulo~3, this simplifies to
\begin{multline*}
\sum _{i=0} ^{2\cdot 3^{\al}-1}\text{Rat}_i(z)\Psi^i(\ep z^\ga)
\mp z^{f_1}
(1+\ep z^\ga)^{f_2+\frac {1}
  {2}(3^\al+1)}
\sum _{j=0} ^{3^{\al}-1}
b_{j,\be+1}(z)\Psi^{j+3^\al}(\ep z^\ga)\\
\mp z^{f_1}
(1+\ep z^\ga)^{f_2-\frac {1}
  {2}(3^\al-1)}
\sum _{j=3^\al} ^{2\cdot 3^{\al}-1}
b_{j,\be+1}(z)\Psi^{j-3^\al}(\ep z^\ga)
=0
\quad 
\text {modulo }3.
\end{multline*}
Comparison of powers of $\Psi(\ep z^\ga)$ then yields the system of congruences
$$
\pm z^{f_1}(1+\ep z^\ga)^{g_i}
b_{i,\be+1}(z)=\text {Rat}_{i+3^\al}(z)\quad 
\text {modulo }3,\quad \quad 
i=0,1,\dots,2\cdot 3^\al-1,
$$
where $g_i=f_2+\frac {1} {2}(3^\al+1)$ if $0\le i<3^\al$,
$g_i=f_2-\frac {1} {2}(3^\al-1)$ if $3^\al\le i<2\cdot 3^\al$,
and the index $i+3^\al$ in $\text {Rat}_{i+3^\al}$ has to be
taken modulo $2\cdot 3^\al$.
This system is trivially uniquely solvable. Since we had two
different solutions modulo~$3$ to start with, we now obtain two
different solutions $F_\be(z)$ to \eqref{eq:generalEF} modulo $3^{\be}$.
When we arrive at $\be=3^\al$,
the uniqueness condition~(4) guarantees that only one of the two
obtained solutions is a formal power series in $z$ (the other must
be a genuine Laurent series), 
which necessarily equals the uniquely determined solution 
of \eqref{eq:generalEF} modulo~$3^{3^\al}$. 
Thus, we have proven that
the algorithm of
Section~\ref{sec:method} will produce a solution $F_{{3^\al}}(z)$ 
to \eqref{eq:generalEF} modulo $3^{3^\al}$ which is a
polynomial in $\Psi(\ep z^\ga)$ with coefficients that are 
Laurent polynomials in $z$ and $1+\ep z^\ga$, and which equals
any formal power series solution $F(z)$ to \eqref{eq:generalEF}
when coefficients are reduced modulo~$3^{3^\al}$.
\end{proof}

\section{Motzkin numbers}
\label{sec:Motzkin}

Let $M_n$ be the $n$-th {\it Motzkin number}, that is, the number of
lattice paths from $(0,0)$ to $(n,0)$ consisting of steps taken from
the set
$\{(1,0),(1,1),(1,-1)\}$ never running below the $x$-axis.
It is well-known (cf.\ \cite[Ex.~6.37]{StanBI})
that the generating function $M(z)=\sum_{n\ge0}M_n\,z^n$ is given by
\begin{equation} \label{eq:Motzkin} 
M(z)=\frac {1-z-\sqrt{1-2z-3z^2}} {2z^2},
\end{equation}
and hence satisfies the functional equation
\begin{equation} \label{eq:MotzkinEF}
z^2M^2(z)+(z-1)M(z)+1=0.
\end{equation}

\begin{theorem} \label{thm:Motzkin}
Let $\Psi(z)$ be given by \eqref{eq:Psidef}, 
and let $\al$ be some positive integer.
Then the generating function $M(z)=\sum_{n\ge0}M_n\,z^n$ 
for the Motzkin numbers,
reduced modulo $3^{3^\al}$, 
can be expressed as a polynomial in $\Psi(z^3)$ of the form
$$
M(z)=a_0(z)+\sum _{i=0} ^{2\cdot 3^{\al}-1}a_{i}(z)\Psi^{i}(z^3)\quad 
\text {\em modulo }3^{3^\al},
$$ 
where the coefficients $a_{i}(z)$, $i=0,1,\dots,2\cdot 3^\al-1$, 
are Laurent polynomials in $z$ and $1+z$.
\end{theorem}

\begin{proof} 
Choosing $\ga=1$, $c_2(z)=z^2$, $c_1(z)=z-1$, $c_0(z)=1$, we see that
the functional equation \eqref{eq:MotzkinEF} fits into the framework
of Theorem~\ref{thm:general}, since 
$$c_1^2(z)-c_0(z)c_2(z)=1+z\quad \text{modulo }3.$$
The constants $e_1,e_2,f_1,f_2$ are given by
$e_1=2$,
$e_2=0$,
$f_1=0$,
$f_2=0$,
$\ep=1$, and
we have
$\mathcal Q(\dots)=0$.
Consequently, the generating function $M(z)$, when reduced modulo
$3^{3^\al}$, can be written as a polynomial in the basic series
$\Psi(z)=(1+z)\Psi(z^3)$. In the formulation of the theorem, we
have chosen to express everything in terms of $\Psi(z^3)$, since this
leads to the more elegant expressions for the moduli 9 and 27 given below.
\end{proof}

We have implemented the algorithm explained in the above proof.
For $\al=1$, it produces the following result.

\begin{theorem} \label{thm:Motzkin-27}
Let $\Psi(z)$ be given by \eqref{eq:Psidef}.
Then, we have
\begin{multline} 
\label{eq:LoesM27}
\sum _{n\ge0} ^{}M_n\,z^n
=
13z^{-1}+14z^{-2}
+\left(9
   z+12+24z^{-1}+21z^{-2}\right)
   \Psi(z^3)\\
+\left(9 z^5+12 z^4+10
   z^3+23 z^2+25
   z+19+14z^{-1}+4z^{-2}\right)
   \Psi^3(z^3)\\
-\left(9 z^7+3 z^6+24
   z^5+30 z^4+6 z^3+21
   z^2+6
   z+3+24z^{-1}+12z^{-2}\right)
   \Psi^5(z^3)\\
\quad 
\text {\em modulo }27.
\end{multline}
\end{theorem}

In the case of modulus 9, this theorem reduces to the following
result.

\begin{corollary}
\label{cor:Motzkin-9}
Let $\Psi(z)$ be given by \eqref{eq:Psidef}.
Then, we have
\begin{multline} 
\label{eq:LoesM9}
\sum _{n\ge0} ^{}M_n\,z^n
=
4z^{-1}+5z^{-2}
-\left(3+6z^{-1}+3z^{-2}
\right) \Psi(z^3)\\
+\left(3
   z^4+4 z^3+2
   z^2+1
   z+4+2z^{-1}+7z^{-2}\right)
   \Psi^3(z^3)
\quad \text {\em modulo }9.
\end{multline}
\end{corollary}

These results generalise
Corollary~4.10 in \cite{DeSaAA}, which is an easy consequence.

\section{Motzkin prefix numbers}
\label{sec:Motzkinpref}

Let $MP_n$ be the $n$-th {\it  Motzkin prefix number}, that is, the number of
lattice paths from $(0,0)$ consisting of $n$ steps taken from
the set
$\{(1,0),(1,1),(1,-1)\}$ never running below the $x$-axis.
It is easy to see
that the generating function $MP(z)=\sum_{n\ge0}MP_n\,z^n$ satisfies the
functional equation
\begin{equation} \label{eq:MotzkinprefEF}
z(1-3z)MP^2(z)+(1-3z)MP(z)-1=0.
\end{equation}
Namely, let again $M(z)$ denote the generating function for Motzkin
numbers (see Section~\ref{sec:Motzkin}). Then, by an elementary
combinatorial decomposition (a Motzkin prefix can be empty, it can
start with a horizontal step followed by a Motzkin prefix, it can start
with a step $(1,1)$ and never return to the $x$-axis, or start with
a step $(1,1)$, return to the $x$-axis for the first time by a step
$(1,-1)$, and then be followed by a Motzkin prefix), we have
$$
MP(z)=1+2z\,MP(z)+z^2M(z)MP(z).
$$
Combined with \eqref{eq:Motzkin}, this yields an explicit expression
for $MP(z)$, from which the functional equation
\eqref{eq:MotzkinprefEF} can be derived.

\begin{theorem} \label{thm:Motzkinpref}
Let $\Psi(z)$ be given by \eqref{eq:Psidef}, 
and let $\al$ be some positive integer.
Then the generating function $MP(z)=\sum_{n\ge0}MP_n\,z^n$ 
for the Motzkin prefix numbers,
reduced modulo $3^{3^\al}$, 
can be expressed as a polynomial in $\Psi(z)$ of the form
$$
MP(z)=a_0(z)+\sum _{i=0} ^{2\cdot 3^{\al}-1}a_{i}(z)\Psi^{i}(z)\quad 
\text {\em modulo }3^{3^\al},
$$ 
where the coefficients $a_{i}(z)$, $i=0,1,\dots,2\cdot 3^\al-1$, 
are Laurent polynomials in $z$ and $1+z$.
\end{theorem}

\begin{proof} 
Choosing $\ga=1$, $c_2(z)=z(1-3z)$, $c_1(z)=1-3z$, $c_0(z)=-1$, we see that
the functional equation \eqref{eq:MotzkinprefEF} fits into the framework
of Theorem~\ref{thm:general}, since 
$$c_1^2(z)-c_0(z)c_2(z)=1+z\quad \text{modulo }3.$$
The constants $e_1,e_2,f_1,f_2$ are given by
$e_1=1$,
$e_2=0$,
$f_1=0$,
$f_2=0$,
$\ep=1$, and
we have
$\mathcal Q(\dots)=0$.
Consequently, the generating function $MP(z)$, when reduced modulo
$3^{3^\al}$, can be written as a polynomial in the basic series
$\Psi(z)$, as claimed.
\end{proof}

We have implemented the algorithm explained in the above proof.
For $\al=1$, it produces the following result.

\begin{theorem} \label{thm:Motzkinpref-27}
Let $\Psi(z)$ be given by \eqref{eq:Psidef}.
Then, we have
\begin{multline} 
\label{eq:LoesMP27}
\sum _{n\ge0} ^{}MP_n\,z^n
=
13z^{-1}
+\left(9
   z+6z^{-1}+15\right)
   \Psi(z)
-\left(6 z^2+16
   z+4z^{-1}+14\right)
   \Psi^3(z)\\
-\left(9 z^3+15
   z^2+18
   z+15z^{-1}\right)
   \Psi^5(z)
\quad 
\text {\em modulo }27.
\end{multline}
\end{theorem}

In the case of modulus 9, this theorem reduces to the following
result.

\begin{corollary}
\label{cor:Motzkinpref-9}
Let $\Psi(z)$ be given by \eqref{eq:Psidef}.
Then, we have
\begin{equation} 
\label{eq:LoesMP9}
\sum _{n\ge0} ^{}MP_n\,z^n
=
4z^{-1}
+3
   \left(1+z^{-1}\right)
   \Psi(z)
+\left(3z^2-z+7+2z^{-1}
\right)
   \Psi^3(z)
\quad \text {\em modulo }9.
\end{equation}
\end{corollary}

These results generalise
Corollary~4.11 in \cite{DeSaAA}, which is an easy consequence.

\section{Riordan numbers}
\label{sec:Riordan}

Let $R_n$ be the $n$-th {\it Riordan number}, that is, the number of
lattice paths from $(0,0)$ to $(n,0)$ consisting of steps taken from
the set
$\{(1,0),(1,1),(1,-1)\}$ never running below the $x$-axis, and where
steps $(1,0)$ are not allowed on the $x$-axis.
It is easy to see
that the generating function $R(z)=\sum_{n\ge0}R_n\,z^n$ satisfies the
functional equation
\begin{equation} \label{eq:RiordanEF}
z(1+z)R^2(z)-(z+1)R(z)+1=0.
\end{equation}
Namely, let again $M(z)$ denote the generating function for Motzkin
numbers (see Section~\ref{sec:Motzkin}). Then, by an elementary
combinatorial decomposition (a path as above can be empty, or it may start with
a step $(1,1)$, return to the $x$-axis for the first time by a step
$(1,-1)$, and then be followed by a path of the above type), we have
$$
R(z)=1+z^2M(z)R(z).
$$
Combined with \eqref{eq:Motzkin}, this yields an explicit expression
for $R(z)$, from which the functional equation
\eqref{eq:RiordanEF} can be derived.

\begin{theorem} \label{thm:Riordan}
Let $\Psi(z)$ be given by \eqref{eq:Psidef}, 
and let $\al$ be some positive integer.
Then the generating function $R(z)=\sum_{n\ge0}R_n\,z^n$ 
for the Riordan numbers,
reduced modulo $3^{3^\al}$, 
can be expressed as a polynomial in $\Psi(z^3)$ of the form
$$
R(z)=a_0(z)+\sum _{i=0} ^{2\cdot 3^{\al}-1}a_{i}(z)\Psi^{i}(z^3)\quad 
\text {\em modulo }3^{3^\al},
$$ 
where the coefficients $a_{i}(z)$, $i=0,1,\dots,2\cdot 3^\al-1$, 
are Laurent polynomials in $z$ and $1+z$.
\end{theorem}

\begin{proof} 
Choosing $\ga=1$, $c_2(z)=z(1+z)$, $c_1(z)=-z-1$, $c_0(z)=1$, we see that
the functional equation \eqref{eq:RiordanEF} fits into the framework
of Theorem~\ref{thm:general}, since 
$$c_1^2(z)-c_0(z)c_2(z)=1+z\quad \text{modulo }3.$$
The constants $e_1,e_2,f_1,f_2$ are given by
$e_1=1$,
$e_2=1$,
$f_1=0$,
$f_2=0$,
$\ep=1$, and
we have
$\mathcal Q(\dots)=0$.
Consequently, the generating function $R(z)$, when reduced modulo
$3^{3^\al}$, can be written as a polynomial in the basic series
$\Psi(z)=(1+z)\Psi(z^3)$. In the formulation of the theorem, we
have chosen to express everything in terms of $\Psi(z^3)$, since this
leads to the more elegant expressions for the moduli 9 and 27 given below.
\end{proof}

We have implemented the algorithm explained in the above proof.
For $\al=1$, it produces the following result.

\begin{theorem} \label{thm:Riordan-27}
Let $\Psi(z)$ be given by \eqref{eq:Psidef}.
Then, we have
\begin{multline} 
\label{eq:LoesR27}
\sum _{n\ge0} ^{}R_n\,z^n
=14z^{-1}
+\left(9+21z^{-1}
   \right) \Psi(z)
+\left(9 z^2+30
   z+25+4z^{-1}\right)
   \Psi^3(z)\\
+\left(18 z^2+24
   z+21+15z^{-1}\right)
   \Psi^5(z)
\quad 
\text {\em modulo }27.
\end{multline}
\end{theorem}

In the case of modulus 9, this theorem reduces to the following
result.

\begin{corollary}
\label{cor:Riordan-9}
Let $\Psi(z)$ be given by \eqref{eq:Psidef}.
Then, we have
\begin{equation} 
\label{eq:LoesR9}
\sum _{n\ge0} ^{}R_n\,z^n
=-4z^{-1}
-3z^{-1}\Psi(z)+\left(3
   z+10+7z^{-1}\right)
   \Psi^3(z)
\quad \text {\em modulo }9.
\end{equation}
\end{corollary}

These results generalise
Corollary~4.12 in \cite{DeSaAA}, which is an easy consequence.

\section{Central trinomial coefficients}
\label{sec:trinomial}

Let $T_n$ be the $n$-th {\it central trinomial coefficient}, that is, the
coefficient of $t^n$ in\break $(1+t+t^2)^n$.
Already Euler knew 
that the generating function $T(z)=\sum_{n\ge0}T_n\,z^n$
equals $1/\sqrt{1-2z-3z^2}$ (cf.\ \cite[solution to 
Exercise~6.42]{StanBI}). Consequently, $T(z)$ satisfies the
functional equation
\begin{equation} \label{eq:trinomialEF}
(1-2z-3z^2)T^2(z)-1=0.
\end{equation}

\begin{theorem} \label{thm:trinomial}
Let $\Psi(z)$ be given by \eqref{eq:Psidef}, 
and let $\al$ be some positive integer.
Then the generating function $T(z)=\sum_{n\ge0}T_n\,z^n$ 
for the trinomial coefficients,
reduced modulo $3^{3^\al}$, 
can be expressed as a polynomial in $\Psi(z^3)$ of the form
$$
T(z)=a_0(z)+\sum _{i=0} ^{2\cdot 3^{\al}-1}a_{i}(z)\Psi^{i}(z^3)\quad 
\text {\em modulo }3^{3^\al},
$$ 
where the coefficients $a_{i}(z)$, $i=0,1,\dots,2\cdot 3^\al-1$, 
are Laurent polynomials in $z$ and $1+z$.
\end{theorem}

\begin{proof} 
Choosing $\ga=1$, $c_2(z)=1-2z-3z^2$, $c_1(z)=0$, $c_0(z)=-1$, we see that
the functional equation \eqref{eq:trinomialEF} fits into the framework
of Theorem~\ref{thm:general}, since 
$$c_1^2(z)-c_0(z)c_2(z)=1+z\quad \text{modulo }3.$$
The constants $e_1,e_2,f_1,f_2$ are given by
$e_1=0$,
$e_2=1$,
$f_1=0$,
$f_2=0$,
$\ep=1$, and
we have
$\mathcal Q(\dots)=0$.
Consequently, the generating function $T(z)$, when reduced modulo
$3^{3^\al}$, can be written as a polynomial in the basic series
$\Psi(z)=(1+z)\Psi(z^3)$. In the formulation of the theorem, we
have chosen to express everything in terms of $\Psi(z^3)$, since this
leads to the more elegant expressions for the moduli 9 and 27 given below.
\end{proof}

We have implemented the algorithm explained in the above proof.
For $\al=1$, it produces the following result.

\begin{theorem} \label{thm:trinomial-27}
Let $\Psi(z)$ be given by \eqref{eq:Psidef}.
Then, we have
\begin{multline} 
\label{eq:LoesT27}
\sum _{n\ge0} ^{}T_n\,z^n
=
-\left(9 z^2+24 z+15\right)
   \Psi(z^3)
+\left(15 z^5+25 z^4+4
   z^3+12 z^2+10 z+19\right)
   \Psi^3(z^3)\\
+\left(9 z^8+6 z^7+6
   z^6+9 z^5+21 z^4+3 z^3+15
   z+24\right) \Psi^5(z^3)
\quad 
\text {\em modulo }27.
\end{multline}
\end{theorem}

In the case of modulus 9, this theorem reduces to the following
result.

\begin{corollary}
\label{cor:trinomial-9}
Let $\Psi(z)$ be given by \eqref{eq:Psidef}.
Then, we have
\begin{equation} 
\label{eq:LoesT9}
\sum _{n\ge0} ^{}T_n\,z^n
=
-3 (z+1) \Psi(z^3)+
\left(-3z^5+z^4+7 z^3+3 z^2+4 z+4
\right)
   \Psi^3(z^3)
\quad \text {\em modulo }9.
\end{equation}
\end{corollary}

These results generalise
Corollary~4.9 in \cite{DeSaAA}, which is an easy consequence.

\section{Central binomial coefficients and their sums}
\label{sec:centbin}

In this section we consider the sequence $\left(\binom
{2n}n\right)_{n\ge0}$ of {\it central binomial numbers}. It follows from
the binomial theorem that the corresponding generating function
$CB(z)$ is given by
$$
CB(z):=\sum_{n\ge0}\binom {2n}n\,z^n=\frac {1} {\sqrt{1-4z}},
$$
and hence satisfies the functional equation
\begin{equation} \label{eq:centbinEF}
(1-4z)CB^2(z)-1=0.
\end{equation}

\begin{theorem} \label{thm:centbin}
Let $\Psi(z)$ be given by \eqref{eq:Psidef}, 
and let $\al$ be some positive integer.
Then the generating function $CB(z)=\sum_{n\ge0}\binom {2n}n\,z^n$ 
for the central binomial numbers,
reduced modulo $3^{3^\al}$, 
can be expressed as a polynomial in $\Psi(-z)$ of the form
$$
CB(z)=a_0(z)+\sum _{i=0} ^{2\cdot 3^{\al}-1}a_{i}(z)\Psi^{i}(-z)\quad 
\text {\em modulo }3^{3^\al},
$$ 
where the coefficients $a_{i}(z)$, $i=0,1,\dots,2\cdot 3^\al-1$, 
are Laurent polynomials in $z$ and $1-z$.
\end{theorem}

\begin{proof} 
Choosing $\ga=1$, $c_2(z)=1-4z$, $c_1(z)=0$, $c_0(z)=-1$, we see that
the functional equation \eqref{eq:centbinEF} fits into the framework
of Theorem~\ref{thm:general}, since 
$$c_1^2(z)-c_0(z)c_2(z)=1-z\quad \text{modulo }3.$$
The constants $e_1,e_2,f_1,f_2$ are given by
$e_1=0$,
$e_2=1$,
$f_1=0$,
$f_2=0$,
$\ep=-1$, and
we have
$\mathcal Q(\dots)=0$.
Consequently, the generating function $CB(z)$, when reduced modulo
$3^{3^\al}$, can be written as a polynomial in the basic series
$\Psi(-z)$, as claimed.
\end{proof}

We have implemented the algorithm explained in the above proof.
For $\al=1$, it produces the following result.

\begin{theorem} \label{thm:centbin-27}
Let $\Psi(z)$ be given by \eqref{eq:Psidef}.
Then, we have
\begin{multline} 
\label{eq:LoesCB27}
\sum _{n\ge0} ^{}\binom {2n}n\,z^n
=
\left(9\frac{
   1+z}{1-z}+3\right) \Psi(-z)
-(4z+8)
   \Psi^3(-z)\\
-\left(12z^2+12z+3
\right)
   \Psi^5(-z)\quad 
\text {\em modulo }27.
\end{multline}
\end{theorem}

In the case of modulus 9, this theorem reduces to the following
result.

\begin{corollary}
\label{cor:centbin-9}
Let $\Psi(z)$ be given by \eqref{eq:Psidef}.
Then, we have
\begin{equation} 
\label{eq:LoesCB9}
\sum _{n\ge0} ^{}\binom {2n}n\,z^n
=
-3 \Psi(-z)+(2z+4
) \Psi^3(-z)
\quad \text {\em modulo }9.
\end{equation}
\end{corollary}

These results generalise
Theorem~4.3 in \cite{DeSaAA}, which is an easy consequence.

\medskip
We next turn to sums of central binomial coefficients.
From \eqref{eq:centbinEF} it follows directly that
\begin{equation} \label{eq:sumcentbinEF}
(1-4z)(1-z)^2SCB^2(z)-1=0,
\end{equation}
where
$$
SCB(z):=\sum_{n\ge0}z^n\sum_{k=0}^n\binom {2k}k.
$$

\begin{theorem} \label{thm:sumcentbin}
Let $\Psi(z)$ be given by \eqref{eq:Psidef}, 
and let $\al$ be some positive integer.
Then the generating function $SCB(z)=\sum_{n\ge0}z^n\sum_{k=0}^n\binom {2k}k$ 
for sums of central binomial numbers,
reduced modulo $3^{3^\al}$, 
can be expressed as a polynomial in $\Psi(-z)$ of the form
$$
SCB(z)=a_0(z)+\sum _{i=0} ^{2\cdot 3^{\al}-1}a_{i}(z)\Psi^{i}(-z)\quad 
\text {\em modulo }3^{3^\al},
$$ 
where the coefficients $a_{i}(z)$, $i=0,1,\dots,2\cdot 3^\al-1$, 
are Laurent polynomials in $z$ and $1-z$.
\end{theorem}

\begin{proof} 
Choosing $\ga=1$, $c_2(z)=(1-4z)(1-z)^2$, $c_1(z)=0$, $c_0(z)=-1$, we see that
the functional equation \eqref{eq:sumcentbinEF} fits into the framework
of Theorem~\ref{thm:general}, since 
$$c_1^2(z)-c_0(z)c_2(z)=(1-z)^3\quad \text{modulo }3.$$
The constants $e_1,e_2,f_1,f_2$ are given by
$e_1=0$,
$e_2=3$,
$f_1=0$,
$f_2=1$,
$\ep=-1$, and
we have
$\mathcal Q(\dots)=0$.
Consequently, the generating function $SCB(z)$, when reduced modulo
$3^{3^\al}$, can be written as a polynomial in the basic series
$\Psi(-z)=(1-z)\Psi(-z^3)$. In the formulation of the theorem, we
have chosen to express everything in terms of $\Psi(-z^3)$, since this
leads to the more elegant expressions for the moduli 9 and 27 given below.
\end{proof}

We have implemented the algorithm explained in the above proof.
For $\al=1$, it produces the following result.

\begin{theorem} \label{thm:sumcentbin-27}
Let $\Psi(z)$ be given by \eqref{eq:Psidef}.
Then, we have
\begin{multline} 
\label{eq:LoesSCB27}
\sum _{n\ge0} ^{}z^n\sum_{k=0}^n\binom {2k}k
=
+\left(3+9\frac{
   z+1}{1-z}\right) \Psi(-z^3)
+\left(-4 z^3+12 z+19\right)
   \Psi^3(-z^3)\\
+\left(15 z^6+9
   z^5+15 z^3+18 z^2+24\right)
   \Psi^5(-z^3)
\quad \text {\em modulo }27.
\end{multline}
\end{theorem}

In the case of modulus 9, this theorem reduces to the following
result.

\begin{corollary}
\label{cor:sumcentbin-9}
Let $\Psi(z)$ be given by \eqref{eq:Psidef}.
Then, we have
\begin{equation} 
\label{eq:LoesSCB9}
\sum _{n\ge0} ^{}z^n\sum_{k=0}^n\binom {2k}k
=
-3 \Psi(-z^3) + \left(4 + 3 z + 2 z^3\right) \Psi^3(-z^3)
\quad \text {\em modulo }9.
\end{equation}
\end{corollary}

These results generalise
Theorem~5.7 in \cite{DeSaAA}, which is an easy consequence.

\section{Catalan numbers}
\label{sec:Catalan}

Let $C_n=\frac {1} {n+1}\binom {2n}n$ be the $n$-th {\it Catalan
number}.
It is well-known (cf.\ \cite[equation below (6.18) with
$k=2$]{StanBI}) 
that the generating function $C(z)=\sum_{n\ge0}C_n\,z^n$
satisfies the functional equation
\begin{equation} \label{eq:CatalanEF}
zC^2(z)-C(z)+1=0.
\end{equation}

\begin{theorem} \label{thm:Catalan}
Let $\Psi(z)$ be given by \eqref{eq:Psidef}, 
and let $\al$ be some positive integer.
Then the generating function $C(z)=\sum_{n\ge0}C_nz^n$ 
for the Catalan numbers,
reduced modulo $3^{3^\al}$, 
can be expressed as a polynomial in $\Psi(-z)$ of the form
$$
C(z)=a_0(z)+\sum _{i=0} ^{2\cdot 3^{\al}-1}a_{i}(z)\Psi^{i}(-z)\quad 
\text {\em modulo }3^{3^\al},
$$ 
where the coefficients $a_{i}(z)$, $i=0,1,\dots,2\cdot 3^\al-1$, 
are Laurent polynomials in $z$ and $1-z$.
\end{theorem}

\begin{proof} 
Choosing $\ga=1$, $c_2(z)=z$, $c_1(z)=-1$, $c_0(z)=1$, we see that
the functional equation \eqref{eq:CatalanEF} fits into the framework
of Theorem~\ref{thm:general}, since 
$$c_1^2(z)-c_0(z)c_2(z)=1-z\quad \text{modulo }3.$$
The constants $e_1,e_2,f_1,f_2$ are given by
$e_1=1$,
$e_2=0$,
$f_1=0$,
$f_2=0$,
$\ep=-1$, and
we have
$\mathcal Q(\dots)=0$.
Consequently, the generating function $C(z)$, when reduced modulo
$3^{3^\al}$, can be written as a polynomial in the basic series
$\Psi(-z)$, as claimed.
\end{proof}

We have implemented the algorithm explained in the above proof.
For $\al=1$, it produces the following result.

\begin{theorem} \label{thm:Catalan-27}
Let $\Psi(z)$ be given by \eqref{eq:Psidef}.
Then, we have
\begin{multline} 
\label{eq:LoesC27}
\sum _{n\ge0} ^{}C_n\,z^n
=-13z^{-1}
-3
   \left(4+2z^{-1}\right)
   \Psi(-z)
+\left(-8
   z-14+4z^{-1}\right)
   \Psi^3(-z)\\
+3 \left(z^2-6
   z+9-4z^{-1}\right)
   \Psi^5(-z)
\quad 
\text {\em modulo }27.
\end{multline}
\end{theorem}

In the case of modulus 9, this theorem reduces to the following
result.

\begin{corollary}
\label{cor:Catalan-9}
Let $\Psi(z)$ be given by \eqref{eq:Psidef}.
Then, we have
\begin{equation} 
\label{eq:LoesC9}
\sum _{n\ge0} ^{}C_n\,z^n
=-4z^{-1}
+3
   \left(1-z^{-1}\right)
   \Psi(-z)
+\left(4 z-2-2z^{-1}\right)
   \Psi^3(-z)
\quad \text {\em modulo }9.
\end{equation}
\end{corollary}

These results generalise
Theorem~5.2 in \cite{DeSaAA}, which is an easy consequence.

\section{Central Delannoy numbers}
\label{sec:Delannoy}

Let $D_n$ be the $n$-th {\it central Delannoy number}, that is, the
the number of
lattice paths from $(0,0)$ to $(n,n)$ consisting of steps taken from
the set $\{(1,0),(0,1),(1,1)\}$.
It is well-known (see \cite[Ex.~6.3.8]{StanBI})
that the generating function $D(z)=\sum_{n\ge0}D_n\,z^n$
equals $1/\sqrt{1-6z+z^2}$.
Consequently, $D(z)$ satisfies the
functional equation
\begin{equation} \label{eq:DelannoyEF}
(1-6z+z^2)D^2(z)-1=0.
\end{equation}

\begin{theorem} \label{thm:Delannoy}
Let $\Psi(z)$ be given by \eqref{eq:Psidef}, 
and let $\al$ be some positive integer.
Then the generating function $D(z)=\sum_{n\ge0}D_n\,z^n$ 
for the central Delannoy numbers,
reduced modulo $3^{3^\al}$, 
can be expressed as a polynomial in $\Psi(z^2)$ of the form
$$
D(z)=a_0(z)+\sum _{i=0} ^{2\cdot 3^{\al}-1}a_{i}(z)\Psi^{i}(z^2)\quad 
\text {\em modulo }3^{3^\al},
$$ 
where the coefficients $a_{i}(z)$, $i=0,1,\dots,2\cdot 3^\al-1$, 
are Laurent polynomials in $z$ and $1+z^2$.
\end{theorem}

\begin{proof} 
Choosing $\ga=2$, $c_2(z)=1-6z+z^2$, $c_1(z)=0$, $c_0(z)=-1$, we see that
the functional equation \eqref{eq:DelannoyEF} fits into the framework
of Theorem~\ref{thm:general}, since 
$$c_1^2(z)-c_0(z)c_2(z)=1+z^2\quad \text{modulo }3.$$
The constants $e_1,e_2,f_1,f_2$ are given by
$e_1=0$,
$e_2=1$,
$f_1=0$,
$f_2=0$,
$\ep=1$, and
we have
$\mathcal Q(\dots)=0$.
Consequently, the generating function $D(z)$, when reduced modulo
$3^{3^\al}$, can be written as a polynomial in the basic series
$\Psi(z^2)$, as claimed.
\end{proof}

We have implemented the algorithm explained in the above proof.
For $\al=1$, it produces the following result.

\begin{theorem} \label{thm:Delannoy-27}
Let $\Psi(z)$ be given by \eqref{eq:Psidef}.
Then, we have
\begin{multline} 
\label{eq:LoesD27}
\sum _{n\ge0} ^{}D_n\,z^n
=
\left(3+9\frac{
   (z+2)^2}{1+z^2}\right)
   \Psi(z^2)
+\left(19 z^2+3 z+19\right)
   \Psi^3(z^2)\\
-\left(3 z^4+9 z^3+6
   z^2+9 z+3\right) \Psi^5(z^2)
\quad 
\text {\em modulo }27.
\end{multline}
\end{theorem}

In the case of modulus 9, this theorem reduces to the following
result.

\begin{corollary}
\label{cor:Delannoy-9}
Let $\Psi(z)$ be given by \eqref{eq:Psidef}.
Then, we have
\begin{equation} 
\label{eq:LoesD9}
\sum _{n\ge0} ^{}D_n\,z^n
=-3 \Psi(z^2)
+\left(4 z^2+3 z+4\right)
   \Psi^3(z^2)
\quad \text {\em modulo }9.
\end{equation}
\end{corollary}

These results generalise the part on central Delannoy numbers in
Theorem~5.15 of \cite{DeSaAA}, which is an easy consequence.

\section{Schr\"oder numbers}
\label{sec:Schroeder}

Let $S_n$ be the $n$-th (large) {\it Schr\"oder
number}, that is, the number of
lattice paths from $(0,0)$ to $(2n,0)$ consisting of steps taken from
the set
$\{(2,0),(1,1),(1,-1)\}$ never running below the $x$-axis.
It is well-known (cf.\ \cite[Second Problem on page~178]{StanBI})
that the generating function $S(z)=\sum_{n\ge0}S_n\,z^n$ is given by
$$S(z)=\frac {1-z-\sqrt{1-6z+z^2}} {2z},$$ 
and hence satisfies the
functional equation
\begin{equation} \label{eq:SchroederEF}
zS^2(z)+(z-1)S(z)+1=0.
\end{equation}

\begin{theorem} \label{thm:Schroeder}
Let $\Psi(z)$ be given by \eqref{eq:Psidef}, 
and let $\al$ be some positive integer.
Then the generating function $S(z)=\sum_{n\ge0}S_nz^n$ 
for the Schr\"oder numbers,
reduced modulo $3^{3^\al}$, 
can be expressed as a polynomial in $\Psi(z^2)$ of the form
$$
S(z)=a_0(z)+\sum _{i=0} ^{2\cdot 3^{\al}-1}a_{i}(z)\Psi^{i}(z^2)\quad 
\text {\em modulo }3^{3^\al},
$$ 
where the coefficients $a_{i}(z)$, $i=0,1,\dots,2\cdot 3^\al-1$, 
are Laurent polynomials in $z$ and $1+z^2$.
\end{theorem}

\begin{proof} 
Choosing $\ga=2$, $c_2(z)=z$, $c_1(z)=z-1$, $c_0(z)=1$, we see that
the functional equation \eqref{eq:SchroederEF} fits into the framework
of Theorem~\ref{thm:general}, since 
$$c_1^2(z)-c_0(z)c_2(z)=1+z^2\quad \text{modulo }3.$$
The constants $e_1,e_2,f_1,f_2$ are given by
$e_1=1$,
$e_2=0$,
$f_1=0$,
$f_2=0$,
$\ep=1$, and
we have
$\mathcal Q(\dots)=0$.
Consequently, the generating function $S(z)$, when reduced modulo
$3^{3^\al}$, can be written as a polynomial in the basic series
$\Psi(z^2)$, as claimed.
\end{proof}

We have implemented the algorithm explained in the above proof.
For $\al=1$, it produces the following result.

\begin{theorem} \label{thm:Schroeder-27}
Let $\Psi(z)$ be given by \eqref{eq:Psidef}.
Then, we have
\begin{multline} 
\label{eq:LoesS27}
\sum _{n\ge0} ^{}S_n\,z^n
=
13+14z^{-1}
+\left(21
   z+18+21z^{-1}\right)
   \Psi(z^2)\\
+\left(4 z^3+15 z^2+17
   z+15+4z^{-1}\right)
   \Psi^3(z^2)\\
+\left(15 z^5+9
   z^4+18 z^3+18 z^2+18
   z+9+15z^{-1}\right)
   \Psi^5(z^2)
\quad 
\text {\em modulo }27.
\end{multline}
\end{theorem}

In the case of modulus 9, this theorem reduces to the following
result.

\begin{corollary}
\label{cor:Schroeder-9}
Let $\Psi(z)$ be given by \eqref{eq:Psidef}.
Then, we have
\begin{multline} 
\label{eq:LoesS9}
\sum _{n\ge0} ^{}S_n\,z^n
=4+5z^{-1}
-3
   \left(z+z^{-1}\right)
   \Psi(z^2)
-\left(2 z^3+3 z^2+4
   z+3+2z^{-1}\right)
   \Psi^3(z^2)\\
\quad \text {\em modulo }9.
\end{multline}
\end{corollary}

Schr\"oder numbers have not been considered in \cite{DeSaAA}.

\section{Hex tree numbers}
\label{sec:hex tree}

Let $H_n$ be the $n$-th {\it hex tree number}, that is, 
the number of planar rooted trees where each vertex may have a left, a
middle, or a right descendant, but never a left {\it and\/} middle
descendant, and never a middle {\it and\/} right descendant.
By an elementary combinatorial decomposition, it is easy to see 
that the generating function $H(z)=\sum_{n\ge0}H_n\,z^n$
satisfies the functional equation
\begin{equation} \label{eq:hex treeEF}
z^2H^2(z)+(3z-1)H(z)+1=0.
\end{equation}

\begin{theorem} \label{thm:hex tree}
Let $\Psi(z)$ be given by \eqref{eq:Psidef}, 
and let $\al$ be some positive integer.
Then the generating function $H(z)=\sum_{n\ge0}H_n\,z^n$ 
for the hex tree numbers,
reduced modulo $3^{3^\al}$, 
can be expressed as a polynomial in $\Psi(-z^2)$ of the form
$$
H(z)=a_0(z)+\sum _{i=0} ^{2\cdot 3^{\al}-1}a_{i}(z)\Psi^{i}(-z^2)\quad 
\text {\em modulo }3^{3^\al},
$$ 
where the coefficients $a_{i}(z)$, $i=0,1,\dots,2\cdot 3^\al-1$, 
are Laurent polynomials in $z$ and $1-z^2$.
\end{theorem}

\begin{proof} 
Choosing $\ga=2$, $c_2(z)=z^2$, $c_1(z)=3z-1$, $c_0(z)=1$, we see that
the functional equation \eqref{eq:hex treeEF} fits into the framework
of Theorem~\ref{thm:general}, since 
$$c_1^2(z)-c_0(z)c_2(z)=1-z^2\quad \text{modulo }3.$$
The constants $e_1,e_2,f_1,f_2$ are given by
$e_1=2$,
$e_2=0$,
$f_1=0$,
$f_2=0$,
$\ep=-1$, and
we have
$\mathcal Q(\dots)=0$.
Consequently, the generating function $H(z)$, when reduced modulo
$3^{3^\al}$, can be written as a polynomial in the basic series
$\Psi(z^2)$, as claimed.
\end{proof}

We have implemented the algorithm explained in the above proof.
For $\al=1$, it produces the following result.

\begin{theorem} \label{thm:hex tree-27}
Let $\Psi(z)$ be given by \eqref{eq:Psidef}.
Then, we have
\begin{multline} 
\label{eq:LoesH27}
\sum _{n\ge0} ^{}H_n\,z^n
=12z^{-1}+14z^{-2}
-3
   \left(4+3z^{-1}+2z^{-2}
\right) \Psi(-z^2)\\
+\left(z^2+21
   z+13+15z^{-1}+4z^{-2}\right)
   \Psi^3(-z^2)\\
+3
   \left(z^4+3 z^3+12
   z^2+21
   z+3z^{-1}+14z^{-2}\right)
   \Psi^5(-z^2)
\quad 
\text {\em modulo }27.
\end{multline}
\end{theorem}

In the case of modulus 9, this theorem reduces to the following
result.

\begin{corollary}
\label{cor:hex tree-9}
Let $\Psi(z)$ be given by \eqref{eq:Psidef}.
Then, we have
\begin{multline} 
\label{eq:LoesH9}
\sum _{n\ge0} ^{}H_n\,z^n
=-4z^{-2}
+3
   \left(1+3z^{-1}+2z^{-2}\right)
   \Psi(-z^2)\\
+\left(4 z^2+7+7z^{-2}\right)
   \Psi^3(-z^2)
\quad \text {\em modulo }9.
\end{multline}
\end{corollary}

These results generalise
Theorem~5.3 in \cite{DeSaAA}, which is an easy consequence.

\section{Free subgroup numbers
for lifts of the inhomogeneous modular group}
\label{sec:Free}

Let $f_\la$ be the number of free subgroups of $PSL_2(\Z)$
of index $6\la$. 
More generally, let $f_\la(m)$ denote the number of free subgroups 
of index $6m\la$ in the lift $\Gamma_m(3)$ of $PSL_2(\mathbb Z)$ defined by
$$
\Gamma_m(3) = C_{2m} \underset{C_m}{\ast} C_{3m} = \big\langle
x,y\,\big\vert\, x^{2m} = y^{3m} = 1,\, x^2 = y^3\big\rangle,
\quad m\ge1. 
$$
Clearly, we have $f_\la=f_\la(1)$.

The generating function
$F_m(z)=1+\sum_{\la \ge1}f_\la(m)z^\la$
satisfies the Riccati differential equation
(see \cite[Eq.~(18)]{MuHecke} with $\mathfrak{G} =  \Gamma_m(q)$, 
or \cite[Eq.~(8.1)]{KaKMAA})
\begin{equation} \label{eq:FreeEF}
zF_m^2(z)-(1-(6m-2)z)F_m(z)+6mz^2F_m'(z)
+1+(1-6m+5m^2)z=0.
\end{equation}

It is not hard to see that,
if $m$ is divisible by~$3$, then $v_3(f_\la(m))$ will increase with $\la$,
that is, the sequence $f_\la(m)$ is ultimately zero modulo any given
$3$-power.

\begin{theorem} \label{thm:Free}
Let $\Psi(z)$ be given by \eqref{eq:Psidef}, let $m$ be a positive
integer not divisible by~$3$,
and let $\al$ be some positive integer.
Then the generating function $F_m(z)$
for the free subgroup numbers $f_\la(m)$,
when reduced modulo $3^{3^\al}$, 
can be expressed as a polynomial in $\Psi(z^6)$ of the form
$$
F_m(z)=a_0(z)+\sum _{i=0} ^{2\cdot 3^{\al}-1}a_{i}(z)\Psi^{i}(z^6)\quad 
\text {\em modulo }3^{3^\al},
$$ 
where the coefficients $a_{i}(z)$, $i=0,1,\dots,2\cdot 3^\al-1$, 
are Laurent polynomials in $z$ and $1+z^2$.
\end{theorem}

\begin{proof} 
Choosing $\ga=2$, $c_2(z)=-z$, $c_1(z)=1-(6m-2)z$, $c_0(z)=1-(1-6m+5m^2)z$, we see that
the functional equation \eqref{eq:FreeEF} fits into the framework
of Theorem~\ref{thm:general}, since 
$$c_1^2(z)-c_0(z)c_2(z)=1+z^2\quad \text{modulo }3.$$
The constants $e_1,e_2,f_1,f_2$ are given by
$e_1=1$,
$e_2=0$,
$f_1=0$,
$f_2=0$,
$\ep=1$, and
we have
$\mathcal Q(\cdots)=6mz^2F_m'(z)$.
Consequently, the generating function $F_m(z)$, when reduced modulo
$3^{3^\al}$, can be written as a polynomial in the basic series
$\Psi(z^2)=(1+z^2)\Psi(z^6)$. In the formulation of the theorem, we
have chosen to express everything in terms of $\Psi(z^6)$, since this
leads to the more elegant expressions for the moduli 9 and 27 given below.
\end{proof}

We have implemented the algorithm explained in the above proof.
For $\al=1$, it produces the following result.

\begin{theorem} \label{thm:Free-27}
Let $\Psi(z)$ be given by \eqref{eq:Psidef},
Then, for all positive integers $m\equiv 1~\text{\em(mod~$3$)}$, we have
\begin{multline} 
\label{eq:Loes1m1}
1 + \sum _{\la\ge1} ^{}f_\lambda(m)\,z^\lambda
=
1-z^{-1}-3 \frac{(z+2)
   \left(z^2+2 z+2\right)}{z
   \left(z^2+1\right)}
  -\frac{9+3(m-1) z^4
   }{\left(z^2+1\right)^2}\\
+\left(3m z^3+(3m+21)
   z+21 z^{-1}\right)
   \Psi(z^6)\\
+\left((-2m+18) z^9+(-m^2+6m+18) z^7+(-3m+25)
   z^5\right.\kern2cm\\
\kern2cm
\left.+(m+15) z^3+(2m^2-3m+6)
   z+4 z^{-1}\right)
   \Psi^3(z^6)\\
+\left((-3m+9) z^{15}+(-3m+6)
   z^{13}+15 z^{11}+(3m+9) z^9+(3m+3)
   z^7\right.\kern2cm\\
\kern2cm\left.
+3 z^5+(-3m+9) z^3+(-3m+6)
   z+15 z^{-1}\right)
   \Psi^5(z^6)\quad 
\text {\em modulo }27.
\end{multline}
Furthermore, 
for all positive integers $m\equiv 2~\text{\em(mod $3$)}$, we have
\begin{multline} 
\label{eq:Loes1m2}
1 + \sum _{\la\ge1} ^{}f_\lambda(m)\,z^\lambda
=
1-z^{-1}+3 \frac{(z+1)
   \left(z^2+ z+2\right)}{z
   \left(z^2+1\right)}\\
  +\frac{9+18z+18z^2+9z^3+9z^4-(3m-24) z^5
   }{\left(z^2+1\right)^2}\\
+\left(-3m z^3+(-3m+21)
   z+21 z^{-1}\right)
   \Psi(z^6)\\
+\left((15m^2-4m+24) z^9+(-m^2+3m) z^7+(3m-2)
   z^5\right.\kern2cm\\
\kern2cm
\left.+(-m+15) z^3+(14m^2+18)
   z+4 z^{-1}\right)
   \Psi^3(z^6)\\
+\left((3m+9) z^{15}+(3m+6)
   z^{13}+15 z^{11}+(-3m+9) z^9+(-3m+3)
   z^7\right.\kern2cm\\
\kern2cm\left.
+3 z^5+(3m+9) z^3+(3m+6)
   z+15 z^{-1}\right)
   \Psi^5(z^6)\quad 
\text {\em modulo }27.
\end{multline}
\end{theorem}

In the case of modulus 9, this theorem reduces to the following
result.

\begin{corollary} \label{thm:Free-9}
Let $\Psi(z)$ be given by \eqref{eq:Psidef},
\!Then, for all positive integers $m\equiv 1~\text{\em(mod~$3$)}$, we have
\begin{multline} 
\label{eq:Loes1m1-9}
1 + \sum _{\la\ge1} ^{}f_\lambda(m)\,z^\lambda
=
1-4z^{-1}-3 \frac{z^2}{
   z^2+1}
-3z^{-1}\left(z^2+1\right)^2
   \Psi(z^6)\\
+\left(
m z^9+(m+1) z^7+7
   z^5+m
   z^3+(m+1) z+7z^{-1}\right) \Psi^3(z^6)
\quad 
\text {\em modulo }9.
\end{multline}
Furthermore, 
for all positive integers $m\equiv 2~\text{\em(mod $3$)}$, we have
\begin{multline} 
\label{eq:Loes1m2-9}
1 + \sum _{\la\ge1} ^{}f_\lambda(m)\,z^\lambda
=
1+5z^{-1}+3\frac{z^2}{z^2+1}
-3z^{-1}(1+z^2)^2
   \Psi(z^6)\\
-
(m z^9+(m-1) z^7+2
   z^5+m
   z^3+(m-1) z+2z^{-1})
   \Psi^3(z^6)\quad 
\text {\em modulo }9.
\end{multline}
\end{corollary}

We have chosen this section to demonstrate that generating function
results such as the ones in Theorem~\ref{thm:Free-27} and
Corollary~\ref{thm:Free-9}, in combination
with the coefficient extraction procedure described in the
appendix, lead naturally to an explicit description of the congruence
classes modulo any given $3$-power 
of the elements of the sequence that we are considering.
We illustrate this point with the sequence
of free subgroup numbers $f_\la=f_\la(1)$ of the inhomogeneous
modular group $PSL_2(z)$.
In the statement of the result below, we use notation from
word calculus: by $S^*$, where $S$ is a set of letters, we
mean the set of all words with letters from $S$. By a
slight misuse of notation, the symbol $a^*$, where $a$ is
a letter, will denote the set $\{\ep,a,aa,aaa,\dots\}$,
with $\ep$ standing for the empty word. Finally, if we write
$S_1S_2$, where $S_1$ and $S_2$ are two sets of words, then we
mean the set of all concatenations $w_1w_2$ with $w_1\in S_1$
and $w_2\in S_2$. Since our words represent ternary strings, 
we make the convention that $0$'s at the beginning of a word
may be added or deleted at will. (The point of view which we
shall adopt later in the statements of Theorem~\ref{thm:Free27},
Proposition~\ref{prop:Psi-3}, Corollary~\ref{cor:Psi-3},
Proposition~\ref{prop:Psi-5}, and Corollary~\ref{cor:Psi-5}
is to consider these strings as left-infinite strings, where
one has added infinitely many $0$'s on the left of a finite
word.)

\begin{theorem} \label{thm:Free9}
We have:
\begin{enumerate} 
\item $f_\la\equiv 1~(\text{\em mod }9)$, if, and only if, 
the $3$-adic expansion of $n$ is an element of
\begin{multline*}
\kern2cm
\{0\}\cup
\bigcup_{k\ge0}\big(00^*22^*\big)^{3k}0^*10^*00\cup
\bigcup_{k\ge0}\big(00^*22^*\big)^{3k+2}0^*10\\
\cup
\bigcup_{k\ge0}\big(22^*00^*\big)^{3k+2}2^*12\cup
\bigcup_{k\ge0}\big(22^*00^*\big)^{3k}2^*12^*22;
\end{multline*}
\item $f_\la\equiv 4~(\text{\em mod }9)$, if, and only if, 
the $3$-adic expansion of $n$ is an element of
\begin{multline*}
\kern2cm
\bigcup_{k\ge0}\big(00^*22^*\big)^{3k+2}0^*10^*00\cup
\bigcup_{k\ge0}\big(00^*22^*\big)^{3k+1}0^*10\\
\cup
\bigcup_{k\ge0}\big(22^*00^*\big)^{3k+1}2^*12\cup
\bigcup_{k\ge0}\big(22^*00^*\big)^{3k+2}2^*12^*22;
\end{multline*}
\item $f_\la\equiv 7~(\text{\em mod }9)$, if, and only if, 
the $3$-adic expansion of $n$ is an element of
\begin{multline*}
\kern2cm
\{10\}\cup
\bigcup_{k\ge0}\big(00^*22^*\big)^{3k+1}0^*10^*00\cup
\bigcup_{k\ge0}\big(00^*22^*\big)^{3k}0^*10\\
\cup
\bigcup_{k\ge0}\big(22^*00^*\big)^{3k}2^*12\cup
\bigcup_{k\ge0}\big(22^*00^*\big)^{3k+1}2^*12^*22;
\end{multline*}
\item $f_\la\equiv 2~(\text{\em mod }9)$, if, and only if, 
the $3$-adic expansion of $n$ is an element of
$$
\kern2cm
\bigcup_{k\ge0}\big(00^*22^*\big)^{3k+2}0^*01\cup
\bigcup_{k\ge0}\big(00^*22^*\big)^{3k+1}1;
$$
\item $f_\la\equiv 5~(\text{\em mod }9)$, if, and only if, 
the $3$-adic expansion of $n$ is an element of
$$
\kern2cm
\bigcup_{k\ge0}\big(00^*22^*\big)^{3k}0^*01\cup
\bigcup_{k\ge0}\big(00^*22^*\big)^{3k+2}1;
$$
\item $f_\la\equiv 8~(\text{\em mod }9)$, if, and only if, 
the $3$-adic expansion of $n$ is an element of
$$
\kern2cm
\bigcup_{k\ge0}\big(00^*22^*\big)^{3k+1}0^*01\cup
\bigcup_{k\ge0}\big(00^*22^*\big)^{3k+3}1;
$$
\item $f_\la\equiv 3~(\text{\em mod }9)$, if, and only if, 
$\la\equiv 0~\text{\em(mod~$4$)}$ or
the $3$-adic expansion of $n$ is an element of
$$
\{0,2\}^*11\{0,2\}^*1,
$$
\item $f_\la\equiv 6~(\text{\em mod }9)$, if, and only if, 
$\la\equiv 2~\text{\em(mod~$4$)}$ or
the $3$-adic expansion of $n$ is an element of
$$
\kern1cm
\{0,2\}^*1022^*\cup
\{0,2\}^*1200^*\cup
\{0,2\}^*11\{0,2\}^*100^*\cup
\{0,2\}^*11\{0,2\}^*122^*.
$$
\item In all other cases, this coefficient is divisible by $9$.
\end{enumerate}
\end{theorem}

\begin{proof}
One performs a (lengthy) case analysis using 
Corollary~\ref{thm:Free-9} with $m=1$ in combination with 
the descriptions
of the coefficients of $\Psi^3(z)$ modulo~$9$, which one obtains by
carrying out the approach sketched in Section~\ref{sec:extr}
(see in particular the paragraphs after \eqref{eq:Psipoteven}
and \eqref{eq:Psipotodd}) and the appendix, the resulting congruences being
explicitly displayed 
in Corollary~\ref{cor:Psi-3}.
\end{proof}

The next corollary corrects Theorem~3(i) in \cite{MuPu}.

\begin{corollary} \label{cor:1}
We have:
\begin{enumerate} 
\item $f_\la\equiv-1~(\text{\em mod }3)$ if, and only if, 
the $3$-adic expansion of $\la$ is an element of
$\{0,2\}^*1$;
\item $f_\la\equiv1~(\text{\em mod }3)$ if, and only if, 
the $3$-adic expansion of $\la$ is an element of
$$
\{0,2\}^*100^*\cup \{0,2\}^*122^*;
$$
\item for all other $\la$, we have $f_\la\equiv0~(\text{\em mod }3)$.
\end{enumerate}
\end{corollary}

\begin{remark}
The conditions in Items~(1) and (2) imply that the $3$-adic expansion of
$\la$ always contains exactly one digit $1$.
In Item~(1), this $1$ is the first
digit from the right. 
In Item~(2), this $1$ is not the right-most digit,
and it has only $0$'s or only $2$'s to its right.
\end{remark}

With some additional effort, by using Theorem~\ref{thm:Free-27} with $m=1$
and the coefficient congruences modulo~27 for $\Psi^3(z)$ and $\Psi^5(z)$
worked out in
Propositions~\ref{prop:Psi-3} and \ref{prop:Psi-5},
one is also able to refine the result in Theorem~\ref{thm:Free9}
to a description of the congruence classes of $f_\la$ modulo~27.
The outcome is lengthy. As an illustration, we provide the
explicit characterisation of all $\la$ for which $f_\la\equiv1$~(mod~27).
For the statement of the result, we need to introduce some statistics
on strings consisting of $0$'s, $1$'s, and $2$'s. 
Given a string $s=\dots s_2s_1s_0$, we define:
\begin{align*}
\EEstring(s)&=\#(\text{maximal strings of $2$'s in $s$}),\\
\Iso_3(s)&=\#(\text{isolated $0$'s and $2$'s in $s$, excluding $s_0,s_1,s_2$}),\\
\#022(s)&=\#(\text{occurrences of the substring $022$ in $s$}),
\end{align*}
and similarly for $\#200(s)$, $\#201(s)$, $\#21(s)$, $\#021(s)$. 
Here, by an isolated letter, we mean a letter $a$ whose neighbours are
different from $a$. 
We use again the word notation explained before Theorem~\ref{thm:Free9}.
In particular, we consider all strings as infinite strings with
infinitely many $0$'s added to the left, so that $\#022(.)$ also
counts an occurrence of $22$ at the beginning of a(n originally finite)
string.

\begin{theorem} \label{thm:Free27}
We have
$f_\la\equiv 1~(\text{\em mod }27)$, if, and only if, 
the $3$-adic expansion of $n$ is an element of
$\{0,2\}^*0100$  and satisfies
$\EEstring(s)\equiv0~(\text{\em mod }3)$ and
$$\tfrac {2\EEstring(s)} {3}-\Iso_3(s)+\#022(s)+\#200(s)
+\#201(s)
\equiv0~(\text{\em mod }3),
$$
or of\/ 
$\{0,2\}^*2100$ and satisfies
$\EEstring(s)\equiv0~(\text{\em mod }3)$ and
$$\tfrac {2\EEstring(s)} {3}-\Iso_3(s)+\#022(s)+\#200(s)
\equiv1~(\text{\em mod }3),
$$
or of\/ $\{0,2\}^*10^*0000$ and satisfies
$\EEstring(s)\equiv0~(\text{\em mod }3)$ and
$$\tfrac {2\EEstring(s)} {3}-\Iso_3(s)-\#21(s)+\#022(s)+\#200(s)
+\#201(s)
\equiv2~(\text{\em mod }3),
$$
or of\/
$\{0,2\}^*1000$ and satisfies
$\EEstring(s)\equiv0~(\text{\em mod }3)$ and
$$\tfrac {2\EEstring(s)} {3}-\Iso_3(s)-\#21(s)+\#022(s)+\#200(s)
+\#201(s)
\equiv1~(\text{\em mod }3),
$$
or of\/
$\{0,2\}^*2210$ and satisfies
$\EEstring(s)\equiv2~(\text{\em mod }3)$ and
$$\tfrac {2\EEstring(s)-1} {3}-\Iso_3(s)+\#022(s)+\#200(s)
\equiv0~(\text{\em mod }3),
$$
or of\/
$\{0,2\}^*0210$ and satisfies
$\EEstring(s)\equiv2~(\text{\em mod }3)$ and
$$\tfrac {2\EEstring(s)-1} {3}-\Iso_3(s)+\#022(s)+\#200(s)
\equiv1~(\text{\em mod }3),
$$
or of\/
$\{0,2\}^*010$ and satisfies
$\EEstring(s)\equiv2~(\text{\em mod }3)$ and
$$\tfrac {2\EEstring(s)-1} {3}-\Iso_3(s)+\#022(s)+\#200(s)
\equiv2~(\text{\em mod }3),
$$
or of\/
$\{0,2\}^*0010$ and satisfies
$\EEstring(s)\equiv2~(\text{\em mod }3)$ and
$$\tfrac {2\EEstring(s)-1} {3}-\Iso_3(s)+\#022(s)+\#200(s)
\equiv0~(\text{\em mod }3),
$$
or of\/
$\{0,2\}^*0012$ and satisfies
$\EEstring(s)\equiv0~(\text{\em mod }3)$ and
$$\tfrac {2\EEstring(s)-3} {3}-\Iso_3(s)+\#022(s)+\#200(s)
\equiv0~(\text{\em mod }3),
$$
or of\/
$\{0,2\}^*0212$ and satisfies
$\EEstring(s)\equiv1~(\text{\em mod }3)$ and
$$\tfrac {2\EEstring(s)-2} {3}-\Iso_3(s)+\#022(s)+\#200(s)
\equiv0~(\text{\em mod }3),
$$
or of\/
$\{0,2\}^*2012$ and satisfies
$\EEstring(s)\equiv0~(\text{\em mod }3)$ and
$$\tfrac {2\EEstring(s)-3} {3}-\Iso_3(s)+\#022(s)+\#200(s)
\equiv1~(\text{\em mod }3),
$$
or of\/
$\{0,2\}^*2212$ and satisfies
$\Estring(s)\equiv1~(\text{\em mod }3)$ and
$$\tfrac {2\Estring(s)-2} {3}-\Iso_3(s)+\#022(s)+\#200(s)
\equiv1~(\text{\em mod }3),
$$
or of\/
$\{0,2\}^*0122$ and satisfies
$\EEstring(s)\equiv1~(\text{\em mod }3)$ and
$$\tfrac {2\EEstring(s)-2} {3}-\Iso_3(s)+\#022(s)+\#200(s)
\equiv0~(\text{\em mod }3),
$$
or of\/
$\{0,2\}^*2122$ and satisfies
$\EEstring(s)\equiv2~(\text{\em mod }3)$ and
$$\tfrac {2\EEstring(s)-4} {3}-\Iso_3(s)+\#022(s)+\#200(s)
-\#021(s)
\equiv2~(\text{\em mod }3),
$$
or of\/
$\{0,2\}^*01222$ and satisfies
$\EEstring(s)\equiv1~(\text{\em mod }3)$ and
$$\tfrac {2\EEstring(s)-2} {3}-\Iso_3(s)+\#022(s)+\#200(s)
\equiv2~(\text{\em mod }3),
$$
or of\/
$\{0,2\}^*21222$ and satisfies
$\EEstring(s)\equiv2~(\text{\em mod }3)$ and
$$\tfrac {2\EEstring(s)-1} {3}-\Iso_3(s)+\#022(s)+\#200(s)
-\#021(s)
\equiv2~(\text{\em mod }3),
$$
or of\/
$\{0,2\}^*012^*2222$ and satisfies
$\EEstring(s)\equiv1~(\text{\em mod }3)$ and
$$\tfrac {2\EEstring(s)-2} {3}-\Iso_3(s)+\#022(s)+\#200(s)
\equiv2~(\text{\em mod }3),
$$
or of\/
$\{0,2\}^*212^*2222$ and satisfies
$\EEstring(s)\equiv2~(\text{\em mod }3)$ and
$$\tfrac {2\EEstring(s)-1} {3}-\Iso_3(s)+\#022(s)+\#200(s)
-\#021(s)
\equiv2~(\text{\em mod }3).
$$
\end{theorem} 

\section{Central Eulerian numbers}
\label{sec:Eulerian}

The {\it Eulerian number} $A(n,k)$ is defined as the number of
permutations of $\{1,2,\dots,n\}$ with exactly $k-1$ descents.
It is well-known that
\begin{equation} \label{eq:Ank}
A(n,k)=\sum _{j=0} ^{k}(-1)^{k-j}\binom {n+1}{k-j}j^n;
\end{equation}
see \cite[Eq.~(1.37)]{StanAP} (one has to multiply 
both sides of that equation by $(1-x)^{d+1}$ and compare coefficients
of powers of $x$).
In this section, we analyse the behaviour of {\it central Eulerian
numbers}, that is, the numbers $A(2n,n)=A(2n,n+1)$ and $A(2n-1,n)$,
modulo powers of~$3$. Theorems~\ref{thm:Euler1} and
\ref{thm:Euler2} below say that, again,
the generating functions for central Eulerian numbers, when
coefficients are reduced modulo a given power of~$3$, can be written
in terms of a polynomial in the basic series $\Psi(z)$. 
However, the arguments establishing this claim are much more
elaborate. The principal reason for this is that 
it is (likely) not possible to find a {\it single}
differential equation of the form \eqref{eq:diffeq}, which we could
use as starting point for the application of our method. Rather,
we have to first
reduce the central Eulerian numbers partially modulo the relevant
$3$-power, and {\it only} then do we succeed to show 
that the generating function for the
reduced numbers satisfies a functional equation of the type
\eqref{eq:generalEF} so that Theorem~\ref{thm:general} can be
applied. In order to accomplish this, we need a number of auxiliary
results.

\begin{lemma} \label{lem:E1}
For any positive integer $S$, we have the
{\em(}differential\/{\em)} operator equation
\begin{equation} \label{eq:E1}
\left((x+1)\frac {d} {dx}\right)^S=\sum _{m=1} ^{S}\frac {1} {m!}
\left(\sum _{k=1} ^{m}(-1)^{m-k}\binom mk k^S\right)
(x+1)^m\frac {d^m} {dx^m}.
\end{equation}
\end{lemma}

\begin{proof}We proceed by induction on $S$. Clearly,
Equation~\eqref{eq:E1} holds for $S=1$. For the induction step,
we compute
\begin{align}
\notag
\left((x+1)\frac {d} {dx}\right)^{S+1}&=
(x+1)\frac {d} {dx}
\sum _{m=1} ^{S}\frac {1} {m!}
\left(\sum _{k=1} ^{m}(-1)^{m-k}\binom mk k^S\right)
(x+1)^m\frac {d^m} {dx^m}\\
\notag
&=
\sum _{m=1} ^{S}\frac {1} {m!}
\left(\sum _{k=1} ^{m}(-1)^{m-k}\binom mk k^S\right)
m(x+1)^m\frac {d^m} {dx^m}\\
\notag
&\kern2cm
+\sum _{m=1} ^{S}\frac {1} {m!}
\left(\sum _{k=1} ^{m}(-1)^{m-k}\binom mk k^S\right)
(x+1)^{m+1}\frac {d^{m+1}} {dx^{m+1}}\\
\notag
&=
\sum _{m=1} ^{S}\frac {1} {m!}
\left(\sum _{k=1} ^{m}(-1)^{m-k}
m\left(\binom mk-\binom{m-1}k\right)k^S\right)
(x+1)^m\frac {d^m} {dx^m}\\
\notag
&\kern2cm
+\frac {1} {S!}
\left(\sum _{k=1} ^{S}(-1)^{S-k}\binom Sk k^S\right)
(x+1)^{S+1}\frac {d^{S+1}} {dx^{S+1}}\\
\notag
&=
\sum _{m=1} ^{S}\frac {1} {m!}
\left(\sum _{k=1} ^{m}(-1)^{m-k}
m\binom{m-1}{k-1}k^S\right)
(x+1)^m\frac {d^m} {dx^m}\\
\notag
&\kern2cm
+\frac {1} {S!}
\left(\sum _{k=1} ^{S}(-1)^{S-k}\binom Sk k^S\right)
(x+1)^{S+1}\frac {d^{S+1}} {dx^{S+1}}\\
\notag
&=\sum _{m=1} ^{S}\frac {1} {m!}
\left(\sum _{k=1} ^{m}(-1)^{m-k}
\binom{m}{k}k^{S+1}\right)
(x+1)^m\frac {d^m} {dx^m}\\
&\kern2cm
+\frac {1} {S!}
\left(\sum _{k=1} ^{S}(-1)^{S-k}\binom Sk k^S\right)
(x+1)^{S+1}\frac {d^{S+1}} {dx^{S+1}}.
\label{eq:DeSk}
\end{align}
At this point we use again 
the shift operator $E$ and
the forward difference operator $\De$ defined 
in the proof of Lemma~\ref{lem:diff}. The symbol $I$ stands for
the identity operator.
In terms of these operators, we may write
\begin{align*}
\frac {1} {S!}
\left(\sum _{k=1} ^{S}(-1)^{S-k}\binom Sk k^S\right)
&=
\frac {1} {S!}
\left(\sum _{k=1} ^{S}(-1)^{S-k}\binom Sk E^k y^S\Big\vert_{y=0}\right)\\
&=\frac {1} {S!}(E-I)^S
y^S\Big\vert_{y=0}=\frac {1} {S!}\De^S
y^S\Big\vert_{y=0}=1,
\end{align*}
the last equality being due to the fact that the difference operator
$\De$ lowers the degree of the polynomial to which it is applied
by $1$, and it multiplies the leading coefficient by the degree
of the corresponding monomial. But then
$$
\frac {1} {S!}
\left(\sum _{k=1} ^{S}(-1)^{S-k}\binom Sk k^S\right)
=
\frac {1} {(S+1)!}
\left(\sum _{k=1} ^{S+1}(-1)^{S+1-k}\binom {S+1}k k^{S+1}\right),
$$
and if this is substituted back into the last line of \eqref{eq:DeSk},
then we get the required result.
\end{proof}

\begin{lemma} \label{lem:E3}
For any positive integer $M$, we have
\begin{equation} \label{eq:subs1} 
\left(-1+\sqrt{1+4z}\right)^M
-\left(-1-\sqrt{1+4z}\right)^M
=(-1)^{M+1}2^{M}\sqrt{1+4z}
\sum _{\ell\ge0} ^{}\binom {M-\ell-1}\ell z^\ell
\end{equation}
and
\begin{equation} \label{eq:subs2} 
\left(-1+\sqrt{1+4z}\right)^M
+\left(-1-\sqrt{1+4z}\right)^M
=(-1)^{M}2^{M}
\sum _{\ell\ge0} ^{}\frac {M} {M-\ell}\binom {M-\ell}\ell z^\ell.
\end{equation}
\end{lemma}

\begin{remark}
The sums over $\ell$ on the right-hand sides of \eqref{eq:subs1} and
\eqref{eq:subs2} are essentially Chebyshev polynomials of the second
and first kind, respectively (cf.\ \cite{RivlAA}). 
It is easy to see (and well-known) that
they are polynomials in $z$ with integer coefficients.
\end{remark}

\begin{proof}[Proof of Lemma~{\em\ref{lem:E3}}]
We start with the proof of \eqref{eq:subs1}.
According to the binomial theorem, the left-hand side can be written as
\begin{align*} 
\left(-1+\sqrt{1+4z}\right)^M&
-\left(-1-\sqrt{1+4z}\right)^M\\
&=
(-1)^ M\sum _{j=0} ^{M}\binom Mj\left((-1)^j(1+4z)^{j/2}-(1+4z)^{j/2}\right)\\
&=
-2(-1)^ M\sum _{j\ge0} ^{}\binom M{2j+1}(1+4z)^{j+\frac {1} {2}}\\
&=
-2(-1)^ M\sqrt{1+4z}\sum _{j\ge0} ^{}
\sum _{\ell=0} ^{j}\binom M{2j+1}\binom j\ell 4^\ell z^\ell.
\end{align*}
Using the standard hypergeometric notation
$$
{}_p F_q\!\left[\begin{matrix} a_1,\dots,a_p\\
b_1,\dots,b_q\end{matrix}; z\right]=\sum _{m=0} ^{\infty}\frac
{\po{a_1}{m}\cdots\po{a_p}{m}} {m!\,\po{b_1}{m}\cdots\po{b_q}{m}}
z^m ,
$$
where, as before, the Pochhammer symbol is defined by
$(\alpha)_m=\alpha(\alpha+1) \cdots(\alpha+m-1)$ for $m\ge1$, and
$(\alpha)_0=1$, we have
$$
\sum _{j \ge0} ^{}\binom M{2j+1}\binom j\ell=
\binom M{2\ell+1}
{}_2 F_1\!\left[\begin{matrix} \frac {1} {2}+\ell-\frac {M} {2},1+\ell-\frac {M} {2}\\
\frac {3} {2}+\ell\end{matrix}; 1\right].
$$
This $_2F_1$-series can be summed by means of Gau{\ss}' summation
formula 
(see \cite[(1.7.6); Appendix (III.3)]{SlatAC})
\begin{equation*} 
{} _{2} F _{1} \!\left [ \begin{matrix} { a, b}\\ { c}\end{matrix} ; {\displaystyle
   1}\right ]  = \frac {\Gamma ( c)\,\Gamma( c-a-b)} {\Gamma( c-a)\,
  \Gamma( c-b)},
\end{equation*}
where $\Gamma(\,.\,)$ denotes the gamma function.
After some simplification, this leads to
$$
\sum _{j \ge0} ^{}\binom M{2j+1}\binom j\ell=
2^{M-2\ell-1}\binom {M-\ell-1}\ell.
$$

The proof of \eqref{eq:subs2} is completely analogous, and is
left to the reader.
\end{proof}

\begin{proposition} \label{lem:E2}
For any positive integer $s$, we have
\begin{equation} \label{eq:E2}
\sum _{n\ge0} ^{}z^n
\sum _{j=0} ^{n+1}(-1)^{n+1-j}\binom {2n+1}{n+1-j}j^{2s}
=\frac {1} {2}\left(1+\sqrt{1+4z}\right)
\left(1+p_s^{(1)}(z)\right),
\end{equation}
where $p_s^{(1)}(z)$ is a polynomial in $z$ with integer coefficients
all of which are divisible by~$3$, and which satisfies $p_s^{(1)}(0)=0$.
An explicit expression for $p_s^{(1)}(z)$ can be derived from
\eqref{eq:q1} and \eqref{eq:p1q1}.
\end{proposition}

\begin{proof}
We write the binomial coefficient in terms of a contour
integral, 
$$
(-1)^{n+1-j}\binom {2n+1}{n+1-j}=\frac {1} {2\pi i}
\int _{C} ^{}\frac {(1-t)^{2n+1}} {t^{n+2-j}}\,dt,
$$
where $C$ is the circle of radius $1/2$ (say) about the origin,
encircling the origin in positive direction. Moreover, we express
$j^{2s}$ as 
$$j^{2s}=\left((x+1)\frac {d} {dx}\right)^{2s}(x+1)^j\bigg\vert_{x=0}.$$
Using these substitutions, the 
series on the left-hand side of \eqref{eq:E2} becomes
\begin{align*}
\sum _{n\ge0} ^{}z^n&
\sum _{j=0} ^{n+1}(-1)^{n+1-j}\binom {2n+1}{n+1-j}j^{2s}\\
&=\frac {1} {2\pi i}
\int _{C} ^{}
\sum _{n\ge0} ^{}z^n
\sum _{j=0} ^{n+1}
\frac {(1-t)^{2n+1}} {t^{n+2-j}}
\left((x+1)\frac {d} {dx}\right)^{2s}(x+1)^j\bigg\vert_{x=0}
\,dt\\
&=\frac {1} {2\pi i}
\int _{C} ^{}\frac {1-t} {t^2}
\sum _{n\ge0} ^{}\left(\frac {(1-t)^2z} {t}\right)^n
\left((x+1)\frac {d} {dx}\right)^{2s}
\sum _{j=0} ^{n+1}
\big(t(x+1)\big)^j\bigg\vert_{x=0}
\,dt\\
&=\frac {1} {2\pi i}
\int _{C} ^{}\frac {1-t} {t^2\left(1-\frac {(1-t)^2z} {t}\right)}
\left((x+1)\frac {d} {dx}\right)^{2s}
\frac {1} {1-(x+1)t}
\bigg\vert_{x=0}
\,dt.
\end{align*}
Here, in order to perform the summation over $n$, we had to assume
that $z$ is sufficiently small, say $\vert z\vert<1/9$, which we shall
do for the moment.

In the last integral, we do the substitution $t=u/(1+u)^2$. 
After some simplification, this leads to
\begin{align*}
\sum _{n\ge0} ^{}z^n&
\sum _{j=0} ^{n+1}(-1)^{n+1-j}\binom {2n+1}{n+1-j}j^{2s}\\
&=\frac {1} {2\pi i}
\int _{C} ^{}\frac {1+u} {u
\left(u^2+u-{z} \right)}
\left((x+1)\frac {d} {dx}\right)^{2s}
\frac {1} {1+u-(x+1)u}
\bigg\vert_{x=0}
\,du\\
&=\frac {1} {2\pi i}
\int _{C} ^{}\frac {1+u} {u
\left(u^2+u-{z} \right)}
\sum _{m=1} ^{2s}\frac {1} {m!}
\left(\sum _{k=1} ^{m}(-1)^{m-k}\binom mk k^{2s}\right)\\
&\kern5cm
\cdot\left(
(x+1)^m\frac {d^m} {dx^m}
\frac {1} {1+u-(x+1)u}\right)
\bigg\vert_{x=0}
\,du\\
&=\frac {1} {2\pi i}
\int _{C} ^{}\frac {1+u} {u^2+u-{z} }
\sum _{m=1} ^{2s}
\left(\sum _{k=1} ^{m}(-1)^{m-k}\binom mk k^{2s}\right)u^{m-1}
\,du,
\end{align*}
where we have used Lemma~\ref{lem:E1} to obtain the next-to-last line.

The integral can easily be evaluated using the residue theorem.
Inside the contour $C$, there is only one singularity, namely
a simple pole at $u=-\frac {1} {2}+\frac {1} {2}\sqrt{1+4z}$.
(The other pole, at $u=-\frac {1} {2}-\frac {1} {2}\sqrt{1+4z}$,
lies outside the contour $C$.)
Hence, we obtain
\begin{multline} \label{eq:res}
\sum _{n\ge0} ^{}z^n
\sum _{j=0} ^{n+1}(-1)^{n+1-j}\binom {2n+1}{n+1-j}j^{2s}\\
=
\frac {1+\sqrt{1+4z}} {2\sqrt{1+4z}}
\bigg(\sum _{m=1} ^{2s}
\left(\sum _{k=1} ^{m}(-1)^{m-k}\binom mk k^{2s}\right)u^{m-1}
\bigg)\bigg\vert_{u=-\frac {1} {2}+\frac {1} {2}\sqrt{1+4z}}.
\end{multline}
By analytic continuation, we can now get rid of the restriction
$\vert z\vert<1/9$, and, by the transfer principle between analytic
and formal power series, we may argue that Identity~\eqref{eq:res} 
holds on the level of formal power series.

Let us write the polynomial in $u$ given by the sum over $m$ on the
right-hand side of \eqref{eq:res} as a polynomial in $2u+1$, say
\begin{equation} \label{eq:q1}
q_s^{(1)}(2u+1):=\sum _{m=1} ^{2s}
\left(\sum _{k=1} ^{m}(-1)^{m-k}\binom mk k^{2s}\right)u^{m-1}.
\end{equation}
We claim that $q_s^{(1)}(\,.\,)$ is an odd polynomial, that is,
$q_s^{(1)}(-2u-1)=-q_s^{(1)}(2u+1)$. In order to see this, we use again
classical difference operator calculus. As before, we denote by
$E$ the shift operator and
by $\De$ the forward difference operator (see the proof of
Lemma~\ref{lem:E1}). Furthermore,
we let $\Na$ be the backward difference operator 
defined by $\Na f(y)=f(y-1)-f(y)$. The reader should observe that
$\De=E-I$ and $\Na=E^{-1}-I$, where $I$ stands for the identity
operator.
In terms of these operators, we may rewrite $q_s^{(1)}(2u+1)$ as
\begin{align*} 
q_s^{(1)}(2u+1)&=\sum _{m=1} ^{2s}
\sum _{k=1} ^{m}(-1)^{m-k}\binom mk E^k
y^{2s}\bigg\vert_{y=0}u^{m-1}\\
&=\sum _{m=1} ^{2s}
\De^m y^{2s}\bigg\vert_{y=0}u^{m-1}\\
&=\De(I-u\De)^{-1}(I-(u\De)^{2s}) y^{2s}\bigg\vert_{y=0}\\
&=\De(I-u\De)^{-1} y^{2s}\bigg\vert_{y=0},
\end{align*}
where the last line is due to the standard fact that $\De^a y^b=0$
for all non-negative integers $a$ and $b$ with $a>b$.
Consequently, for $q_s^{(1)}(-2u-1)$ we obtain
\begin{align*} 
q_s^{(1)}(-2u-1)
&=\De(I+(u+1)\De)^{-1} y^{2s}\bigg\vert_{y=0}\\
&=\De(E+u\De)^{-1} y^{2s}\bigg\vert_{y=0}\\
&=-\Na(I-u\Na)^{-1} y^{2s}\bigg\vert_{y=0}\\
&=-\Na(I-u\Na)^{-1}(I-(u\Na)^{2s}) y^{2s}\bigg\vert_{y=0}\\
&=-\sum _{m=1} ^{2s}
u^{m-1}\Na^m y^{2s}\bigg\vert_{y=0}\\
&=-\sum _{m=1} ^{2s}
\sum _{k=1} ^{m}(-1)^{m-k}\binom mk E^{-k}
y^{2s}\bigg\vert_{y=0}u^{m-1}\\
&=-\sum _{m=1} ^{2s}
\sum _{k=1} ^{m}(-1)^{m-k}\binom mk (-k)^{2s}
u^{m-1}
=-q_s^{(1)}(-2u-1),
\end{align*}
as we claimed.
If we recall that \eqref{eq:res} implies
\begin{equation} \label{eq:p1q1} 
1+p_s^{(1)}(z)=\frac {1}
{\sqrt{1+4z}}q_s^{(1)}\!\left(\sqrt{1+4z}\right),
\end{equation}
and take into account that, under the substitution $u=-\frac {1}
{2}+\frac {1} {2}\sqrt{1+4z}$, we have $2u+1=\sqrt{1+4z}$, 
this establishes all claims of the lemma except for the special form
of the polynomial $p_s^{(1)}(z)$.

First we prove that $p_s^{(1)}(0)=0$. To see this, we put $z=0$
in \eqref{eq:p1q1}, and we set $u=0$ in \eqref{eq:q1}, to get
$$
1+p_s^{(1)}(0)=q_s^{(1)}(1)=1,
$$
so that indeed $p_s^{(1)}(0)=0$.

Next we prove that $p_s^{(1)}(z)$ is a polynomial in $z$ with integer
coefficients. We use again the relation between $p_s^{(1)}(\,.\,)$ and
$q_s^{(1)}(\,.\,)$ given in \eqref{eq:p1q1}. For our purposes, it
turns out that it is more convenient to work with an anti-symmetrised
version of $q_s^{(1)}(2u+1)$. Namely, since we already know that
$q_s^{(1)}(-2u-1)=-q_s^{(1)}(2u+1)$, we may replace $q_s^{(1)}(2u+1)$
by $\frac {1} {2}\left(q_s^{(1)}(2u+1)-q_s^{(1)}(-2u-1)\right)$.
In view of \eqref{eq:q1}, this means that
\begin{equation*} 
q_s^{(1)}(2u+1)=\frac {1} {2}\sum _{m=1} ^{2s}
\left(\sum _{k=1} ^{m}(-1)^{m-k}\binom mk
k^{2s}\right)\left(u^{m-1}-(-1-u)^{m-1}\right).
\end{equation*}
If this is used in \eqref{eq:p1q1}, then we obtain
\begin{align*} 
1+p_s^{(1)}(z)&=\frac {1} {2\sqrt{1+4z}}
\sum _{m=1} ^{2s}
\left(\sum _{k=1} ^{m}(-1)^{m-k}\binom mk
k^{2s}\right)\\
&\kern3cm
\cdot
\left(\left(-\tfrac {1} {2}+\tfrac {1}
{2}\sqrt{1+4z}\right)^{m-1}
-\left(-\tfrac {1} {2}-\tfrac {1} {2}\sqrt{1+4z}\right)^{m-1}\right)\\
&=\frac {1} {2^{m}\sqrt{1+4z}}
\sum _{m=2} ^{2s}
\left(\sum _{k=1} ^{m}(-1)^{m-k}\binom mk
k^{2s}\right)\\
&\kern3cm
\cdot
\left(\left(-1+\sqrt{1+4z}\right)^{m-1}
-\left(-1-\sqrt{1+4z}\right)^{m-1}\right).
\end{align*}
The reader should note that we were free to start the summation over
$m$ at $m=2$ in the last expression, since the bracketed factor
vanishes for $m=1$.
Use of Lemma~\ref{lem:E3} then transforms this identity into
\begin{align*} 
1+p_s^{(1)}(z)
&=\frac {(-1)^{m}} {2}
\sum _{m=2} ^{2s}
\left(\sum _{k=1} ^{m}(-1)^{m-k}\binom mk
k^{2s}\right)
\sum _{\ell\ge0} ^{}\binom {m-\ell-2}\ell z^\ell.
\end{align*}
The right-hand side is ``almost" a polynomial with integer
coefficients; only a denominator of $2$ remains. This denominator
is taken care of by the sum over $k$; namely we have
\begin{align*}
\sum _{k=1} ^{m}(-1)^{m-k}\binom mk
k^{2s}&\equiv
\sum _{k\ge0} ^{}(-1)^{m-2k-1}\binom m{2k+1}
\pmod2\\
&\equiv
(-1)^{m-1}\sum _{k\ge0} ^{}\binom m{2k+1}
\pmod2\\
&\equiv (-1)^{m-1}2^{m-1}\equiv0\pmod2
\end{align*}
for $m\ge2$. Hence, the polynomial $p_s^{(1)}(z)$ has indeed integer
coefficients. 

The last remaining assertion is that the coefficients of
$p_s^{(1)}(z)$ are divisible by~$3$.
Once we have shown that $q_s^{(1)}(2u+1)=2u+1$~modulo~3, then,
in view of \eqref{eq:p1q1}, this last
point will also follow. Consequently, in the sequel we
consider $q_s^{(1)}(2u+1)$ with coefficients reduced modulo~$3$. 
Since, by assumption, $s$ is a positive integer, we have
\begin{align*}
q_s^{(1)}(2u+1)&=\sum _{m=1} ^{2s}u^{m-1}
\underset{k\not\equiv 0~\text {mod }3}
{\sum _{k=0} ^{m}}(-1)^{m-k}\binom mk \quad 
\text {modulo }3\\
&=-\sum _{m=1} ^{2s}u^{m-1}
{\sum _{k=0} ^{m}}(-1)^{m-k}\binom m{3k} \quad 
\text {modulo }3\\
&=\sum _{m=1} ^{2s}(-u)^{m-1}
{\sum _{k=0} ^{m}}(-1)^{3k}\binom m{3k} \quad 
\text {modulo }3.
\end{align*}
``Trisection" of the binomial theorem then yields
$$
q_s^{(1)}(2u+1)
=\frac {1} {3}\sum _{m=1} ^{2s}(-u)^{m-1}\left((1-\om)^m+(1-\om^2)^m\right)
\quad 
\text {modulo }3,
$$
where $\om$ denotes a primitive third root of unity. Since, for
$m\ge3$, we have
$$
\frac {1} {3}\left((1-\om)^m+(1-\om^2)^m\right)=
\frac {1} {3}\left(\left(\tfrac {3-i\sqrt3} {2}\right)^m+
\left(\tfrac {3+i\sqrt3} {2}\right)^m\right)\equiv0\pmod3,
$$
only the summands for $m=1$ and $m=2$ contribute in the above sum.
Thus,
$$
q_s^{(1)}(2u+1)=1+2u\quad \text {modulo }3,
$$
as required. This completes the proof of the lemma.
\end{proof}

\begin{proposition} \label{lem:E2B}
For any positive integer $s$, we have
\begin{equation} \label{eq:E2B}
\sum _{n\ge0} ^{}z^n
\sum _{j=0} ^{n}(-1)^{n-j}\binom {2n}{n-j}j^{2s-1}
=\frac {z} {\sqrt{1+4z}}
\left(1+p_s^{(2)}(z)\right),
\end{equation}
where $p_s^{(2)}(z)$ is a polynomial in $z$ with integer coefficients
all of which are divisible by~$3$, and which satisfies $p_s^{(2)}(0)=0$.
An explicit expression for $p_s^{(2)}(z)$ can be derived from
\eqref{eq:q2} and \eqref{eq:p2q2}.
\end{proposition}

\begin{proof}
The proof is completely analogous to the proof of Proposition~\ref{lem:E2},
so we content ourselves with just stating the crucial
identities without proof.

Using the contour integration method, one obtains
\begin{multline*}
\sum _{n\ge0} ^{}z^n
\sum _{j=0} ^{n}(-1)^{n-j}\binom {2n}{n-j}j^{2s-1}\\
=\frac {1} {2\pi i}
\int _{C} ^{}\frac {1+u} {u^2+u-{z} }
\sum _{m=1} ^{2s-1}
\left(\sum _{k=1} ^{m}(-1)^{m-k}\binom mk k^{2s-1}\right)u^{m}
\,du.
\end{multline*}
Application of the residue theorem then leads to
\begin{multline*}
\sum _{n\ge0} ^{}z^n
\sum _{j=0} ^{n}(-1)^{n-j}\binom {2n}{n-j}j^{2s-1}\\
=\frac {z} {\sqrt{1+4z}}
\bigg(\sum _{m=1} ^{2s-1}
\left(\sum _{k=1} ^{m}(-1)^{m-k}\binom mk k^{2s-1}\right)u^{m-1}
\bigg)\bigg\vert_{u=-\frac {1} {2}+\frac {1} {2}\sqrt{1+4z}}.
\end{multline*}
Next one defines
\begin{equation} \label{eq:q2}
q_s^{(2)}(2u+1):=\sum _{m=1} ^{2s-1}
\left(\sum _{k=1} ^{m}(-1)^{m-k}\binom mk k^{2s-1}\right)u^{m-1},
\end{equation}
so that
\begin{equation} \label{eq:p2q2} 
1+p_s^{(2)}(z)=q_s^{(2)}\left(\sqrt{1+4z}\right).
\end{equation}
Now one proves that $q_s^{(2)}(-2u-1)=q_s^{(2)}(2u+1)$ (that is, {\it 
symmetry} of the polynomial $q_s^{(2)}(\,.\,)$), establishing that
$p_s^{(2)}(z)$ is indeed a polynomial in $z$. By definition,
it is obvious that $q_s^{(2)}(1)=1$, whence $p_s^{(2)}(0)=0$.
Furthermore, using
the {\it symmetrisation} of $q_s^{(2)}(\,.\,)$, Identity~\eqref{eq:subs2},
and Relation \eqref{eq:p2q2}, one shows that $p_s^{(2)}(z)$ has
integer coefficients. In the final step, one proves that all these
coefficients are divisible by~$3$, by demonstrating that
$$
q_s^{(2)}(2u+1)=1\quad \text{modulo }3.
$$
\end{proof}

\begin{theorem} \label{thm:Euler1}
Let $\Psi(z)$ be given by \eqref{eq:Psidef}, 
and let $\al$ be some positive integer.
Then the generating function $E^{(1)}(z)=\sum_{n\ge0}A(2n,n+1)\,z^n$ 
for central Eulerian numbers with even first index,
reduced modulo $3^{3^\al}$, 
can be expressed as a polynomial in $\Psi(z)$ of the form
$$
E^{(1)}(z)=a_0(z)+\sum _{i=0} ^{2\cdot 3^{\al}-1}a_{i}(z)\Psi^{i}(z^3)\quad 
\text {\em modulo }3^{3^\al},
$$ 
where the coefficients $a_{i}(z)$, $i=0,1,\dots,2\cdot 3^\al-1$, 
are Laurent polynomials in $z$ and $1+z$.
\end{theorem}

\begin{proof} 
By \eqref{eq:Ank}, we have
$$
A(2n,n+1)=\sum _{j=0} ^{n+1}(-1)^{n+1-j}\binom {2n+1}{n+1-j}j^{2n}.
$$
In order to analyse this expression modulo $3^\be$ (the reader should 
imagine that $\be=3^\al$), 
we have to distinguish between several cases, depending on the
congruence class of $n$ modulo $3^{\be-1}$. Namely, since
$\varphi(3^\be)=2\cdot 3^{\be-1}$ (with $\varphi(\,.\,)$ denoting the Euler
totient function), we have
\begin{multline} \label{eq:Ans} 
A(2n,n+1)\equiv\sum _{j=0} ^{n+1}(-1)^{n+1-j}\binom
{2n+1}{n+1-j}j^{2s}\pmod{3^\be}\\
\text{for}\quad n\equiv s~(\text{mod }3^{\be-1})
\text{ and }n,s\ge\tfrac {1} {2}(\be-1).
\end{multline}
(The condition $n,s\ge\frac {1} {2}(\be-1)$ is needed to ensure that
$j^{2n}\equiv j^{2s}\equiv0$~(mod~$3^{\be-1}$) for all $j$ which are
divisible by~$3$.)
From Proposition~\ref{lem:E2}, we know that the generating function for the
numbers on the right-hand side of the above congruence, say
$$
E_s^{(1)}(z)=
\sum _{n\ge0} ^{}z^n\sum _{j=0} ^{n+1}(-1)^{n+1-j}\binom
{2n+1}{n+1-j}j^{2s},
$$
has the form
given on the right-hand side of \eqref{eq:E2}. Consequently, it
satisfies the functional equation
\begin{equation} \label{eq:E1eq} 
(E_s^{(1)})^2(z)-\left(1+p_s^{(1)}(z)\right)E_s^{(1)}(z)-z\left(1+p_s^{(1)}(z)\right)^2=
0.
\end{equation}
If we rewrite this in the form
$$
(E_s^{(1)})^2(z)-E_s^{(1)}(z)-z
-p_s^{(1)}(z)E_s^{(1)}(z)-zp_s^{(1)}(z)\left(2+p_s^{(1)}(z)\right)=
0,
$$
then we see that it fits into the framework of Theorem~\ref{thm:general}.
Indeed, one chooses
$\ga=1$, $c_2(z)=1$, $c_1(z)=-1$, $c_0(z)=-z$, which implies that
$$c_1^2(z)-c_0(z)c_2(z)=1+z\quad \text{modulo }3,$$
so that the constants $e_1,e_2,f_1,f_2$ are given by
$e_1=0$,
$e_2=0$,
$f_1=0$,
$f_2=0$,
$\ep=1$, and
we have
$$\mathcal Q(\dots)=-\tfrac {1} {3}p_s^{(1)}(z)E_s^{(1)}(z)
-\tfrac {1} {3}zp_s^{(1)}(z)\left(2+p_s^{(1)}(z)\right).$$
This is indeed a polynomial in $z$ and $E_s^{(1)}(z)$ with integer
coefficients since, by Proposition~\ref{lem:E2}, all coefficients of the
polynomial $p_s^{(1)}(z)$ are divisible by~$3$.
Consequently, for each $s$, the generating function $E_s^{(1)}(z)$, when reduced modulo
$3^\be$, can be written as a polynomial in the basic series
$\Psi(z)$. 

However, we are actually not interested in all coefficients in
$E_s^{(1)}(z)$. By \eqref{eq:Ans}, the relevant coefficients are those
of powers $z^n$, where $n\equiv s$~(mod~$3^{\be-1}$). In other words,
for each fixed $s$, we need a particular $3^{\be-1}$-section of
$E_s^{(1)}(z)$, and these $3^{\be-1}$-sections have then to be summed to
obtain the series $E^{(1)}(z)$ modulo~$3^\be$. To be precise, we have
$$
E^{(1)}(z)=
\sum _{s=\cl{\frac {1} {2}(\be-1)}} ^{3^{\be-1}+\cl{\frac {1} {2}(\be-1)}-1}
\underset{s+3^{\be-1}\ell\ge0}
{\sum _{\ell\in\mathbb Z} ^{}}z^{s+3^{\be-1}\ell}\coef{t^{s+3^{\be-1}\ell}}E_s^{(1)}(t)
+\mathcal E(z),
$$
where $\coef{t^m}f(t)$ denotes the coefficient of $t^m$ in the series
$f(t)$,
and the polynomial $\mathcal E(z)$ represents the possible correction which takes
care of deviations from the congruence \eqref{eq:Ans} for $n$'s with
$n<\frac {1} {2}(\be-1)$. 
However, computing $3^{\be-1}$-sections of $E_s^{(1)}(z)$ is
straightforward, once we observe that
\begin{equation} \label{eq:psi3al} 
\Psi(z)=\Psi(z^{3^{\be-1}})
\prod _{j=0} ^{{\be-2}}(1+z^{3^j}).
\end{equation}
Each $3^{\be-1}$-section, with coefficients reduced by $3^\be$, 
is then a polynomial in $\Psi(z^{3^{\be-1}})$,
or, equivalently, if we use \eqref{eq:psi3al} in the other direction,
a polynomial in $\Psi(z)$.
Hence, the generating function $E^{(1)}(z)$ of central Eulerian
numbers, when coefficients are reduced modulo $3^\be$, can be
expressed as a polynomial in $\Psi(z)$, which was exactly our claim.
\end{proof}

How the procedure explained in the above proof works for modulus~$9$
is described in detail in the proof of the next result.

\begin{theorem} \label{thm:Euler1-9}
Let $\Psi(z)$ be given by \eqref{eq:Psidef}.
Then, we have
\begin{equation} 
\label{eq:LoesE19}
\sum _{n\ge0} ^{}A(2n,n+1)\,z^n
=-1
+3
   (z+1) \Psi(z)
+\frac{3 z^4+2 z^3+6
   z^2+2}{1+z}\Psi^3(z)
\quad \text {\em modulo }9.
\end{equation}
\end{theorem}

\begin{proof}
We start by considering the case $n\equiv 1$~(mod~3).
Under this condition, we have
$$
A(2n,n+1)\equiv \sum _{j=0} ^{n+1}(-1)^{n+1-j}\binom
{2n+1}{n+1-j}j^{2}
\pmod 9.
$$
From Proposition~\ref{lem:E2}, we obtain
\begin{equation} \label{eq:EEF2}
E_1^{(1)}(z)=\sum _{n\ge0} ^{}z^n
\sum _{j=0} ^{n+1}(-1)^{n+1-j}\binom {2n+1}{n+1-j}j^{2}
=\frac {1} {2}\left(1+\sqrt{1+4z}\right).
\end{equation}
Consequently, by applying \eqref{eq:E1eq} with $s=1$ and
$p_1^{(1)}(z)=0$, we see that
the series $E_1^{(1)}(z)$ satisfies the functional equation 
$$
(E_s^{(1)})^2(z)-E_s^{(1)}(z)-z=
0.
$$
If we use the method from Section~\ref{sec:method} in the form of
Theorem~\ref{thm:general}, with the relation \eqref{eq:PsiRel}
replaced by \eqref{eq:PsiEq}, then we obtain
$$
E_1^{(1)}(z)=-4
+3 (z+1) \Psi(z)
+\left(5 z^2+7 z+2\right)
   \Psi^3(z)\quad 
\text{modulo }9.
$$
In this series, we are only interested in the terms involving powers
$z^n$ with $n\equiv1$~(mod~3). This $3$-section of the right-hand side
of the above congruence equals
\begin{equation} \label{eq:E11sec} 
6 z \Psi(z^3) + (  4 z^4+4 z) \Psi^3(z^3).
\end{equation}

In a similar vein, if
$n\equiv 2$~(mod~3), then
$$
A(2n,n+1)\equiv \sum _{j=0} ^{n+1}(-1)^{n+1-j}\binom
{2n+1}{n+1-j}j^{4}
\pmod 9.
$$
From Proposition~\ref{lem:E2}, we obtain
\begin{equation} \label{eq:EEF4}
E_2^{(1)}(z)=\sum _{n\ge0} ^{}z^n
\sum _{j=0} ^{n+1}(-1)^{n+1-j}\binom {2n+1}{n+1-j}j^{4}
=\frac {1} {2}\left(1+\sqrt{1+4z}\right)
(1+12z)
\end{equation}
Consequently, by applying \eqref{eq:E1eq} with $s=1$ and
$p_1^{(1)}(z)=12z$, we see that
the series $E_2^{(1)}(z)$ satisfies the functional equation 
$$
(E_s^{(1)})^2(z)-(1+12z)E_s^{(1)}(z)-z(1+12z)^2=
0.
$$
The method from Section~\ref{sec:method} then yields
$$
E_2^{(1)}(z)=
-3z-4
+3 (z+1) \Psi(z)
+\left(6 z^3+8 z^2+4 z+2\right)
   \Psi^3(z)
\quad 
\text{modulo }9.
$$
The relevant part, obtained by extracting only terms involving powers
$z^n$ with $n\equiv2$~(mod~3) is
\begin{equation} \label{eq:E12sec} 
3 z^2 \Psi(z^3) - ( z^5 + z^2) \Psi^3(z^3).
\end{equation}

Finally, if
$n\equiv 0$~(mod~3), then
$$
A(2n,n+1)\equiv \sum _{j=0} ^{n+1}(-1)^{n+1-j}\binom
{2n+1}{n+1-j}j^{6}
\pmod 9.
$$
(Although $n=0$ violates the inequality $n\ge\frac {1} {2}(\be-1)=\frac {1} {2}$ in 
\eqref{eq:Ans},
this holds indeed for {\it all\/} $n\ge0$ which are divisible by~3.)
From Proposition~\ref{lem:E2}, we obtain
\begin{equation} \label{eq:EEF6}
E_3^{(1)}(z)=\sum _{n\ge0} ^{}z^n
\sum _{j=0} ^{n+1}(-1)^{n+1-j}\binom {2n+1}{n+1-j}j^{6}
=\frac {1} {2}\left(1+\sqrt{1+4z}\right)
(1+60z+360z^2)
\end{equation}
Consequently, 
the series $E_2^{(1)}(z)$ satisfies the functional equation 
$$
(E_s^{(1)})^2(z)-(1+60z+360z^2)E_s^{(1)}(z)-z(1+60z+360z^2)^2=
0.
$$
The method from Section~\ref{sec:method} then yields
$$
E_3^{(1)}(z)=
3z-4
+3 (z+1) \Psi(z)
+\left(3 z^3+2 z^2+z+2\right)
   \Psi^3(z)
\quad 
\text{modulo }9.
$$
The relevant part, obtained by extracting only terms involving powers
$z^n$ with $n\equiv0$~(mod~3) is
\begin{equation} \label{eq:E13sec} 
-1 + 3  \Psi(z^3) + (2 + 5 z^3 + 3 z^6) \Psi^3(z^3).
\end{equation}
Summing the expressions \eqref{eq:E11sec}, \eqref{eq:E12sec},
and \eqref{eq:E13sec}, 
we arrive at \eqref{eq:LoesE19}.
\end{proof}

Applying the same procedure for the modulus~27, now involving an
analysis of 9 different cases which have to be combined in the end,
one is able to prove the following result.

\begin{theorem} \label{thm:Euler1-27}
Let $\Psi(z)$ be given by \eqref{eq:Psidef}.
Then, we have
\begin{multline} 
\label{eq:LoesE127}
\sum _{n\ge0} ^{}A(2n,n+1)\,z^n
=14
+3 \left(3 z^2-4 z+2\right)
   \Psi(z)
+\left(21 z^3+20 z^2+13
   z+23\right) \Psi^3(z)\\
+3
   \left(6 z^4+4 z^3+3
   z^2+4\right) \Psi^5(z)
\quad \text {\em modulo }27.
\end{multline}
\end{theorem}

These results generalise
Theorem~5.17 in \cite{DeSaAA}, which is an easy consequence.

\medskip
Now we turn to the analogous theorems for the 
central Eulerian numbers $A(2n-1,n)$.

\begin{theorem} \label{thm:Euler2}
Let $\Psi(z)$ be given by \eqref{eq:Psidef}, 
and let $\al$ be some positive integer.
Then the generating function $E^{(2)}(z)=\sum_{n\ge0}A(2n-1,n)\,z^n$ 
for central Eulerian numbers with odd first index,
reduced modulo $3^{3^\al}$, 
can be expressed as a polynomial in $\Psi(z)$ of the form
$$
E^{(2)}(z)=a_0(z)+\sum _{i=0} ^{2\cdot 3^{\al}-1}a_{i}(z)\Psi^{i}(z^3)\quad 
\text {\em modulo }3^{3^\al},
$$ 
where the coefficients $a_{i}(z)$, $i=0,1,\dots,2\cdot 3^\al-1$, 
are Laurent polynomials in $z$ and $1+z$.
\end{theorem}

\begin{proof} 
By \eqref{eq:Ank}, we have
$$
A(2n-1,n)=\sum _{j=0} ^{n}(-1)^{n-j}\binom {2n}{n-j}j^{2n-1}.
$$
Again, in order to analyse this expression modulo $3^\be$ (the reader should 
imagine that $\be=3^\al$), 
we have to distinguish several cases, depending on the
congruence class of $n$ modulo $3^{\be-1}$. Namely, we have
\begin{multline} \label{eq:AnsB} 
A(2n-1,n)\equiv\sum _{j=0} ^{n}(-1)^{n-j}\binom
{2n}{n-j}j^{2s-1}\pmod{3^\be}\\
\text{for}\quad n\equiv s~(\text{mod }3^{\be-1}).
\text{ and }n,s\ge\tfrac {\be} {2}.
\end{multline}
(The condition $n,s\ge\frac {\be} {2}$ is needed to ensure that
$j^{2n-1}\equiv j^{2s-1}\equiv0$~(mod~$3^{\be-1}$) for all $j$ which are
divisible by~$3$.)
From Proposition~\ref{lem:E2B}, we know that the generating function for the
numbers on the right-hand side of the above congruence, say
$$
E_s^{(2)}(z)=
\sum _{n\ge0} ^{}z^n\sum _{j=0} ^{n}(-1)^{n-j}\binom
{2n}{n-j}j^{2s-1},
$$
has the form
given on the right-hand side of \eqref{eq:E2B}. Consequently, it
satisfies the functional equation
\begin{equation} \label{eq:E2eq} 
(4z+1)(E_s^{(2)})^2(z)-z^2\left(1+p_s^{(2)}(z)\right)^2=
0.
\end{equation}
If we rewrite this in the form
$$
(1+z)(E_s^{(2)})^2(z)-z^2+3(E_s^{(2)})^2(z)
-z^2p_s^{(2)}(z)\left(2+p_s^{(2)}(z)\right)=
0,
$$
then we see that it fits the framework of Theorem~\ref{thm:general}.
Here, one chooses
$\ga=1$, $c_2(z)=1+z$, $c_1(z)=0$, $c_0(z)=-z^2$, which implies that
$$c_1^2(z)-c_0(z)c_2(z)=z^2(1+z)\quad \text{modulo }3,$$
so that the constants $e_1,e_2,f_1,f_2$ are given by
$e_1=0$,
$e_2=1$,
$f_1=1$,
$f_2=0$,
$\ep=1$, and
we have
$$\mathcal Q(\dots)=
(E_s^{(2)})^2(z)
-\tfrac {1} {3}zp_s^{(2)}(z)\left(2+p_s^{(2)}(z)\right).$$
This is indeed a polynomial in $z$ and $E_s^{(2)}(z)$ with integer
coefficients since, by Proposition~\ref{lem:E2B}, all coefficients of the
polynomial $p_s^{(2)}(z)$ are divisible by~$3$.
Consequently, for each $s$, the generating function $E_s^{(2)}(z)$, when reduced modulo
$3^\be$, can be written as a polynomial in the basic series
$\Psi(z)$. 

The rest of the argument is identical with corresponding arguments in
the proof of Theorem~\ref{thm:Euler1}.
\end{proof}

For the modulus 9, 
the procedure explained in the above proof leads to the following result.

\begin{theorem} \label{thm:Euler2-9}
Let $\Psi(z)$ be given by \eqref{eq:Psidef}.
Then, we have
\begin{equation} 
\label{eq:LoesE29}
\sum _{n\ge0} ^{}A(2n-1,n)\,z^n
=
-3 z \Psi(z)+
\left(6 z^4+6 z^3+4 z^2+4
   z\right) \Psi^3(z)
\quad \text {\em modulo }9.
\end{equation}
\end{theorem}

\begin{proof}
As in the proof of Theorem~\ref{thm:Euler1-9}, we distinguish three cases, 
depending on the congruence class of $n$ modulo~3.
In order to address the case where $n\equiv 1$~(mod~3), 
from Proposition~\ref{lem:E2B} we compute
\begin{equation} \label{eq:EEF1}
\sum _{n\ge0} ^{}z^n
\sum _{j=0} ^{n}(-1)^{n-j}\binom {2n}{n-j}j^7
=\frac {z} {\sqrt{1+4z}}(1 + 126 z + 1680 z^2 + 5040 z^3).
\end{equation}
Consequently, by applying \eqref{eq:E1eq} with $s=1$ and
$p_1^{(2)}(z)=1 + 126 z + 1680 z^2 + 5040 z^3$, we see that
the series $E_2^{(2)}(z)$ satisfies the functional equation 
$$
(1+4z)(E_s^{(2)})^2(z)-z^2(1 + 126 z + 1680 z^2 + 5040 z^3)^2=
0.
$$
The method from Section~\ref{sec:method} in the form of
Theorem~\ref{thm:general} then yields
$$
E_1^{(2)}(z)=
-3 z \Psi(z)+\left(6 z^4+6 z^3+7 z^2+4
   z\right) \Psi^3(z)
\quad 
\text{modulo }9,
$$
and the relevant part of it, consisting only of terms involving
powers $z^n$ with $n\equiv1$~(mod~3) is
\begin{equation} \label{eq:E21sec} 
-3 z \Psi(z^3)+\left(6 z^7+4z^4+4z\right)
   \Psi^3(z^3).
\end{equation}
Next, addressing the case where $n\equiv2$~(mod~3), we have
\begin{equation} \label{eq:EEF3}
\sum _{n\ge0} ^{}z^n
\sum _{j=0} ^{n}(-1)^{n-j}\binom {2n}{n-j}j^{3}
=\frac {z} {\sqrt{1+4z}}
(1+6z),
\end{equation}
so that this series satisfies the functional equation
$$
(1+4z)(E_s^{(2)})^2(z)-z^2(1 + 6 z )^2=
0.
$$
Our algorithm then yields
$$
E_2^{(2)}(z)=
-3 z \Psi(z)+\left(6 z^3+4 z^2+4 z\right)
   \Psi^3(z)
\quad 
\text{modulo }9,
$$
with the relevant part being
\begin{equation} \label{eq:E22sec} 
-3 z ^2 \Psi(z^3) + ( 22 z^5+16 z^2) \Psi^3(z^3).
\end{equation}
Finally, in order to address the case where $n\equiv0$~(mod~3),
we compute
\begin{equation} \label{eq:EEF5}
\sum _{n\ge0} ^{}z^n
\sum _{j=0} ^{n}(-1)^{n-j}\binom {2n}{n-j}j^{5}
=\frac {z} {\sqrt{1+4z}}
(1+30z+120z^2),
\end{equation}
so that this series satisfies the functional equation
$$
(1+4z)(E_s^{(2)})^2(z)-z^2(1 + 30z+120z^2)^2=
0.
$$
Our algorithm yields
$$
E_3^{(2)}(z)=
-3 z \Psi(z)
+\left(3 z^4+6 z^3+z^2+4
   z\right) \Psi^3(z)
\quad 
\text{modulo }9,
$$
with the relevant part being
\begin{equation} \label{eq:E23sec} 
( 15 z^6+21 z^3) \Psi^3(z^3).
\end{equation}
Summation of \eqref{eq:E21sec},  \eqref{eq:E22sec}, and \eqref{eq:E23sec}
then leads to \eqref{eq:LoesE29}. 
\end{proof}

The same procedure produces the following result for the modulus~27.

\begin{theorem} \label{thm:Euler2-27}
Let $\Psi(z)$ be given by \eqref{eq:Psidef}.
Then, we have
\begin{multline} 
\label{eq:LoesE227}
\sum _{n\ge0} ^{}A(2n-1,n)\,z^n
=
-{3
   z \left(3 z^2+5
   \right) }\Psi(z)
+z \left(24 z^3+15 z^2+10
   z+19\right) \Psi^3(z)
\\
+3 z
   \left(3 z^4+6 z^3+2 z^2+7
   z+8\right) \Psi^5(z)
\quad \text {\em modulo }27.
\end{multline}
\end{theorem}

These results generalise
Theorem~5.16 in \cite{DeSaAA}, which is an easy consequence.

\begin{remark}
It is obvious that we could also treat ``almost" central Eulerian
numbers, such as $A(2n,n-1)$ or $A(2n-1,n-1)$, modulo powers of~$3$
using the same approach, and this would lead to similar results.
\end{remark}

\section{Ap\'ery numbers}
\label{sec:Apery}

In \cite{DeSaAA}, Deutsch and Sagan also discuss {\it Ap\'ery numbers}
(and, in fact, generalised Ap\'ery numbers), and they derive in
particular explicit congruences for Ap\'ery numbers modulo~3.
The corresponding result, Theorem~5.14 in \cite{DeSaAA}, has a
form which seems to suggest that our method from Section~\ref{sec:method}
should be applicable here as well. 
It turns out however that this is not the case. Embarrassingly,
we are not even able to make our method work for the modulus~$3$,
due to the fact that we were not able to find an equation 
for the corresponding generating functions which would be
suitable for our method.

Let us begin with the Ap\'ery numbers of the first kind, defined by
\begin{equation} \label{eq:Apery2} 
A^{(2)}_n=\sum_{k=0}^n {\binom nk}^2\binom {n+k}k.
\end{equation}
They are associated with $\zeta(2)$, see e.g.\ \cite{fi}.

There is a well-known recurrence relation for these numbers.
Namely,
Ap\'ery computed (and the Gosper--Zeilberger algorithm
(cf.\ \cite[Ch.~6]{PeWZAA}) finds
automatically) the recurrence
\begin{equation} \label{eq:Apery2Rek}
(n+2)^2
   A^{(2)}_{n+2}
-\left(11 n^2+33
   n+25\right)
   A^{(2)}_{n+1}
-(n+1)^2
   A^{(2)}_n
=0.
\end{equation}
This translates into the linear differential equation
\begin{equation} \label{eq:Apery2EF}
(z+3)
   A^{(2)}(z)
+\left(3
   z^2+22 z-1\right)
   (A^{(2)})'(z)
+\left(z^3+11
   z^2-z\right) (A^{(2)})''(z)=0
\end{equation}
for the generating function $A^{(2)}(z):=\sum_{n\ge0}A^{(2)}_n\,z^n$,
with initial conditions $A^{(2)}(0)=1$ and $(A^{(2)})'(0)=3$.
Unfortunately, this equation is {\it not\/} suitable for our
method, as Equation~\eqref{eq:Apery2EF} does not determine 
a unique solution modulo~$3$ (see the last paragraph in
Section~\ref{sec:method}). Thus, in lack of a different equation,
we cannot apply our method.

As it turns out, there is apparently a deeper reason why our method 
does not apply. By looking at enough data, we have worked
out a conjectural description of the Ap\'ery numbers of the
first kind modulo~$9$, see the conjecture below.
The patterns describing the various congruence classes do
not fit with the patterns which appear in the descriptions
of the coefficients of $\Psi^3(z)$ modulo~$9$ which one obtains by
carrying out the approach sketched in Section~\ref{sec:extr}
(see in particular the paragraphs after \eqref{eq:Psipoteven}
and \eqref{eq:Psipotodd}), the resulting congruences being
explicitly displayed 
in Corollary~\ref{cor:Psi-3}.

\begin{conjecture} \label{thm:Apery9}
The Ap\'ery numbers $A^{(2)}_n$ obey the following congruences
modulo $9$:
\begin{enumerate} 
\item [(i)]
$A^{(2)}_n\equiv 1$~{\em(mod~$9$)} if, and only if,
the $3$-adic expansion of $n$ is an element of
$$
\{0,2\}^*;
$$
\item [(ii)]
$A^{(2)}_n\equiv 3$~{\em(mod~$9$)} if, and only if,
the $3$-adic expansion of $n$ has exactly one occurrence of the
string $01$ and otherwise contains only $0$'s and $2$'s;
\item [(iii)]
$A^{(2)}_n\equiv 6$~{\em(mod~$9$)} if, and only if,
the $3$-adic expansion of $n$ has exactly one occurrence of the
string $21$ and otherwise contains only $0$'s and $2$'s;
\item[(iv)]in the cases not covered by Items~{\em(i)}--{\em(iii),}
$A^{(2)}_n$ is divisible by $9;$
in particular, 
$A^{(2)}_n\not\equiv 2,4,5,7,8$~{\em(mod~$9$)} for all $n$.
\end{enumerate}
\end{conjecture}

The Ap\'ery numbers of the second kind are defined by
$$
A^{(3)}_n=\sum_{k=0}^n {\binom nk}^2{\binom {n+k}k}^2.
$$
They are associated with $\zeta(3)$, see e.g.\ \cite{fi}.

These numbers satisfy also a well-known recurrence relation
which 
Ap\'ery computed (and the Gosper--Zeilberger algorithm finds
automatically), namely
$$
(n+2)^3
   A^{(3)}_{n+2}
-(2n+3)\left(17 n^2+51
   n+39\right)
   A^{(3)}_{n+1}
+(n+1)^3
   A^{(3)}_n
=0.
$$
This translates into the linear differential equation
\begin{multline} \label{eq:Apery3EF}
(z-5)
   A^{(3)}(z)
+\left(7 z^2-112
   z+1\right) (A^{(3)})'(z)\\
+3
   \left(2 z^3-51 z^2+z\right)
   (A^{(3)})''(z)
+\left(z^4-34
   z^3+z^2\right) (A^{(3)})'''(z)=0
\end{multline}
for the generating function $A^{(3)}(z):=\sum_{n\ge0}A^{(3)}_n\,z^n$,
with initial conditions $A^{(3)}(0)=1$ and $(A^{(3)})'(0)=5$.

Again, this equation is {\it not\/} suitable for our
method, for the same reason as before.

Also here we have worked out a conjectural description modulo~$9$,
displaying patterns which do not fit with those appearing in the
description of the coefficients of $\Psi^3(z)$ modulo~$9$.

\begin{conjecture} \label{thm:Apery39}
The Ap\'ery numbers $A^{(3)}_n$ obey the following congruences
modulo $9$:
\begin{enumerate} 
\item [(i)]
$A^{(3)}_n\equiv 1$~{\em(mod~$9$)} if, and only if,
the $3$-adic expansion of $n$ contains $6k$ digits $1$, for some $k$,
and otherwise only $0$'s and $2$'s;
\item [(ii)]
$A^{(3)}_n\equiv 2$~{\em(mod~$9$)} if, and only if,
the $3$-adic expansion of $n$ contains $6k+5$ digits $1$, for some $k$,
and otherwise only $0$'s and $2$'s;
\item [(iii)]
$A^{(3)}_n\equiv 4$~{\em(mod~$9$)} if, and only if,
the $3$-adic expansion of $n$ contains $6k+4$ digits $1$, for some $k$,
and otherwise only $0$'s and $2$'s;
\item [(iv)]
$A^{(3)}_n\equiv 5$~{\em(mod~$9$)} if, and only if,
the $3$-adic expansion of $n$ contains $6k+1$ digits $1$, for some $k$,
and otherwise only $0$'s and $2$'s;
\item [(v)]
$A^{(3)}_n\equiv 7$~{\em(mod~$9$)} if, and only if,
the $3$-adic expansion of $n$ contains $6k+2$ digits $1$, for some $k$,
and otherwise only $0$'s and $2$'s;
\item [(vi)]
$A^{(3)}_n\equiv 8$~{\em(mod~$9$)} if, and only if,
the $3$-adic expansion of $n$ contains $6k+3$ digits $1$, for some $k$,
and otherwise only $0$'s and $2$'s;
\item[(vii)]in the cases not covered by Items~{\em(i)}--{\em(vi),}
$A^{(3)}_n$ is divisible by $9;$
in particular, 
$A^{(3)}_n\not\equiv 3,6$~{\em(mod~$9$)} for all $n$.
\end{enumerate}
\end{conjecture}

\appendix

\section{Expansions of powers of the Cantor-like power series $\Psi(z)$}
\label{appA}

In this section, we indicate how to explicitly extract
coefficients from powers of $\Psi(z)$ modulo~$3^e$, where $e$
is a given positive integer. The reader should recall that,
by \eqref{eq:Psipoteven} and \eqref{eq:Psipotodd}, it suffices to
explain how to extract coefficients from expressions of the form
\begin{equation} \label{eq:psimala} 
\Psi(z)\frac {1} {(1+z)^K}
\prod _{j=1} ^{s}\left(\frac {z^{3^{k_j}}(1+z^{3^{k_j}})} 
{1+z^{3^{k_j+1}}}\right)^{a_j}
\end{equation}
and
\begin{equation} \label{eq:mala} 
\frac {1} {(1+z)^K}
\prod _{j=1} ^{s}\left(\frac {z^{3^{k_j}}(1+z^{3^{k_j}})} 
{1+z^{3^{k_j+1}}}\right)^{a_j}.
\end{equation}
We start with \eqref{eq:psimala}. Rather than providing a general
description of how to proceed (which is possible but would be 
hardly comprehensible), we prefer to work through the
special case where $s=2$, that is,
\begin{equation} \label{eq:a1a2} 
\Psi(z)
\frac {1} {(1+z)^K}
\left(\frac {z^{3^{k_1}}(1+z^{3^{k_1}})} 
{1+z^{3^{k_1+1}}}\right)^{a_1}
\left(\frac {z^{3^{k_2}}(1+z^{3^{k_2}})} 
{1+z^{3^{k_2+1}}}\right)^{a_2},
\end{equation}
where $k_1>k_2\ge0$ and $a_1,a_2\ge1$.
We have to distinguish several cases, depending on whether $K,k_1,k_2$
are ``close to each other," or not.

First let $K\le k_2$ and $k_2+a_2<k_1$. In that case, the expression
\eqref{eq:a1a2} can be transformed into
\begin{multline*}
\prod _{i=1} ^{K-1}(1-z+z^2-\dots+z^{3^i-1})
\prod _{i=K} ^{k_2-1}(1+z^{3^i})\\
\times
(1+z^{3^{k_2}})(z^{3^{k_2}}+z^{2\cdot 3^{k_2}})^{a_2}
\prod _{i=k_2+2} ^{k_2+a_2}(1-z^{3^{k_2+1}}+z^{2\cdot 3^{k_2+1}}-\dots
+z^{3^i-3^{k_2+1}})
\prod _{i=k_2+a_2+1} ^{k_1-1}(1+z^{3^i})\\
\times
(1+z^{3^{k_1}})
(z^{3^{k_1}}+z^{2\cdot 3^{k_1}})^{a_1}
\prod _{i=k_1+2} ^{k_1+a_1}(1-z^{3^{k_1+1}}+z^{2\cdot 3^{k_1+1}}-\dots
+z^{3^i-3^{k_1+1}})
\prod _{i=k_1+a_1+1} ^{\infty}(1+z^{3^i}).
\end{multline*}
We write this as
\begin{equation} \label{eq:P012} 
P_0(z)
\left(\prod _{i=K} ^{k_2-1}(1+z^{3^i})\right)
P_1(z)
\left(\prod _{i=k_2+a_2+1} ^{k_1-1}(1+z^{3^i})\right)P_2(z)
\left(\prod _{i=k_1+a_1+1} ^{\infty}(1+z^{3^i})\right),
\end{equation}
which implicitly defines the polynomials $P_0(z),P_1(z),P_2(z)$.
It should be noted that $P_0(z)$ has degree at most $3^K$,
$P_1(z)$ has order $a_23^{k_2}$ and degree at most $3^{k_2+a_2}$,
and 
$P_2(z)$ has order $a_13^{k_1}$ and degree at most $3^{k_1+a_1}$.
If we ``organise" the polynomials $P_0(z),P_1(z),P_2(z)$ in the form
\begin{align*}
P_0(z)&=\sum_{\ga\in\Z}\ga\sum_{(m)_3\in Z_\ga}z^m,\\
P_1(z)&=\sum_{\be\in\Z}\be\sum_{(m)_3\in Y_\be 0^{k_2}}z^m,\\
P_2(z)&=\sum_{\al\in\Z}\al\sum_{(m)_3\in X_\al 0^{k_1}}z^m,
\end{align*}
where $(m)_3$ denotes the $3$-adic representation of the integer $m$,
considered as a string of $0$'s, $1$'s, and $2$'s,
$X_\al$ is an appropriate set of strings contained in
$\{0,1,2\}^{K}$,
$Y_\al$ is an appropriate set of strings contained in
$\{0,1,2\}^{a_2+1}$, and
$Z_\al$ is an appropriate set of strings contained in
$\{0,1,2\}^{a_1+1}$. Inspection of \eqref{eq:P012} makes it now
obvious that the coefficients $c_n$ in the series $\sum_{n\ge0}c_nz^n$
given by \eqref{eq:a1a2} can be described in the form
\begin{multline*}
c_n=C\quad \text{if, and and only if,}\quad \\
(n)_3\in \bigcup_{\al\cdot\be\cdot\ga=C}
\bigg(\bigcup_{X\in X_\al,\,Y\in Y_\be,\,Z\in Z_\ga}
\{0,1\}^*X 
\{0,1\}^{k_1-k_2-a_2-1}Y
\{0,1\}^{k_2-K}Z\bigg),
\end{multline*}
where we used again the word notation explained before 
Theorem~\ref{thm:Free9}.

Next let $K> k_2\ge0$ (``$k_1$ is close to $K$") 
and $K+a_2\le k_1$. In this case, the expression
\eqref{eq:a1a2} can be transformed into
\begin{multline*}
\prod _{i=1} ^{K-1}(1-z+z^2-\dots+z^{3^i-1})\\
\times
(1+z^{3^{k_2}})(z^{3^{k_2}}+z^{2\cdot 3^{k_2}})^{a_2}
\prod _{i=K} ^{K+a_2-1}(1-z^{3^{k_2+1}}+z^{2\cdot 3^{k_2+1}}-\dots
+z^{3^i-3^{k_2+1}})
\prod _{i=K+a_2} ^{k_1-1}(1+z^{3^i})\\
\times
(1+z^{3^{k_1}})
(z^{3^{k_1}}+z^{2\cdot 3^{k_1}})^{a_1}
\prod _{i=k_1+2} ^{k_1+a_1}(1-z^{3^{k_1+1}}+z^{2\cdot 3^{k_1+1}}-\dots
+z^{3^i-3^{k_1+1}})
\prod _{i=k_1+a_1+1} ^{\infty}(1+z^{3^i}).
\end{multline*}
Here, we write this as
\begin{equation} \label{eq:P012b} 
P_1(z)
\left(\prod _{i=K+a_2} ^{k_1-1}(1+z^{3^i})\right)P_2(z)
\left(\prod _{i=k_1+a_1+1} ^{\infty}(1+z^{3^i})\right),
\end{equation}
which implicitly defines the polynomials $P_1(z)$ and $P_2(z)$.
By going through the same reasoning as before, one comes to the
conclusion that in the present case
the coefficients $c_n$ in the series $\sum_{n\ge0}c_nz^n$
given by \eqref{eq:a1a2} can be described in the form
$$
c_n=C\quad \text{if, and and only if,}\quad 
(n)_3\in \bigcup_{\al\cdot\be=C}
\bigg(\bigcup_{X\in X_\al,\,Y\in Y_\be}
\{0,1\}^*X 
\{0,1\}^{k_1-a_2-K}Y\bigg),
$$
for appropriate sets of strings $X_\al$ and $Y_\be$.

The other cases to be discussed are the case where 
$K\le k_2$ and $k_2+a_2\ge k_1$ (``$k_1$ is close to $k_2$";
we always assume $k_2<k_1$),
and the case where $K\ge k_1>k_2\ge0$ (``$K.k_1,k_2$ are all close to
each other").  
Both of them can be treated in a similar fashion.

\medskip
We illustrate the above procedure by making explicit 
how to extract coefficients from odd powers of $\Psi(z)$
modulo~$27$. By Lemma~\ref{lem:Psi-2}, we have
\begin{multline} \label{eq:Psi-3mod27}
\Psi^3(z)=\frac {1} {1+z}\Psi(z)\Bigg(
1
+3\sum _{k_1\ge0} ^{}\frac {z^{3^{k_1}}(1+z^{3^{k_1}})} {1+z^{3^{k_1+1}}}
+9\sum _{k_1>k_2\ge0} ^{}\frac {z^{3^{k_1}+3^{k_2}}
(1+z^{3^{k_1}})(1+z^{3^{k_{2}}})} {(1+z^{3^{k_1+1}})(1+z^{3^{k_2+1}})}
\Bigg)\\
\text {modulo }27.
\end{multline}
The following lemma provides the means for coefficient extraction
from $\Psi^3(z)$ modulo any power of~$3$.

\begin{lemma} \label{lem:Psi-3exp}
{\em (1)} If $k_1=0$, we have
$$
\frac {1} {1+z}\Psi(z)\frac {z^{3^{k_1}}(1+z^{3^{k_1}})}
{1+z^{3^{k_1+1}}}=
\Psi(z)\frac {z}
{1+z^{3}}
=\sum _{n\ge0} ^{}c_nz^n,
$$
where
$$
c_n=\begin{cases} 
1,&\displaystyle\text {if\/ }(n)_3\in
{\bigcup_{X\in\{01,02\}}}
\{0,1\}^*X,\\
0,&\text {otherwise,}
\end{cases}
$$
while, if $k_1\ge1$, we have
$$
\frac {1} {1+z}\Psi(z)\frac {z^{3^{k_1}}(1+z^{3^{k_1}})}
{1+z^{3^{k_1+1}}}=\sum _{n\ge0} ^{}c_nz^n,
$$
where
$$
c_n=\begin{cases} 
1,&\displaystyle\text {if\/ }(n)_3\in
{\bigcup_{X\in\{01,10\}}}
\{0,1\}^*X\{0,1\}^{k_1-1}0,\\
2,&\displaystyle\text {if\/ }(n)_3\in
\{0,1\}^*02\{0,1\}^{k_1-1}0,\\
0,&\text {otherwise.}
\end{cases}
$$

{\em (2)} If $k_2=0$, we have
$$
\frac {1} {1+z}\Psi(z)\frac {z^{3^{k_2+1}+3^{k_2}}(1+z^{3^{k_2}})}
{1+z^{3^{k_2+2}}}=
\Psi(z)\frac {z^4}
{1+z^{9}}
=\sum _{n\ge0} ^{}c_nz^n,
$$
where
$$
c_n=\begin{cases} 
1,&\displaystyle\text {if\/ }(n)_3\in
{\bigcup_{X\in\{011,012,021,022\}}}
\{0,1\}^*X,\\
0,&\text {otherwise,}
\end{cases}
$$
while, if $k_2\ge1$, we have
$$
\frac {1} {1+z}\Psi(z)\frac {z^{3^{k_2+1}+3^{k_2}}(1+z^{3^{k_2}})}
{1+z^{3^{k_2+2}}}
=\sum _{n\ge0} ^{}c_nz^n,
$$
where
$$
c_n=\begin{cases} 
1,&\displaystyle\text {if\/ }(n)_3\in
{\bigcup_{X\in\{011,020,021,100\}}}
\{0,1\}^*X\{0,1\}^{k_2-1}0,\\
2,&\displaystyle\text {if\/ }(n)_3\in
{\bigcup_{X\in\{012,022\}}}
\{0,1\}^*X\{0,1\}^{k_2-1}0,\\
0,&\text {otherwise.}
\end{cases}
$$

{\em (3)} If $k_1-1>k_2=0$, we have
$$
\frac {1} {1+z}\Psi(z)\frac
{z^{3^{k_1}+3^{k_2}}(1+z^{3^{k_1}})(1+z^{3^{k_2}})}
{(1+z^{3^{k_1+1}})(1+z^{3^{k_2+1}})}=
\Psi(z)\frac {z^{3^{k_1}+1}(1+z^{3^{k_1}})}
{(1+z^3)(1+z^{3^{k_1+1}})}
=\sum _{n\ge0} ^{}c_nz^n,
$$
where
$$
c_n=\begin{cases} 
1,&\displaystyle\text {if\/ }(n)_3\in
{\bigcup_{X\in\{01,10\},\ 
Y\in\{01,02\}}}
\{0,1\}^*X\{0,1\}^{k_1-2}Y,\\
2,&\displaystyle\text {if\/ }(n)_3\in
{\bigcup_{X\in\{01,02\}}}
\{0,1\}^*02\{0,1\}^{k_1-2}X,\\
0,&\text {otherwise,}
\end{cases}
$$
while, if $k_1-1>k_2\ge1$, we have
$$
\frac {1} {1+z}\Psi(z)\frac
{z^{3^{k_1}+3^{k_2}}(1+z^{3^{k_1}})(1+z^{3^{k_2}})}
{(1+z^{3^{k_1+1}})(1+z^{3^{k_2+1}})}
=\sum _{n\ge0} ^{}c_nz^n,
$$
where
$$
c_n=\begin{cases} 
1,&\displaystyle\text {if\/ }(n)_3\in
{\bigcup_{X,Y\in\{01,10\}}}
\{0,1\}^*X\{0,1\}^{k_1-k_2-2}Y\{0,1\}^{k_2-1}0,\\
2,&\displaystyle\text {if\/ }(n)_3\in
\Bigg(
{\bigcup_{X\in\{01,10\}}}
\{0,1\}^*02\{0,1\}^{k_1-k_2-2}X\{0,1\}^{k_2-1}0\Bigg)\\
&\kern1.5cm\displaystyle
\cup
\Bigg(
{\bigcup_{X\in\{01,10\}}}
\{0,1\}^*X\{0,1\}^{k_1-k_2-2}02\{0,1\}^{k_2-1}0\Bigg),\\
4,&\displaystyle\text {if\/ }(n)_3\in
\{0,1\}^*02\{0,1\}^{k_1-k_2-2}02\{0,1\}^{k_2-1}0,\\
0,&\text {otherwise.}
\end{cases}
$$
\end{lemma}

For convenience, given a (left-infinite) string $s=\dots s_2s_1s_0$ 
consisting of $0$'s, $1$'s, and $2$'s, 
we introduce the following notation:
\begin{align*}
\Estring(s)&=\#(\text{maximal strings of $1$'s in $s$}),\\
\iso_2(s)&=\#(\text{isolated $0$'s and $1$'s in $s$, excluding $s_0$ and
$s_1$}),\\
\iso_3(s)&=\#(\text{isolated $0$'s and $1$'s in $s$, excluding $s_0,s_1,s_2$}),\\
\End(s;e)&=\begin{cases} 1&\text{if $s$ ends with the string $e$,}\\
0&\text{otherwise,}\end{cases}\\
\#100_1(s)&=\#(\text{occurrences of the substring $100$ in $s$ not
     involving $s_0$}),\\
\#011(s)&=\#(\text{occurrences of the substring $011$ in $s$}),
\end{align*}
and similarly for $\#020(s)$, $\#021(s)$, $\#102(s)$, $\#10(s)$, $\#01(s)$.

\begin{proposition} \label{prop:Psi-3}
Let $s$ denote the $3$-adic expansion of the non-negative integer $n$.
The coefficient of $z^n$ in the series $\Psi^3(z)$, when reduced
modulo~$27$, equals
\begin{enumerate} 
\item $1$, if, and only if, 
$s$ is either an element of
$\{0,1\}^*10$ and satisfies
$\Estring(s)\equiv2~(\text{\em mod }3)$ and
$$\tfrac {2\Estring(s)-1} {3}-\iso_2(s)+\#011(s)+\#100_1(s)
\equiv0~(\text{\em mod }3),
$$
or of\/
$\{0,1\}^*00$ and satisfies
$\Estring(s)\equiv0~(\text{\em mod }3)$ and
$$\tfrac {2\Estring(s)} {3}-\iso_2(s)+\#011(s)+\#100_1(s)
\equiv0~(\text{\em mod }3);
$$
\item $4$, if, and only if, 
$s$ is either an element of
$\{0,1\}^*10$ and satisfies
$\Estring(s)\equiv1~(\text{\em mod }3)$ and
$$\tfrac {2\Estring(s)-2} {3}-\iso_2(s)+\#011(s)+\#100_1(s)
\equiv0~(\text{\em mod }3),
$$
or of\/
$\{0,1\}^*00$ and satisfies
$\Estring(s)\equiv2~(\text{\em mod }3)$ and
$$\tfrac {2\Estring(s)-1} {3}-\iso_2(s)+\#011(s)+\#100_1(s)
\equiv0~(\text{\em mod }3);
$$
\item $7$, if, and only if, 
$s$ is either an element of
$\{0,1\}^*10$ and satisfies
$\Estring(s)\equiv0~(\text{\em mod }3)$ and
$$\tfrac {2\Estring(s)-3} {3}-\iso_2(s)+\#011(s)+\#100_1(s)
\equiv2~(\text{\em mod }3),
$$
or of\/
$\{0,1\}^*00$ and satisfies
$\Estring(s)\equiv1~(\text{\em mod }3)$ and
$$\tfrac {2\Estring(s)-2} {3}-\iso_2(s)+\#011(s)+\#100_1(s)
\equiv2~(\text{\em mod }3);
$$
\item $10$, if, and only if, 
$s$ is either an element of
$\{0,1\}^*10$ and satisfies
$\Estring(s)\equiv2~(\text{\em mod }3)$ and
$$\tfrac {2\Estring(s)-1} {3}-\iso_2(s)+\#011(s)+\#100_1(s)
\equiv1~(\text{\em mod }3),
$$
or of\/
$\{0,1\}^*00$ and satisfies
$\Estring(s)\equiv0~(\text{\em mod }3)$ and
$$\tfrac {2\Estring(s)} {3}-\iso_2(s)+\#011(s)+\#100_1(s)
\equiv1~(\text{\em mod }3);
$$
\item $13$, if, and only if, 
$s$ is either an element of
$\{0,1\}^*10$ and satisfies
$\Estring(s)\equiv1~(\text{\em mod }3)$ and
$$\tfrac {2\Estring(s)-2} {3}-\iso_2(s)+\#011(s)+\#100_1(s)
\equiv1~(\text{\em mod }3),
$$
or of\/
$\{0,1\}^*00$ and satisfies
$\Estring(s)\equiv2~(\text{\em mod }3)$ and
$$\tfrac {2\Estring(s)-1} {3}-\iso_2(s)+\#011(s)+\#100_1(s)
\equiv1~(\text{\em mod }3);
$$
\item $16$, if, and only if, 
$s$ is either an element of
$\{0,1\}^*10$ and satisfies
$\Estring(s)\equiv0~(\text{\em mod }3)$ and
$$\tfrac {2\Estring(s)-3} {3}-\iso_2(s)+\#011(s)+\#100_1(s)
\equiv0~(\text{\em mod }3),
$$
or of\/
$\{0,1\}^*00$ and satisfies
$\Estring(s)\equiv1~(\text{\em mod }3)$ and
$$\tfrac {2\Estring(s)-2} {3}-\iso_2(s)+\#011(s)+\#100_1(s)
\equiv0~(\text{\em mod }3);
$$
\item $19$, if, and only if, 
$s$ is either an element of
$\{0,1\}^*10$ and satisfies
$\Estring(s)\equiv2~(\text{\em mod }3)$ and
$$\tfrac {2\Estring(s)-1} {3}-\iso_2(s)+\#011(s)+\#100_1(s)
\equiv2~(\text{\em mod }3),
$$
or of\/
$\{0,1\}^*00$ and satisfies
$\Estring(s)\equiv0~(\text{\em mod }3)$ and
$$\tfrac {2\Estring(s)} {3}-\iso_2(s)+\#011(s)+\#100_1(s)
\equiv2~(\text{\em mod }3);
$$
\item $22$, if, and only if, 
$s$ is either an element of
$\{0,1\}^*10$ and satisfies
$\Estring(s)\equiv1~(\text{\em mod }3)$ and
$$\tfrac {2\Estring(s)-2} {3}-\iso_2(s)+\#011(s)+\#100_1(s)
\equiv2~(\text{\em mod }3),
$$
or of\/
$\{0,1\}^*00$ and satisfies
$\Estring(s)\equiv2~(\text{\em mod }3)$ and
$$\tfrac {2\Estring(s)-1} {3}-\iso_2(s)+\#011(s)+\#100_1(s)
\equiv2~(\text{\em mod }3);
$$
\item $25$, if, and only if, 
$s$ is either an element of
$\{0,1\}^*10$ and satisfies
$\Estring(s)\equiv0~(\text{\em mod }3)$ and
$$\tfrac {2\Estring(s)-3} {3}-\iso_2(s)+\#011(s)+\#100_1(s)
\equiv1~(\text{\em mod }3),
$$
or of\/
$\{0,1\}^*00$ and satisfies
$\Estring(s)\equiv1~(\text{\em mod }3)$ and
$$\tfrac {2\Estring(s)-2} {3}-\iso_2(s)+\#011(s)+\#100_1(s)
\equiv1~(\text{\em mod }3);
$$
\item $9$, if, and only if, 
the $3$-adic expansion of $n$ is an element of
$$
\Big(\{0,1\}^*02\{0,1\}^*02\{0,1\}^*0\Big)\cup
\bigg(
\bigcup_{X\in\{011,012,021,022\}}
\{0,1\}^*X\bigg);
$$
\item $18$, if, and only if, 
the $3$-adic expansion of $n$ is an element of
$$
\bigg(\bigcup_{X\in\{01,02\}}
\{0,1\}^*02\{0,1\}^*X\bigg)\cup
\bigg(\bigcup_{X\in\{012,022\}}
\{0,1\}^*X\{0,1\}^*0\bigg);
$$
\item $3$, if, and only if, 
the $3$-adic expansion of $n$ is an element of
$$
\bigg(\bigcup_{k\ge0,\ X\in\{01,02\}}\big(11^*00^*\big)^{3k+1}11^*X
\bigg)\cup\bigg(
\bigcup_{k\ge0,\ X\in\{01,02\}}\big(11^*00^*\big)^{3k+2}11^*0^*0X
\bigg)\cup\{1,2\};
$$
\item $12$, if, and only if, 
the $3$-adic expansion of $n$ is an element of
$$
\bigg(\bigcup_{k\ge0,\ X\in\{01,02\}}\big(11^*00^*\big)^{3k}11^*X
\bigg)\cup\bigg(
\bigcup_{k\ge0,\ X\in\{01,02\}}\big(11^*00^*\big)^{3k+1}11^*0^*0X
\bigg);
$$
\item $21$, if, and only if, 
the $3$-adic expansion of $n$ is an element of
$$
\bigg(\bigcup_{k\ge0,\ X\in\{01,02\}}\big(11^*00^*\big)^{3k+2}11^*X
\bigg)\cup\bigg(
\bigcup_{k\ge0,\ X\in\{01,02\}}\big(11^*00^*\big)^{3k}11^*0^*0X
\bigg);
$$
\item $6$, if, and only if, 
$s$ is either an element of
$\{0,1\}^*02\{0,1\}^*00$ and satisfies
$$
\#020(s)+\#021(s)+\#102(s)-\#01(s)-\#10(s)
\equiv0~(\text{\em mod }3),
$$
or an element of
$\{0,1\}^*02\{0,1\}^*10$ and satisfies
$$
\#020(s)+\#021(s)+\#102(s)-\#01(s)-\#10(s)
\equiv2~(\text{\em mod }3),
$$
or an element of
$\{0,1\}^*020$ and satisfies
$$
\#020(s)+\#021(s)+\#102(s)-\#01(s)-\#10(s)
-1
\equiv0~(\text{\em mod }3);
$$
\item $15$, if, and only if, 
$s$ is either an element of
$\{0,1\}^*02\{0,1\}^*00$ and satisfies
$$
\#020(s)+\#021(s)+\#102(s)-\#01(s)-\#10(s)
\equiv1~(\text{\em mod }3),
$$
or an element of
$\{0,1\}^*02\{0,1\}^*10$ and satisfies
$$
\#020(s)+\#021(s)+\#102(s)-\#01(s)-\#10(s)
\equiv0~(\text{\em mod }3),
$$
or an element of
$\{0,1\}^*020$ and satisfies
$$
\#020(s)+\#021(s)+\#102(s)-\#01(s)-\#10(s)
-1
\equiv1~(\text{\em mod }3);
$$
\item $24$, if, and only if, 
$s$ is either an element of
$\{0,1\}^*02\{0,1\}^*00$ and satisfies
$$
\#020(s)+\#021(s)+\#102(s)-\#01(s)-\#10(s)
\equiv2~(\text{\em mod }3),
$$
or an element of
$\{0,1\}^*02\{0,1\}^*10$ and satisfies
$$
\#020(s)+\#021(s)+\#102(s)-\#01(s)-\#10(s)
\equiv1~(\text{\em mod }3),
$$
or an element of
$\{0,1\}^*020$ and satisfies
$$
\#020(s)+\#021(s)+\#102(s)-\#01(s)-\#10(s)
-1
\equiv2~(\text{\em mod }3);
$$
\item In all other cases, this coefficient is divisible by $27$. In
particular, it is never congruent to $2$ modulo~$3$.
\end{enumerate}
\end{proposition}

If the above proposition is specialised to modulus~9, then one
obtains the following result.

\begin{corollary} \label{cor:Psi-3}
The coefficient of $z^n$ in the series $\Psi^3(z)$, when reduced
modulo~$9$, equals
\begin{enumerate} 
\item $1$, if, and only if, 
the $3$-adic expansion of $n$ is an element of
$$
\{0\}\cup\bigcup_{k\ge0}\big(00^*11^*\big)^{3k+2}0\cup
\bigcup_{k\ge0}\big(00^*11^*\big)^{3k}0^*00;
$$
\item $4$, if, and only if, 
the $3$-adic expansion of $n$ is an element of
$$
\bigcup_{k\ge0}\big(00^*11^*\big)^{3k+1}0\cup
\bigcup_{k\ge0}\big(00^*11^*\big)^{3k+2}0^*00;
$$
\item $7$, if, and only if, 
the $3$-adic expansion of $n$ is an element of
$$
\bigcup_{k\ge0}\big(00^*11^*\big)^{3k+3}0\cup
\bigcup_{k\ge0}\big(00^*11^*\big)^{3k+1}0^*00;
$$
\item $3$, if, and only if, 
the $3$-adic expansion of $n$ is an element of
$$
\{0,1\}^*01\cup\{0,1\}^*02,
$$
\item $6$, if, and only if, 
the $3$-adic expansion of $n$ is an element of
$$
\{0,1\}^*02\{0,1\}^*0.
$$
\item In all other cases, this coefficient is divisible by $9$. In
particular, it is never congruent to $2$ modulo~$3$.
\end{enumerate}
\end{corollary}

Now we turn our attention to coefficient extraction
from $\Psi^5(z)$ modulo~$27$. By Lemma~\ref{lem:Psi-2}, we have
\begin{multline} \label{eq:Psi-5mod27}
\Psi^5(z)=\frac {1} {(1+z)^2}\Psi(z)\Bigg(
1
+6\sum _{k_1\ge0} ^{}\frac {z^{3^{k_1}}(1+z^{3^{k_1}})} {1+z^{3^{k_1+1}}}
+9\sum _{k_1>k_2\ge0} ^{}\frac {z^{3^{k_1}+3^{k_2}}
(1+z^{3^{k_1}})(1+z^{3^{k_{2}}})}
{(1+z^{3^{k_1+1}})(1+z^{3^{k_2+1}})}\\
+9\sum _{k_1\ge0} ^{}\frac {z^{2\cdot3^{k_1}}(1+z^{3^{k_1}})^2} 
{(1+z^{3^{k_1+1}})^2}
\Bigg)\quad 
\text {modulo }27.
\end{multline}
The following lemma provides the means for coefficient extraction
from $\Psi^5(z)$ modulo any power of~$3$.

\begin{lemma} \label{lem:Psi-5exp}
{\em (1)} We have
$$
\frac {1} {(1+z)^2}\Psi(z)
=\sum _{n\ge0} ^{}c_nz^n,
$$
where
$$
c_n=\begin{cases} 
1,&\displaystyle\text {if\/ }(n)_3\in
{\bigcup_{X\in\{00,02\}}}
\{0,1\}^*X,\\
-1,&\displaystyle\text {if\/ }(n)_3\in
\{0,1\}^*01,\\
0,&\text {otherwise.}
\end{cases}
$$

{\em (2)} If $k_1=0$, we have
$$
\frac {1} {(1+z)^2}\Psi(z)\frac {z^{3^{k_1}}(1+z^{3^{k_1}})}
{1+z^{3^{k_1+1}}}=
\Psi(z)\frac {z}
{(1+z)(1+z^{3})}
=\sum _{n\ge0} ^{}c_nz^n,
$$
where
$$
c_n=\begin{cases} 
1,&\displaystyle\text {if\/ }(n)_3\in
\{0,1\}^*01,\\
0,&\text {otherwise,}
\end{cases}
$$
if $k_1=1$, we have
$$
\frac {1} {(1+z)^2}\Psi(z)\frac {z^{3^{k_1}}(1+z^{3^{k_1}})}
{1+z^{3^{k_1+1}}}=
\Psi(z)\frac {z(1-z+z^2)}
{(1+z)(1+z^{9})}
=\sum _{n\ge0} ^{}c_nz^n,
$$
where
$$
c_n=\begin{cases} 
1,&\displaystyle\text {if\/ }(n)_3\in
{\bigcup_{X\in\{010,012,020,022\}}}
\{0,1\}^*X,\\
-1,&\displaystyle\text {if\/ }(n)_3\in
{\bigcup_{X\in\{011,021\}}}
\{0,1\}^*X,\\
0,&\text {otherwise,}
\end{cases}
$$
while, if $k_1\ge2$, we have
$$
\frac {1} {(1+z)^2}\Psi(z)\frac {z^{3^{k_1}}(1+z^{3^{k_1}})}
{1+z^{3^{k_1+1}}}=\sum _{n\ge0} ^{}c_nz^n,
$$
where
$$
c_n=\begin{cases} 
1,&\displaystyle\text {if\/ }(n)_3\in
{\bigcup_{X\in\{01,10\},\ Y\in\{00,02\}}}
\{0,1\}^*X\{0,1\}^{k_1-2}Y,\\
-1,&\displaystyle\text {if\/ }(n)_3\in
{\bigcup_{X\in\{01,10\}}}
\{0,1\}^*X\{0,1\}^{k_1-2}01,\\
2,&\displaystyle\text {if\/ }(n)_3\in
{\bigcup_{X\in\{00,02\}}}
\{0,1\}^*02\{0,1\}^{k_1-2}X,\\
-2,&\displaystyle\text {if\/ }(n)_3\in
\{0,1\}^*02\{0,1\}^{k_1-2}01,\\
0,&\text {otherwise.}
\end{cases}
$$

{\em (3)} If $k_2=0$, we have
$$
\frac {1} {(1+z)^2}\Psi(z)\frac {z^{3^{k_2+1}+3^{k_2}}(1+z^{3^{k_2}})}
{1+z^{3^{k_2+2}}}=
\Psi(z)\frac {z^4}
{(1+z)(1+z^{9})}
=\sum _{n\ge0} ^{}c_nz^n,
$$
where
$$
c_n=\begin{cases} 
1,&\displaystyle\text {if\/ }(n)_3\in
{\bigcup_{X\in\{011,021\}}}
\{0,1\}^*X,\\
0,&\text {otherwise,}
\end{cases}
$$
if $k_2=1$, we have
$$
\frac {1} {(1+z)^2}\Psi(z)\frac {z^{3^{k_2+1}+3^{k_2}}(1+z^{3^{k_2}})}
{1+z^{3^{k_2+2}}}=
\Psi(z)\frac {z^{12}(1-z+z^2)}
{(1+z)(1+z^{27})}
=\sum _{n\ge0} ^{}c_nz^n,
$$
where
$$
c_n=\begin{cases} 
1,&\displaystyle\text {if\/ }(n)_3\in
{\bigcup_{X\in\{0110,0112,0120,0122,0210,0212,0220,0222\}}}
\{0,1\}^*X,\\
-1,&\displaystyle\text {if\/ }(n)_3\in
{\bigcup_{X\in\{0111,0121,0211,0221\}}}
\{0,1\}^*X,\\
0,&\text {otherwise,}
\end{cases}
$$
while, if $k_2\ge2$, we have
$$
\frac {1} {(1+z)^2}\Psi(z)\frac {z^{3^{k_2+1}+3^{k_2}}(1+z^{3^{k_2}})}
{1+z^{3^{k_2+2}}}
=\sum _{n\ge0} ^{}c_nz^n,
$$
where
$$
c_n=\begin{cases} 
1,&\displaystyle\text {if\/ }(n)_3\in
{\bigcup_{X\in\{011,020,021,100\},\
Y\in\{00,02\}}}
\{0,1\}^*X\{0,1\}^{k_2-2}Y,\\
-1,&\displaystyle\text {if\/ }(n)_3\in
{\bigcup_{X\in\{011,020,021,100\}}}
\{0,1\}^*X\{0,1\}^{k_2-2}01,\\
2,&\displaystyle\text {if\/ }(n)_3\in
{\bigcup_{X\in\{012,022\},\ Y\in\{00,02\}}}
\{0,1\}^*X\{0,1\}^{k_2-2}Y,\\
-2,&\displaystyle\text {if\/ }(n)_3\in
{\bigcup_{X\in\{012,022\}}}
\{0,1\}^*X\{0,1\}^{k_2-2}01,\\
0,&\text {otherwise.}
\end{cases}
$$

{\em (4)} If $k_1-1>k_2=0$, we have
\begin{multline*}
\frac {1} {(1+z)^2}\Psi(z)\frac
{z^{3^{k_1}+3^{k_2}}(1+z^{3^{k_1}})(1+z^{3^{k_2}})}
{(1+z^{3^{k_1+1}})(1+z^{3^{k_2+1}})}\\=
\Psi(z)\frac {z^{3^{k_1}+1}(1+z^{3^{k_1}})}
{(1+z)(1+z^3)(1+z^{3^{k_1+1}})}
=\sum _{n\ge0} ^{}c_nz^n,
\end{multline*}
where
$$
c_n=\begin{cases} 
1,&\displaystyle\text {if\/ }(n)_3\in
{\bigcup_{X\in\{01,10\}}}
\{0,1\}^*X\{0,1\}^{k_1-2}01,\\
2,&\displaystyle\text {if\/ }(n)_3\in
\{0,1\}^*02\{0,1\}^{k_1-2}01,\\
0,&\text {otherwise,}
\end{cases}
$$
if $k_1-1>k_2=1$, we have
\begin{multline*}
\frac {1} {(1+z)^2}\Psi(z)\frac
{z^{3^{k_1}+3^{k_2}}(1+z^{3^{k_1}})(1+z^{3^{k_2}})}
{(1+z^{3^{k_1+1}})(1+z^{3^{k_2+1}})}\\=
\Psi(z)\frac {z^{3^{k_1}+3}(1-z+z^2)(1+z^{3^{k_1}})}
{(1+z)(1+z^9)(1+z^{3^{k_1+1}})}
=\sum _{n\ge0} ^{}c_nz^n,
\end{multline*}
where
$$
c_n=\begin{cases} 
1,&\displaystyle\text {if\/ }(n)_3\in
{\bigcup_{X\in\{01,10\},\
Y\in\{010,012,020,022\}}}
\{0,1\}^*X\{0,1\}^{k_1-3}Y,\\
-1,&\displaystyle\text {if\/ }(n)_3\in
{\bigcup_{X\in\{01,10\},\ Y\in\{011,021\}}}
\{0,1\}^*X\{0,1\}^{k_1-3}Y,\\
2,&\displaystyle\text {if\/ }(n)_3\in
{\bigcup_{X\in\{010,012,020,022\}}}
\{0,1\}^*02\{0,1\}^{k_1-3}X,\\
-2,&\displaystyle\text {if\/ }(n)_3\in
{\bigcup_{X\in\{011,021\}}}
\{0,1\}^*02\{0,1\}^{k_1-3}X,\\
0,&\text {otherwise,}
\end{cases}
$$
while, if $k_1-1>k_2\ge2$, we have
$$
\frac {1} {(1+z)^2}\Psi(z)\frac
{z^{3^{k_1}+3^{k_2}}(1+z^{3^{k_1}})(1+z^{3^{k_2}})}
{(1+z^{3^{k_1+1}})(1+z^{3^{k_2+1}})}
=\sum _{n\ge0} ^{}c_nz^n,
$$
where
{\allowdisplaybreaks
$$
c_n=\begin{cases} 
1,&\displaystyle\text {if\/ }(n)_3\in
{\bigcup_{X,Y\in\{01,10\},\ 
Z\in\{00,02\}}}
\{0,1\}^*X\{0,1\}^{k_1-k_2-2}Y\{0,1\}^{k_2-2}Z,\\
-1,&\displaystyle\text {if\/ }(n)_3\in
{\bigcup_{X,Y\in\{01,10\}}}
\{0,1\}^*X\{0,1\}^{k_1-k_2-2}Y\{0,1\}^{k_2-2}01,\\
2,&\displaystyle\text {if\/ }(n)_3\in
\Bigg(
{\bigcup_{X\in\{01,10\},\ Y\in\{00,02\}}}
\{0,1\}^*02\{0,1\}^{k_1-k_2-2}X\{0,1\}^{k_2-2}Y\Bigg)\\
&\kern1.5cm\displaystyle
\cup
\Bigg(
{\bigcup_{X\in\{01,10\},\ Y\in\{00,02\}}}
\{0,1\}^*X\{0,1\}^{k_1-k_2-2}02\{0,1\}^{k_2-2}Y\Bigg),\\
-2,&\displaystyle\text {if\/ }(n)_3\in
\Bigg(
{\bigcup_{X\in\{01,10\}}}
\{0,1\}^*02\{0,1\}^{k_1-k_2-2}X\{0,1\}^{k_2-2}01\Bigg)\\
&\kern1.5cm\displaystyle
\cup
\Bigg(
{\bigcup_{X\in\{01,10\}}}
\{0,1\}^*X\{0,1\}^{k_1-k_2-2}02\{0,1\}^{k_2-2}01\Bigg),\\
4,&\displaystyle\text {if\/ }(n)_3\in
{\bigcup_{X\in\{00,02\}}}
\{0,1\}^*02\{0,1\}^{k_1-k_2-2}02\{0,1\}^{k_2-2}X,\\
-4,&\displaystyle\text {if\/ }(n)_3\in
\{0,1\}^*02\{0,1\}^{k_1-k_2-2}02\{0,1\}^{k_2-2}01,\\
0,&\text {otherwise.}
\end{cases}
$$}

{\em (5)} If $k_1=0$, we have
$$
\frac {1} {(1+z)^2}\Psi(z)\frac {z^{2\cdot 3^{k_1}}(1+z^{3^{k_1}})^2}
{(1+z^{3^{k_1+1}})^2}=
\Psi(z)\frac {z^2}
{(1+z^{3})^2}
=\sum _{n\ge0} ^{}c_nz^n,
$$
where
$$
c_n=\begin{cases} 
1,&\displaystyle\text {if\/ }(n)_3\in
{\bigcup_{X\in\{002,010,022,100\}}}
\{0,1\}^*X,\\
-1,&\displaystyle\text {if\/ }(n)_3\in
{\bigcup_{X\in\{012,020\}}}
\{0,1\}^*X,\\
0,&\text {otherwise,}
\end{cases}
$$
if $k_1=1$, we have
$$
\frac {1} {(1+z)^2}\Psi(z)\frac {z^{2\cdot 3^{k_1}}(1+z^{3^{k_1}})^2}
{(1+z^{3^{k_1+1}})^2}=
\Psi(z)\frac {z^{6}(1-z+z^2)(1+z^3)}
{(1+z)(1+z^{9})^2}
=\sum _{n\ge0} ^{}c_nz^n,
$$
where
$$
c_n=\begin{cases} 
1,&\displaystyle\text {if\/ }(n)_3\in
{\bigcup_{X\in\{0020,0022,0110,0112,0121,0211,0220,0222,1010,1012\}}}
\{0,1\}^*X,\\
-1,&\displaystyle\text {if\/ }(n)_3\in
{\bigcup_{X\in\{0021,0111,0120,0122,0210,0212,0221,1011\}}}
\{0,1\}^*X,\\
2,&\displaystyle\text {if\/ }(n)_3\in
{\bigcup_{X\in\{0100,0102,0201,1000,1002\}}}
\{0,1\}^*X,\\
-2,&\displaystyle\text {if\/ }(n)_3\in
{\bigcup_{X\in\{0101,0200,0202,1001\}}}
\{0,1\}^*X,\\
0,&\text {otherwise,}
\end{cases}
$$
while, if $k_1\ge2$, we have
$$
\frac {1} {(1+z)^2}\Psi(z)\frac {z^{2\cdot3^{k_1}}(1+z^{3^{k_1}})^2}
{(1+z^{3^{k_1+1}})^2}
=\sum _{n\ge0} ^{}c_nz^n,
$$
where
$$
c_n=\begin{cases} 
1,&\displaystyle\text {if\/ }(n)_3\in
{\bigcup_{X\in\{002,102\},\
Y\in\{00,02\}}}
\{0,1\}^*X\{0,1\}^{k_1-2}Y,\\
-1,&\displaystyle\text {if\/ }(n)_3\in
{\bigcup_{X\in\{002,102\}}}
\{0,1\}^*X\{0,1\}^{k_1-2}01,\\
3,&\displaystyle\text {if\/ }(n)_3\in
{\bigcup_{X\in\{010,011,100,101\},\ Y\in\{00,02\}}}
\{0,1\}^*X\{0,1\}^{k_1-2}Y\\
&\kern1.5cm\displaystyle
\cup
{\bigcup_{X\in\{020,021\}}}
\{0,1\}^*X\{0,1\}^{k_1-2}01,\\
-3,&\displaystyle\text {if\/ }(n)_3\in
{\bigcup_{X\in\{010,011,100,101\}}}
\{0,1\}^*X\{0,1\}^{k_1-2}01\\
&\kern1.5cm\displaystyle
\cup
{\bigcup_{X\in\{020,021\},\ Y\in\{00,02\}}}
\{0,1\}^*X\{0,1\}^{k_1-2}Y,\\
0,&\text {otherwise.}
\end{cases}
$$
\end{lemma}

\allowdisplaybreaks
\begin{proposition} \label{prop:Psi-5}
Let $s$ denote the $3$-adic expansion of the non-negative integer $n$.
The coefficient of $z^n$ in the series $\Psi^5(z)$, when reduced
modulo~$27$, equals
\begin{enumerate} 
\item $1$, if, and only if, 
$s$ is either an element of
$\bigcup_{X\in\{000,002\}}\{0,1\}^*X$ and satisfies
$\Estring(s)\equiv0~(\text{\em mod }3)$ and
\begin{multline*}
\tfrac {4\Estring(s)} {3}-\iso_3(s)+\#011(s)+\#100_1(s)\\
+\End(s;002)
+\End(s;1000)
+\End(s;1002)
\equiv0~(\text{\em mod }3),
\end{multline*}
or of\/
$\bigcup_{X\in\{100,102\}}\{0,1\}^*X$ and satisfies
$\Estring(s)\equiv2~(\text{\em mod }3)$ and
\begin{multline*}
\tfrac {4\Estring(s)-2} {3}-\iso_3(s)+\#011(s)+\#100_1(s)\\
+\End(s;100)
+2\End(s;0102)
\equiv0~(\text{\em mod }3);
\end{multline*}
\item $10$, if, and only if, 
$s$ is either an element of
$\bigcup_{X\in\{000,002\}}\{0,1\}^*X$ and satisfies
$\Estring(s)\equiv0~(\text{\em mod }3)$ and
\begin{multline*}
\tfrac {4\Estring(s)} {3}-\iso_3(s)+\#011(s)+\#100_1(s)\\
+\End(s;002)
+\End(s;1000)
+\End(s;1002)
\equiv1~(\text{\em mod }3),
\end{multline*}
or of\/
$\bigcup_{X\in\{100,102\}}\{0,1\}^*X$ and satisfies
$\Estring(s)\equiv2~(\text{\em mod }3)$ and
\begin{multline*}
\tfrac {4\Estring(s)-2} {3}-\iso_3(s)+\#011(s)+\#100_1(s)\\
+\End(s;100)
+2\End(s;0102)
\equiv1~(\text{\em mod }3);
\end{multline*}
\item $19$, if, and only if, 
$s$ is either an element of
$\bigcup_{X\in\{000,002\}}\{0,1\}^*X$ and satisfies
$\Estring(s)\equiv0~(\text{\em mod }3)$ and
\begin{multline*}
\tfrac {4\Estring(s)} {3}-\iso_3(s)+\#011(s)+\#100_1(s)\\
+\End(s;002)
+\End(s;1000)
+\End(s;1002)
\equiv2~(\text{\em mod }3),
\end{multline*}
or of\/
$\bigcup_{X\in\{100,102\}}\{0,1\}^*X$ and satisfies
$\Estring(s)\equiv2~(\text{\em mod }3)$ and
\begin{multline*}
\tfrac {4\Estring(s)-2} {3}-\iso_3(s)+\#011(s)+\#100_1(s)\\
+\End(s;100)
+2\End(s;0102)
\equiv2~(\text{\em mod }3);
\end{multline*}
\item $4$, if, and only if, 
$s$ is either an element of
$\bigcup_{X\in\{000,002\}}\{0,1\}^*X$ and satisfies
$\Estring(s)\equiv1~(\text{\em mod }3)$ and
\begin{multline*}
\tfrac {4\Estring(s)-1} {3}-\iso_3(s)+\#011(s)+\#100_1(s)\\
+\End(s;002)
+\End(s;1000)
+\End(s;1002)
\equiv0~(\text{\em mod }3),
\end{multline*}
or of\/
$\bigcup_{X\in\{100,102\}}\{0,1\}^*X$ and satisfies
$\Estring(s)\equiv0~(\text{\em mod }3)$ and
\begin{multline*}
\tfrac {4\Estring(s)-3} {3}-\iso_3(s)+\#011(s)+\#100_1(s)\\
+\End(s;100)
+2\End(s;0102)
\equiv0~(\text{\em mod }3);
\end{multline*}
\item $13$, if, and only if, 
$s$ is either an element of
$\bigcup_{X\in\{000,002\}}\{0,1\}^*X$ and satisfies
$\Estring(s)\equiv1~(\text{\em mod }3)$ and
\begin{multline*}
\tfrac {4\Estring(s)-1} {3}-\iso_3(s)+\#011(s)+\#100_1(s)\\
+\End(s;002)
+\End(s;1000)
+\End(s;1002)
\equiv1~(\text{\em mod }3),
\end{multline*}
or of\/
$\bigcup_{X\in\{100,102\}}\{0,1\}^*X$ and satisfies
$\Estring(s)\equiv0~(\text{\em mod }3)$ and
\begin{multline*}
\tfrac {4\Estring(s)-3} {3}-\iso_3(s)+\#011(s)+\#100_1(s)\\
+\End(s;100)
+2\End(s;0102)
\equiv1~(\text{\em mod }3);
\end{multline*}
\item $22$, if, and only if, 
$s$ is either an element of
$\bigcup_{X\in\{000,002\}}\{0,1\}^*X$ and satisfies
$\Estring(s)\equiv1~(\text{\em mod }3)$ and
\begin{multline*}
\tfrac {4\Estring(s)-1} {3}-\iso_3(s)+\#011(s)+\#100_1(s)\\
+\End(s;002)
+\End(s;1000)
+\End(s;1002)
\equiv2~(\text{\em mod }3),
\end{multline*}
or of\/
$\bigcup_{X\in\{100,102\}}\{0,1\}^*X$ and satisfies
$\Estring(s)\equiv0~(\text{\em mod }3)$ and
\begin{multline*}
\tfrac {4\Estring(s)-3} {3}-\iso_3(s)+\#011(s)+\#100_1(s)\\
+\End(s;100)
+2\End(s;0102)
\equiv2~(\text{\em mod }3);
\end{multline*}
\item $7$, if, and only if, 
$s$ is either an element of
$\bigcup_{X\in\{000,002\}}\{0,1\}^*X$ and satisfies
$\Estring(s)\equiv2~(\text{\em mod }3)$ and
\begin{multline*}
\tfrac {4\Estring(s)-2} {3}-\iso_3(s)+\#011(s)+\#100_1(s)\\
+\End(s;002)
+\End(s;1000)
+\End(s;1002)
\equiv0~(\text{\em mod }3),
\end{multline*}
or of\/
$\bigcup_{X\in\{100,102\}}\{0,1\}^*X$ and satisfies
$\Estring(s)\equiv1~(\text{\em mod }3)$ and
\begin{multline*}
\tfrac {4\Estring(s)-4} {3}-\iso_3(s)+\#011(s)+\#100_1(s)\\
+\End(s;100)
+2\End(s;0102)
\equiv0~(\text{\em mod }3);
\end{multline*}
\item $16$, if, and only if, 
$s$ is either an element of
$\bigcup_{X\in\{000,002\}}\{0,1\}^*X$ and satisfies
$\Estring(s)\equiv2~(\text{\em mod }3)$ and
\begin{multline*}
\tfrac {4\Estring(s)-2} {3}-\iso_3(s)+\#011(s)+\#100_1(s)\\
+\End(s;002)
+\End(s;1000)
+\End(s;1002)
\equiv1~(\text{\em mod }3),
\end{multline*}
or of\/
$\bigcup_{X\in\{100,102\}}\{0,1\}^*X$ and satisfies
$\Estring(s)\equiv1~(\text{\em mod }3)$ and
\begin{multline*}
\tfrac {4\Estring(s)-4} {3}-\iso_3(s)+\#011(s)+\#100_1(s)\\
+\End(s;100)
+2\End(s;0102)
\equiv1~(\text{\em mod }3);
\end{multline*}
\item $25$, if, and only if, 
$s$ is either an element of
$\bigcup_{X\in\{000,002\}}\{0,1\}^*X$ and satisfies
$\Estring(s)\equiv2~(\text{\em mod }3)$ and
\begin{multline*}
\tfrac {4\Estring(s)-2} {3}-\iso_3(s)+\#011(s)+\#100_1(s)\\
+\End(s;002)
+\End(s;1000)
+\End(s;1002)
\equiv2~(\text{\em mod }3),
\end{multline*}
or of\/
$\bigcup_{X\in\{100,102\}}\{0,1\}^*X$ and satisfies
$\Estring(s)\equiv1~(\text{\em mod }3)$ and
\begin{multline*}
\tfrac {4\Estring(s)-4} {3}-\iso_3(s)+\#011(s)+\#100_1(s)\\
+\End(s;100)
+2\End(s;0102)
\equiv2~(\text{\em mod }3);
\end{multline*}
\item $2$, if, and only if, 
$s$ is either an element of
$\{0,1\}^*001$ and satisfies
$\Estring(s)\equiv1~(\text{\em mod }3)$ and
\begin{equation*}
\tfrac {1-4\Estring(s)} {3}+\iso_3(s)-\#011(s)-\#100_1(s)
-\End(s;1001)
\equiv2~(\text{\em mod }3),
\end{equation*}
or of\/
$\{0,1\}^*101$ and satisfies
$\Estring(s)\equiv0~(\text{\em mod }3)$ and
\begin{equation*}
\tfrac {3-4\Estring(s)} {3}+\iso_3(s)-\#011(s)-\#100_1(s)
-2\End(s;1001)
\equiv2~(\text{\em mod }3);
\end{equation*}
\item $11$, if, and only if, 
$s$ is either an element of
$\{0,1\}^*001$ and satisfies
$\Estring(s)\equiv1~(\text{\em mod }3)$ and
\begin{equation*}
\tfrac {1-4\Estring(s)} {3}+\iso_3(s)-\#011(s)-\#100_1(s)
-\End(s;1001)
\equiv0~(\text{\em mod }3),
\end{equation*}
or of\/
$\{0,1\}^*101$ and satisfies
$\Estring(s)\equiv0~(\text{\em mod }3)$ and
\begin{equation*}
\tfrac {3-4\Estring(s)} {3}+\iso_3(s)-\#011(s)-\#100_1(s)
-2\End(s;1001)
\equiv0~(\text{\em mod }3);
\end{equation*}
\item $20$, if, and only if, 
$s$ is either an element of
$\{0,1\}^*001$ and satisfies
$\Estring(s)\equiv1~(\text{\em mod }3)$ and
\begin{equation*}
\tfrac {1-4\Estring(s)} {3}+\iso_3(s)-\#011(s)-\#100_1(s)
-\End(s;1001)
\equiv1~(\text{\em mod }3),
\end{equation*}
or of\/
$\{0,1\}^*101$ and satisfies
$\Estring(s)\equiv0~(\text{\em mod }3)$ and
\begin{equation*}
\tfrac {3-4\Estring(s)} {3}+\iso_3(s)-\#011(s)-\#100_1(s)
-2\End(s;1001)
\equiv1~(\text{\em mod }3);
\end{equation*}
\item $5$, if, and only if, 
$s$ is either an element of
$\{0,1\}^*001$ and satisfies
$\Estring(s)\equiv0~(\text{\em mod }3)$ and
\begin{equation*}
-\tfrac {4\Estring(s)} {3}+\iso_3(s)-\#011(s)-\#100_1(s)
-\End(s;1001)
\equiv2~(\text{\em mod }3),
\end{equation*}
or of\/
$\{0,1\}^*101$ and satisfies
$\Estring(s)\equiv2~(\text{\em mod }3)$ and
\begin{equation*}
\tfrac {2-4\Estring(s)} {3}+\iso_3(s)-\#011(s)-\#100_1(s)
-2\End(s;1001)
\equiv2~(\text{\em mod }3);
\end{equation*}
\item $14$, if, and only if, 
$s$ is either an element of
$\{0,1\}^*001$ and satisfies
$\Estring(s)\equiv0~(\text{\em mod }3)$ and
\begin{equation*}
-\tfrac {4\Estring(s)} {3}+\iso_3(s)-\#011(s)-\#100_1(s)
-\End(s;1001)
\equiv0~(\text{\em mod }3),
\end{equation*}
or of\/
$\{0,1\}^*101$ and satisfies
$\Estring(s)\equiv2~(\text{\em mod }3)$ and
\begin{equation*}
\tfrac {2-4\Estring(s)} {3}+\iso_3(s)-\#011(s)-\#100_1(s)
-2\End(s;1001)
\equiv0~(\text{\em mod }3);
\end{equation*}
\item $23$, if, and only if, 
$s$ is either an element of
$\{0,1\}^*001$ and satisfies
$\Estring(s)\equiv0~(\text{\em mod }3)$ and
\begin{equation*}
-\tfrac {4\Estring(s)} {3}+\iso_3(s)-\#011(s)-\#100_1(s)
-\End(s;1001)
\equiv1~(\text{\em mod }3),
\end{equation*}
or of\/
$\{0,1\}^*101$ and satisfies
$\Estring(s)\equiv2~(\text{\em mod }3)$ and
\begin{equation*}
\tfrac {2-4\Estring(s)} {3}+\iso_3(s)-\#011(s)-\#100_1(s)
-2\End(s;1001)
\equiv1~(\text{\em mod }3);
\end{equation*}
\item $8$, if, and only if, 
$s$ is either an element of
$\{0,1\}^*001$ and satisfies
$\Estring(s)\equiv2~(\text{\em mod }3)$ and
\begin{equation*}
-\tfrac {1+4\Estring(s)} {3}+\iso_3(s)-\#011(s)-\#100_1(s)
-\End(s;1001)
\equiv2~(\text{\em mod }3),
\end{equation*}
or of\/
$\{0,1\}^*101$ and satisfies
$\Estring(s)\equiv1~(\text{\em mod }3)$ and
\begin{equation*}
\tfrac {1-4\Estring(s)} {3}+\iso_3(s)-\#011(s)-\#100_1(s)
-2\End(s;1001)
\equiv2~(\text{\em mod }3);
\end{equation*}
\item $17$, if, and only if, 
$s$ is either an element of
$\{0,1\}^*001$ and satisfies
$\Estring(s)\equiv2~(\text{\em mod }3)$ and
\begin{equation*}
-\tfrac {1+4\Estring(s)} {3}+\iso_3(s)-\#011(s)-\#100_1(s)
-\End(s;1001)
\equiv0~(\text{\em mod }3),
\end{equation*}
or of\/
$\{0,1\}^*101$ and satisfies
$\Estring(s)\equiv1~(\text{\em mod }3)$ and
\begin{equation*}
\tfrac {1-4\Estring(s)} {3}+\iso_3(s)-\#011(s)-\#100_1(s)
-2\End(s;1001)
\equiv0~(\text{\em mod }3);
\end{equation*}
\item $26$, if, and only if, 
$s$ is either an element of
$\{0,1\}^*001$ and satisfies
$\Estring(s)\equiv2~(\text{\em mod }3)$ and
\begin{equation*}
-\tfrac {1+4\Estring(s)} {3}+\iso_3(s)-\#011(s)-\#100_1(s)
-\End(s;1001)
\equiv1~(\text{\em mod }3),
\end{equation*}
or of\/
$\{0,1\}^*101$ and satisfies
$\Estring(s)\equiv1~(\text{\em mod }3)$ and
\begin{equation*}
\tfrac {1-4\Estring(s)} {3}+\iso_3(s)-\#011(s)-\#100_1(s)
-2\End(s;1001)
\equiv1~(\text{\em mod }3);
\end{equation*}
\item $9$, if, and only if, 
the $3$-adic expansion of $n$ is an element of
\begin{multline*}
\bigg(\bigcup_{X\in\{0111,0121,0211,0221\}}\{0,1\}^*X\bigg)
\cup
\bigg({\bigcup_{X\in\{012,022\}}}
\{0,1\}^*X\{0,1\}^*01\bigg)\\
\cup
\bigg({\bigcup_{X\in\{011,021\}}}
\{0,1\}^*02\{0,1\}^*X\bigg)
\cup
\bigg({\bigcup_{X\in\{00,02\}}}
\{0,1\}^*02\{0,1\}^*02\{0,1\}^*X\bigg);
\end{multline*}
\item $18$, if, and only if, 
the $3$-adic expansion of $n$ is an element of
\begin{multline*}
\bigg(\bigcup_{X\in\{0110,0112,0220,0222\}}\{0,1\}^*X\bigg)
\cup
\bigg({\bigcup_{X\in\{012,022\},\ Y\in\{00,02\}}}
\{0,1\}^*X\{0,1\}^*Y\bigg)\\
\cup
\bigg({\bigcup_{X\in\{010,012,020,022\}}}
\{0,1\}^*02\{0,1\}^*X\bigg)
\cup
\Big(\{0,1\}^*02\{0,1\}^*02\{0,1\}^*01\Big);
\end{multline*}
\item $6$, if, and only if, 
the $3$-adic expansion of $n$ is an element of
\begin{multline*}
\bigg(\bigcup_{k\ge0}\big(11^*00^*\big)^{3k+2}0201\bigg)
\cup
\bigg(\bigcup_{k_1+k_2\equiv1\,(\text{\em mod }3)}
\big(11^*00^*\big)^{k_1}0200^*\big(11^*00^*\big)^{k_2}01\bigg)\\
\cup
\bigg(\bigcup_{k_1+k_2\equiv2\,(\text{\em mod }3)}
\big(11^*00^*\big)^{k_1}11^*0200^*\big(11^*00^*\big)^{k_2}01\bigg)\\
\cup
\bigg(\bigcup_{k_1+k_2\equiv0\,(\text{\em mod }3),\,
k_2>0}
\big(11^*00^*\big)^{k_1}02\big(11^*00^*\big)^{k_2}01\bigg)\\
\cup
\bigg(\bigcup_{k_1+k_2\equiv1\,(\text{\em mod }3),\,
k_2>0}
\big(11^*00^*\big)^{k_1}11^*02\big(11^*00^*\big)^{k_2}01\bigg)\\
\cup
\bigg(\bigcup_{k_1+k_2\equiv2\,(\text{\em mod }3)}
\big(11^*00^*\big)^{k_1}0200^*\big(11^*00^*\big)^{k_2}11^*01\bigg)\\
\cup
\bigg(\bigcup_{k_1+k_2\equiv0\,(\text{\em mod }3)}
\big(11^*00^*\big)^{k_1}11^*0200^*\big(11^*00^*\big)^{k_2}11^*01\bigg)\\
\cup
\bigg(\bigcup_{k_1+k_2\equiv1\,(\text{\em mod }3)}
\big(11^*00^*\big)^{k_1}02\big(11^*00^*\big)^{k_2}11^*01\bigg)\\
\cup
\bigg(\bigcup_{k_1+k_2\equiv2\,(\text{\em mod }3)}
\big(11^*00^*\big)^{k_1}11^*02\big(11^*00^*\big)^{k_2}11^*01\bigg)\\
\cup
\Big(\big(11^*00^*\big)^{3k+1}010\Big)
\cup
\Big(\big(11^*00^*\big)^{3k+2}012\Big)\\
\cup
\Big(\big(11^*00^*\big)^{3k}020\Big)
\cup
\Big(\big(11^*00^*\big)^{3k+2}022\Big)\\
\cup
\Big(\big(11^*00^*\big)^{3k}11^*010\Big)
\cup
\Big(\big(11^*00^*\big)^{3k+1}11^*012\Big)\\
\cup
\Big(\big(11^*00^*\big)^{3k}11^*020\Big)
\cup
\Big(\big(11^*00^*\big)^{3k+2}11^*022\Big);
\end{multline*}
\item $15$, if, and only if, 
the $3$-adic expansion of $n$ is an element of
\begin{multline*}
\bigg(\bigcup_{k\ge0}\big(11^*00^*\big)^{3k+1}0201\bigg)
\cup
\bigg(\bigcup_{k_1+k_2\equiv0\,(\text{\em mod }3)}
\big(11^*00^*\big)^{k_1}0200^*\big(11^*00^*\big)^{k_2}01\bigg)\\
\cup
\bigg(\bigcup_{k_1+k_2\equiv1\,(\text{\em mod }3)}
\big(11^*00^*\big)^{k_1}11^*0200^*\big(11^*00^*\big)^{k_2}01\bigg)\\
\cup
\bigg(\bigcup_{k_1+k_2\equiv2\,(\text{\em mod }3),\,
k_2>0}
\big(11^*00^*\big)^{k_1}02\big(11^*00^*\big)^{k_2}01\bigg)\\
\cup
\bigg(\bigcup_{k_1+k_2\equiv0\,(\text{\em mod }3),\,
k_2>0}
\big(11^*00^*\big)^{k_1}11^*02\big(11^*00^*\big)^{k_2}01\bigg)\\
\cup
\bigg(\bigcup_{k_1+k_2\equiv1\,(\text{\em mod }3)}
\big(11^*00^*\big)^{k_1}0200^*\big(11^*00^*\big)^{k_2}11^*01\bigg)\\
\cup
\bigg(\bigcup_{k_1+k_2\equiv2\,(\text{\em mod }3)}
\big(11^*00^*\big)^{k_1}11^*0200^*\big(11^*00^*\big)^{k_2}11^*01\bigg)\\
\cup
\bigg(\bigcup_{k_1+k_2\equiv0\,(\text{\em mod }3)}
\big(11^*00^*\big)^{k_1}02\big(11^*00^*\big)^{k_2}11^*01\bigg)\\
\cup
\bigg(\bigcup_{k_1+k_2\equiv1\,(\text{\em mod }3)}
\big(11^*00^*\big)^{k_1}11^*02\big(11^*00^*\big)^{k_2}11^*01\bigg)\\
\cup
\Big(\big(11^*00^*\big)^{3k}010\Big)
\cup
\Big(\big(11^*00^*\big)^{3k+1}012\Big)\\
\cup
\Big(\big(11^*00^*\big)^{3k+2}020\Big)
\cup
\Big(\big(11^*00^*\big)^{3k+1}022\Big)\\
\cup
\Big(\big(11^*00^*\big)^{3k+2}11^*010\Big)
\cup
\Big(\big(11^*00^*\big)^{3k}11^*012\Big)\\
\cup
\Big(\big(11^*00^*\big)^{3k+2}11^*020\Big)
\cup
\Big(\big(11^*00^*\big)^{3k+1}11^*022\Big);
\end{multline*}
\item $24$, if, and only if, 
the $3$-adic expansion of $n$ is an element of
\begin{multline*}
\bigg(\bigcup_{k\ge0}\big(11^*00^*\big)^{3k}0201\bigg)
\cup
\bigg(\bigcup_{k_1+k_2\equiv2\,(\text{\em mod }3)}
\big(11^*00^*\big)^{k_1}0200^*\big(11^*00^*\big)^{k_2}01\bigg)\\
\cup
\bigg(\bigcup_{k_1+k_2\equiv0\,(\text{\em mod }3)}
\big(11^*00^*\big)^{k_1}11^*0200^*\big(11^*00^*\big)^{k_2}01\bigg)\\
\cup
\bigg(\bigcup_{k_1+k_2\equiv1\,(\text{\em mod }3),\,
k_2>0}
\big(11^*00^*\big)^{k_1}02\big(11^*00^*\big)^{k_2}01\bigg)\\
\cup
\bigg(\bigcup_{k_1+k_2\equiv2\,(\text{\em mod }3),\,
k_2>0}
\big(11^*00^*\big)^{k_1}11^*02\big(11^*00^*\big)^{k_2}01\bigg)\\
\cup
\bigg(\bigcup_{k_1+k_2\equiv0\,(\text{\em mod }3)}
\big(11^*00^*\big)^{k_1}0200^*\big(11^*00^*\big)^{k_2}11^*01\bigg)\\
\cup
\bigg(\bigcup_{k_1+k_2\equiv1\,(\text{\em mod }3)}
\big(11^*00^*\big)^{k_1}11^*0200^*\big(11^*00^*\big)^{k_2}11^*01\bigg)\\
\cup
\bigg(\bigcup_{k_1+k_2\equiv2\,(\text{\em mod }3)}
\big(11^*00^*\big)^{k_1}02\big(11^*00^*\big)^{k_2}11^*01\bigg)\\
\cup
\bigg(\bigcup_{k_1+k_2\equiv0\,(\text{\em mod }3)}
\big(11^*00^*\big)^{k_1}11^*02\big(11^*00^*\big)^{k_2}11^*01\bigg)\\
\cup
\Big(\big(11^*00^*\big)^{3k+2}010\Big)
\cup
\Big(\big(11^*00^*\big)^{3k}012\Big)\\
\cup
\Big(\big(11^*00^*\big)^{3k+1}020\Big)
\cup
\Big(\big(11^*00^*\big)^{3k}022\Big)\\
\cup
\Big(\big(11^*00^*\big)^{3k+1}11^*010\Big)
\cup
\Big(\big(11^*00^*\big)^{3k+2}11^*012\Big)\\
\cup
\Big(\big(11^*00^*\big)^{3k+1}11^*020\Big)
\cup
\Big(\big(11^*00^*\big)^{3k}11^*022\Big);
\end{multline*}
\item $3$, if, and only if, 
the $3$-adic expansion of $n$ is an element of
\begin{multline*}
\bigg(\bigcup_{k\ge0,\,X\in\{00,02\}}\big(11^*00^*\big)^{3k+1}02X\bigg)\\
\cup
\bigg(\bigcup_{k_1+k_2\equiv0\,(\text{\em mod }3),\,X\in\{00,02\}}
\big(11^*00^*\big)^{k_1}0200^*\big(11^*00^*\big)^{k_2}X\bigg)\\
\cup
\bigg(\bigcup_{k_1+k_2\equiv1\,(\text{\em mod }3),\,X\in\{00,02\}}
\big(11^*00^*\big)^{k_1}11^*0200^*\big(11^*00^*\big)^{k_2}X\bigg)\\
\cup
\bigg(\bigcup_{k_1+k_2\equiv2\,(\text{\em mod }3),\,
k_2>0,\,X\in\{00,02\}}
\big(11^*00^*\big)^{k_1}02\big(11^*00^*\big)^{k_2}X\bigg)\\
\cup
\bigg(\bigcup_{k_1+k_2\equiv0\,(\text{\em mod }3),\,
k_2>0,\,X\in\{00,02\}}
\big(11^*00^*\big)^{k_1}11^*02\big(11^*00^*\big)^{k_2}X\bigg)\\
\cup
\bigg(\bigcup_{k_1+k_2\equiv1\,(\text{\em mod }3),\,X\in\{00,02\}}
\big(11^*00^*\big)^{k_1}0200^*\big(11^*00^*\big)^{k_2}11^*X\bigg)\\
\cup
\bigg(\bigcup_{k_1+k_2\equiv2\,(\text{\em mod }3),\,X\in\{00,02\}}
\big(11^*00^*\big)^{k_1}11^*0200^*\big(11^*00^*\big)^{k_2}11^*X\bigg)\\
\cup
\bigg(\bigcup_{k_1+k_2\equiv0\,(\text{\em mod }3),\,X\in\{00,02\}}
\big(11^*00^*\big)^{k_1}02\big(11^*00^*\big)^{k_2}11^*X\bigg)\\
\cup
\bigg(\bigcup_{k_1+k_2\equiv1\,(\text{\em mod }3),\,X\in\{00,02\}}
\big(11^*00^*\big)^{k_1}11^*02\big(11^*00^*\big)^{k_2}11^*X\bigg)\\
\cup
\Big(\big(11^*00^*\big)^{3k+1}011\Big)
\cup
\Big(\big(11^*00^*\big)^{3k+2}021\Big)\\
\cup
\Big(\big(11^*00^*\big)^{3k}11^*011\Big)
\cup
\Big(\big(11^*00^*\big)^{3k+2}11^*021\Big);
\end{multline*}
\item $12$, if, and only if, 
the $3$-adic expansion of $n$ is an element of
\begin{multline*}
\bigg(\bigcup_{k\ge0,\,X\in\{00,02\}}\big(11^*00^*\big)^{3k+2}02X\bigg)\\
\cup
\bigg(\bigcup_{k_1+k_2\equiv1\,(\text{\em mod }3),\,X\in\{00,02\}}
\big(11^*00^*\big)^{k_1}0200^*\big(11^*00^*\big)^{k_2}X\bigg)\\
\cup
\bigg(\bigcup_{k_1+k_2\equiv2\,(\text{\em mod }3),\,X\in\{00,02\}}
\big(11^*00^*\big)^{k_1}11^*0200^*\big(11^*00^*\big)^{k_2}X\bigg)\\
\cup
\bigg(\bigcup_{k_1+k_2\equiv0\,(\text{\em mod }3),\,
k_2>0,\,X\in\{00,02\}}
\big(11^*00^*\big)^{k_1}02\big(11^*00^*\big)^{k_2}X\bigg)\\
\cup
\bigg(\bigcup_{k_1+k_2\equiv1\,(\text{\em mod }3),\,
k_2>0,\,X\in\{00,02\}}
\big(11^*00^*\big)^{k_1}11^*02\big(11^*00^*\big)^{k_2}X\bigg)\\
\cup
\bigg(\bigcup_{k_1+k_2\equiv2\,(\text{\em mod }3),\,X\in\{00,02\}}
\big(11^*00^*\big)^{k_1}0200^*\big(11^*00^*\big)^{k_2}11^*X\bigg)\\
\cup
\bigg(\bigcup_{k_1+k_2\equiv0\,(\text{\em mod }3),\,X\in\{00,02\}}
\big(11^*00^*\big)^{k_1}11^*0200^*\big(11^*00^*\big)^{k_2}11^*X\bigg)\\
\cup
\bigg(\bigcup_{k_1+k_2\equiv1\,(\text{\em mod }3),\,X\in\{00,02\}}
\big(11^*00^*\big)^{k_1}02\big(11^*00^*\big)^{k_2}11^*X\bigg)\\
\cup
\bigg(\bigcup_{k_1+k_2\equiv2\,(\text{\em mod }3),\,X\in\{00,02\}}
\big(11^*00^*\big)^{k_1}11^*02\big(11^*00^*\big)^{k_2}11^*X\bigg)\\
\cup
\Big(\big(11^*00^*\big)^{3k+2}011\Big)
\cup
\Big(\big(11^*00^*\big)^{3k}021\Big)\\
\cup
\Big(\big(11^*00^*\big)^{3k+1}11^*011\Big)
\cup
\Big(\big(11^*00^*\big)^{3k}11^*021\Big);
\end{multline*}
\item $21$, if, and only if, 
the $3$-adic expansion of $n$ is an element of
\begin{multline*}
\bigg(\bigcup_{k\ge0,\,X\in\{00,02\}}\big(11^*00^*\big)^{3k}02X\bigg)\\
\cup
\bigg(\bigcup_{k_1+k_2\equiv2\,(\text{\em mod }3),\,X\in\{00,02\}}
\big(11^*00^*\big)^{k_1}0200^*\big(11^*00^*\big)^{k_2}X\bigg)\\
\cup
\bigg(\bigcup_{k_1+k_2\equiv0\,(\text{\em mod }3),\,X\in\{00,02\}}
\big(11^*00^*\big)^{k_1}11^*0200^*\big(11^*00^*\big)^{k_2}X\bigg)\\
\cup
\bigg(\bigcup_{k_1+k_2\equiv1\,(\text{\em mod }3),\,
k_2>0,\,X\in\{00,02\}}
\big(11^*00^*\big)^{k_1}02\big(11^*00^*\big)^{k_2}X\bigg)\\
\cup
\bigg(\bigcup_{k_1+k_2\equiv2\,(\text{\em mod }3),\,
k_2>0,\,X\in\{00,02\}}
\big(11^*00^*\big)^{k_1}11^*02\big(11^*00^*\big)^{k_2}X\bigg)\\
\cup
\bigg(\bigcup_{k_1+k_2\equiv0\,(\text{\em mod }3),\,X\in\{00,02\}}
\big(11^*00^*\big)^{k_1}0200^*\big(11^*00^*\big)^{k_2}11^*X\bigg)\\
\cup
\bigg(\bigcup_{k_1+k_2\equiv1\,(\text{\em mod }3),\,X\in\{00,02\}}
\big(11^*00^*\big)^{k_1}11^*0200^*\big(11^*00^*\big)^{k_2}11^*X\bigg)\\
\cup
\bigg(\bigcup_{k_1+k_2\equiv2\,(\text{\em mod }3),\,X\in\{00,02\}}
\big(11^*00^*\big)^{k_1}02\big(11^*00^*\big)^{k_2}11^*X\bigg)\\
\cup
\bigg(\bigcup_{k_1+k_2\equiv0\,(\text{\em mod }3),\,X\in\{00,02\}}
\big(11^*00^*\big)^{k_1}11^*02\big(11^*00^*\big)^{k_2}11^*X\bigg)\\
\cup
\Big(\big(11^*00^*\big)^{3k}011\Big)
\cup
\Big(\big(11^*00^*\big)^{3k+1}021\Big)\\
\cup
\Big(\big(11^*00^*\big)^{3k+2}11^*011\Big)
\cup
\Big(\big(11^*00^*\big)^{3k+1}11^*021\Big).
\end{multline*}
\item In all other cases, this coefficient is divisible by $27$. 
\end{enumerate}
\end{proposition}

If the above proposition is specialised to modulus~9, then one
obtains the following result.

\begin{corollary} \label{cor:Psi-5}
Let $s$ denote the $3$-adic expansion of the non-negative integer $n$.
The coefficient of $z^n$ in the series $\Psi^5(z)$, when reduced
modulo~$9$, equals
\begin{enumerate} 
\item $1$, if, and only if, 
$s$ is an element of
\begin{multline*} 
\{0\}\cup
\bigcup_{k\ge0}\big(00^*11^*\big)^{3k}000
\cup
\bigcup_{k\ge0}\big(00^*11^*\big)^{3k}002\\
\cup
\bigcup_{k\ge0}\big(00^*11^*\big)^{3k+2}00
\cup
\bigcup_{k\ge0}\big(00^*11^*\big)^{3k+2}02;
\end{multline*}
\item $4$, if, and only if, 
$s$ is an element of
$$
\bigcup_{k\ge0}\big(00^*11^*\big)^{3k+1}000
\cup
\bigcup_{k\ge0}\big(00^*11^*\big)^{3k+1}002
\cup
\bigcup_{k\ge0}\big(00^*11^*\big)^{3k+3}00
\cup
\bigcup_{k\ge0}\big(00^*11^*\big)^{3k+3}02;
$$
\item $7$, if, and only if, 
$s$ is an element of
$$
\bigcup_{k\ge0}\big(00^*11^*\big)^{3k+2}000
\cup
\bigcup_{k\ge0}\big(00^*11^*\big)^{3k+2}002
\cup
\bigcup_{k\ge0}\big(00^*11^*\big)^{3k+1}00
\cup
\bigcup_{k\ge0}\big(00^*11^*\big)^{3k+1}02;
$$
\item $2$, if, and only if, 
$s$ is an element of
$$
\bigcup_{k\ge0}\big(00^*11^*\big)^{3k+1}001
\cup
\bigcup_{k\ge0}\big(00^*11^*\big)^{3k+3}01;
$$
\item $5$, if, and only if, 
$s$ is an element of
$$
\bigcup_{k\ge0}\big(00^*11^*\big)^{3k}001
\cup
\bigcup_{k\ge0}\big(00^*11^*\big)^{3k+2}01;
$$
\item $8$, if, and only if, 
$s$ is an element of
$$
\bigcup_{k\ge0}\big(00^*11^*\big)^{3k+2}001
\cup
\bigcup_{k\ge0}\big(00^*11^*\big)^{3k+1}01;
$$
\item $6$, if, and only if, 
the $3$-adic expansion of $n$ is an element of
\begin{multline*}
\bigg(\bigcup_{k\ge0}\big(11^*00^*\big)^{k}0201\bigg)
\cup
\bigg(\bigcup_{k_1,k_2\ge0}
\big(11^*00^*\big)^{k_1}0200^*\big(11^*00^*\big)^{k_2}01\bigg)\\
\cup
\bigg(\bigcup_{k_1,k_2\ge0}
\big(11^*00^*\big)^{k_1}11^*0200^*\big(11^*00^*\big)^{k_2}01\bigg)\\
\cup
\bigg(\bigcup_{k_1\ge0,\,k_2>0}
\big(11^*00^*\big)^{k_1}02\big(11^*00^*\big)^{k_2}01\bigg)\\
\cup
\bigg(\bigcup_{k_1\ge0,\,k_2>0}
\big(11^*00^*\big)^{k_1}11^*02\big(11^*00^*\big)^{k_2}01\bigg)\\
\cup
\bigg(\bigcup_{k_1,k_2\ge0}
\big(11^*00^*\big)^{k_1}0200^*\big(11^*00^*\big)^{k_2}11^*01\bigg)\\
\cup
\bigg(\bigcup_{k_1,k_2\ge0}
\big(11^*00^*\big)^{k_1}11^*0200^*\big(11^*00^*\big)^{k_2}11^*01\bigg)\\
\cup
\bigg(\bigcup_{k_1,k_2\ge0}
\big(11^*00^*\big)^{k_1}02\big(11^*00^*\big)^{k_2}11^*01\bigg)\\
\cup
\bigg(\bigcup_{k_1,k_2\ge0}
\big(11^*00^*\big)^{k_1}11^*02\big(11^*00^*\big)^{k_2}11^*01\bigg)\\
\cup
\Big(\big(11^*00^*\big)^{k}010\Big)
\cup
\Big(\big(11^*00^*\big)^{k}012\Big)
\cup
\Big(\big(11^*00^*\big)^{k}020\Big)
\cup
\Big(\big(11^*00^*\big)^{k}022\Big)\\
\cup
\Big(\big(11^*00^*\big)^{k}11^*010\Big)
\cup
\Big(\big(11^*00^*\big)^{k}11^*012\Big)\\
\cup
\Big(\big(11^*00^*\big)^{k}11^*020\Big)
\cup
\Big(\big(11^*00^*\big)^{k}11^*022\Big);
\end{multline*}
\item $3$, if, and only if, 
the $3$-adic expansion of $n$ is an element of
\begin{multline*}
\bigg(\bigcup_{k\ge0,\,X\in\{00,02\}}\big(11^*00^*\big)^{k}02X\bigg)\\
\cup
\bigg(\bigcup_{k_1,k_2\ge0,\,X\in\{00,02\}}
\big(11^*00^*\big)^{k_1}0200^*\big(11^*00^*\big)^{k_2}X\bigg)\\
\cup
\bigg(\bigcup_{k_1,k_2\ge0,\,X\in\{00,02\}}
\big(11^*00^*\big)^{k_1}11^*0200^*\big(11^*00^*\big)^{k_2}X\bigg)\\
\cup
\bigg(\bigcup_{k_1,k_2\ge0,\,k_2>0,\,X\in\{00,02\}}
\big(11^*00^*\big)^{k_1}02\big(11^*00^*\big)^{k_2}X\bigg)\\
\cup
\bigg(\bigcup_{k_1,k_2\ge0,\,k_2>0,\,X\in\{00,02\}}
\big(11^*00^*\big)^{k_1}11^*02\big(11^*00^*\big)^{k_2}X\bigg)\\
\cup
\bigg(\bigcup_{k_1,k_2\ge0,\,X\in\{00,02\}}
\big(11^*00^*\big)^{k_1}0200^*\big(11^*00^*\big)^{k_2}11^*X\bigg)\\
\cup
\bigg(\bigcup_{k_1,k_2\ge0,\,X\in\{00,02\}}
\big(11^*00^*\big)^{k_1}11^*0200^*\big(11^*00^*\big)^{k_2}11^*X\bigg)\\
\cup
\bigg(\bigcup_{k_1,k_2\ge0,\,X\in\{00,02\}}
\big(11^*00^*\big)^{k_1}02\big(11^*00^*\big)^{k_2}11^*X\bigg)\\
\cup
\bigg(\bigcup_{k_1,k_2\ge0,\,X\in\{00,02\}}
\big(11^*00^*\big)^{k_1}11^*02\big(11^*00^*\big)^{k_2}11^*X\bigg)\\
\cup
\Big(\big(11^*00^*\big)^{k}011\Big)
\cup
\Big(\big(11^*00^*\big)^{k}021\Big)
\cup
\Big(\big(11^*00^*\big)^{k}11^*011\Big)
\cup
\Big(\big(11^*00^*\big)^{k}11^*021\Big);
\end{multline*}
\item In all other cases, this coefficient is divisible by $9$. 
\end{enumerate}
\end{corollary}


\end{document}